\documentclass[11pt]{article}
\usepackage{xcolor}

\usepackage[margin=3cm]{geometry}

\usepackage[normalem]{ulem}
\usepackage{amsthm, bbm}
\usepackage{amsmath, amssymb}
\usepackage{mathrsfs}
\usepackage{hyperref}
\usepackage{graphicx, psfrag, epstopdf}
\usepackage[capitalise,noabbrev]{cleveref}
\usepackage{tikz}

\theoremstyle{definition}
\newtheorem{definition}{Definition}[section]
\newtheorem{theorem}[definition]{Theorem}
\newtheorem{lemma}[definition]{Lemma}
\newtheorem{corollary}[definition]{Corollary}
\newtheorem{proposition}[definition]{Proposition}

\newtheorem{remark}[definition]{Remark}

\newtheorem{bigthm}{Theorem} 
\numberwithin{equation}{section}

\definecolor{OliveGreen}{rgb}{0,0.6,0}

\def\one{{\mathbbm{1}}}

\def\R{\mathbb{R}}
\def\Z{\mathbb{Z}}
\def\T{\mathbb{T}}
\def\cQ{\mathcal{Q}}
\def\N{\mathbb{N}}

\def\norm{\mathcal{N}}

\def\cT{\mathcal{T}}
\def\cR{\mathcal{R}}
\def\cP{\mathcal{P}}
\def\cN{\mathcal{N}}

\def\tb{\mathbf{t}}
\def\pb{\mathbf{p}}
\def\qb{\mathbf{q}}

\newcommand*{\diff}{\mathop{}\!\mathrm{d}}

\newcommand{\flow}{\mathsf{f}}
\newcommand{\flowR}{(\mathsf{f}_t)_{t \in \R}}
\newcommand{\mapg}{\mathsf{g}}
\newcommand{\baseX}{X}
\newcommand{\susp}{X^{\roof}}
\newcommand{\roof}{\Phi}
\newcommand{\Deltat}{\Delta \mathbf{t}}
\newcommand{\lieg}{\mathfrak{g}}

\DeclareMathOperator{\height}{h}

\DeclareMathOperator{\Par}{Par}
\DeclareMathOperator{\QPar}{QPar}
\DeclareMathOperator{\Leb}{Leb}
\DeclareMathOperator{\vol}{vol}
\DeclareMathOperator{\SL}{SL}
\DeclareMathOperator{\diam}{diam}
\DeclareMathOperator{\Tow}{Tow}
\DeclareMathOperator{\dist}{d}
\DeclareMathOperator{\totvar}{TV}
\DeclareMathOperator{\Lip}{Lip}
\DeclareMathOperator{\Erg}{Erg}

\title{Multiple mixing for parabolic systems}
\author{Adam Kanigowski, Davide Ravotti}
\begin{document}

\maketitle

\begin{abstract}
The famous Rokhlin Problem asks whether mixing implies higher order mixing. So far, all the known examples of zero entropy, mixing dynamical systems enjoy a variant of the {\em mixing via shearing} mechanism. In this paper we introduce the notion of locally uniformly shearing systems (LUS) which is a rigorous way of describing the mixing via shearing mechanism. We prove that all LUS flows are mixing of all orders. 
We then show that mixing smooth flows on surfaces and smooth time-changes of unipotent flow are LUS. We also introduce the notion of quantitative LUS. We show that polynomially mixing systems that are polynomially LUS are in fact polynomially mixing of all orders. As a consequence we show that Kochergin flows on $\T^2$ (for a.e. irrational frequency) as well as smooth time-changes of unipotent flows are polynomially mixing of all orders. 
\end{abstract}


\section{Introduction}\label{sec:intro}
Mixing is one of the central notions describing randomness of a dynamical system. It claims that two events become asymptotically independent as time goes to $\infty$. More precisely, we say that an invertible dynamical system $(T,X,\mu)$ (or a flow $(T_t,X,\mu)$) is {\em mixing} if for any two measurable sets $A,B\in X$, $\mu(A\cap T^nB)\to \mu(A)\mu(B)$ as $n\to \infty$. There are now many dynamical systems for which mixing was shown to hold both in positive and zero entropy classes. For smooth systems with some form of hyperbolicity, the ultimate argument to show mixing (and in fact also stronger ergodic properties) is the {\em Hopf argument}, first used by Hopf \cite{Ho} to show ergodicity of the geodesic flow on the tangent bundle of a negatively curved surface. This argument has been also used by Sinai to show that hyperbolic systems enjoy the $K$--property. The key feature of the Hopf argument lies in the existence of an invariant foliation (the unstable foliation) which is stretched by the dynamics and is ergodic; hence, it equidistributes in space under the action of the dynamics. As mentioned above whenever the foliation is ergodic, the system is not only mixing but in fact a $K$--system (and so mixing of all orders). This in particular implies a positive answer to the famous \emph{Rokhlin problem}, \cite{Rok}, on whether mixing implies higher order mixing for hyperbolic systems. In fact, it follows from \cite{Thou} that if the Rokhlin problem has a positive answer for zero entropy systems then it has a positive answer in general. Let us recall that  an invertible dynamical system $(T,X,\mu)$ (or a flow $(T_t,X,\mu)$) is mixing of order $k$ (or $k$-mixing) if 
$$
\mu\Big(\bigcap_{i=0}^{k-1} T^{n_i}A_i\Big)\to \prod_{i=0}^{k-1}\mu(A_i),
$$
where the limit is taken over all increasing sequences $\{n_i\}$ with $n_0=0$ satisfying $\min_i|n_{i+1}-n_i|\to \infty$, with analogous definitions for flows. For zero entropy systems, the classical Hopf argument cannot be used, as the unstable foliation is trivial in this case. However, a mechanism which shares similarity with Hopf's argument has been successfully used to show mixing for many classes of zero entropy systems. In vague terms, it is still based on existence of (partial) foliations that equidistribute in space when pushed by the dynamics; however, they are {\em not }invariant under the dynamics. This mechanism, called  {\em mixing via shearing}, allows one to show mixing for zero entropy systems (we will describe it in more details later).  We should point out that there is no formal definition of the mixing via shearing argument (one such possible related definition is the LUS property given in \Cref{def:LUS_flows}).
The prime example here is the horocycle flow and more generally unipotent flows on quotients of semi-simple Lie groups. Geometrically, the renormalization by the geodesic flow implies that short geodesic arcs stretch and become long in the direction of the flow. The mixing via shearing mechanism has been used to show mixing for time-changes of horocycle flows, \cite{FU} and unipotent flows \cite{Rav}, smooth flows on tori and higher genus surfaces \cite{Koch, Kochh, KS, Ulc, Fay, Fay2, FFK, ChW}, time-changes of nilflows \cite{AFU,Ravotti2,AFRU}. A non-smooth variant of this mechanism was also used to show mixing for abstract measure preserving systems \cite{Adams,CS}. In fact, to the best of our knowledge, all known mixing systems enjoy the mixing via shearing mechanism. In vague terms, the main results of this paper are that mixing via shearing implies multiple mixing, and, as a consequence, all known mixing flows are multiple mixing (we will be more precise below). One crucial difference between the Hopf argument and mixing via shearing is that it does not give any information on higher order mixing for the system. As a matter of fact, except some special cases, there is no geometric method to show higher order mixing for zero entropy systems. These special cases are horocycle flows and unipotent flows, for which higher mixing was established in \cite{Marc}, see also \cite{BEG} for quantitative results. The main reason behind this phenomenon is the homogeneity of the space which implies that the local behavior is in fact the same throughout the space. However, for non-algebraic systems, the Rokhlin problem is still open for many classes of mixing via shearing systems. Therefore, to study the Rokhlin problem one usually uses some extra information on the ergodic or spectral features of the system. In this direction, a classical result of Host \cite{Host} states that mixing implies multiple mixing for systems with singular spectrum. This gives a method to establish multiple mixing for mixing system and it was exploited in \cite{Fay3}. Nonetheless, it is generally very hard to prove that a system has singular spectrum, and the examples in \cite{Fay3} were actually tailored to have singular spectrum while maintaining mixing.  Another method to show multiple mixing is through understanding joinings of the system. The main mechanism here is so called {\em Ratner's property} (and its variants) \cite{Thou,FL}, which implies strong restrictions on the possible joinings of the system. It was originally discovered by M.\ Ratner for horocycle flows and has been recently used to show multiple mixing for some important classes of mixing flows, \cite{FK,KKU}.  Let us point out that there are classes of mixing flows, such as those considered in \cite{FK, AFRU, Fay, ChW, KS, Ulc}, for which Ratner's property does not seem to hold and also the spectral type is mysterious. As a consequence multiple mixing for these flows is not known. The other drawback of the above methods is that they are non-quantitative by their very nature. As a consequence, beyond the homogeneous world, the problem of quantitative higher order mixing for systems that are mixing with rates is completely open. The other main result of the paper is establishing polynomial higher order mixing for natural classes of polynomially mixing flows. By our results mixing via shearing implies multiple mixing and it seems to be  the only mechanism which is known to produce mixing in zero entropy dynamical systems; it would be interesting to construct a mixing system which is not mixing via shearing. This would possibly shed a new light on the Rokhlin problem.

We will now pass to a more detailed description of our results. To keep the presentation relatively simple in this paper we will restrict our attention to smooth flows on surfaces and time-changes of unipotent flows. We deal with other mixing via shearing flows in future work.

\section{Main results}
Here we present our main results. 
\subsection{Smooth flows on surfaces}
Surface flows are one of the most natural classes of smooth zero-entropy systems for which mixing was established. In this setting we will consider so called Kochergin flows on the two-dimensional torus (with one singularity)
and Arnol'd flows on arbitrary higher genus surfaces. Surface flows are usually studied through their special representation. As a consequence of our results we answer Question 38 in \cite{FaKr} (see also \cite{FrUl}).

\paragraph{Kochergin flows on the torus.}
Let $\alpha\in \R\setminus \mathbb{Q}$  and let $\Phi\in \mathscr{C}^{2}(\T\setminus\{0\})$, $\Phi>0$ satisfy 
$$
\lim_{x\to 0^{\pm}}\frac{\Phi''(x)}{x^{-2-\gamma}}=N_{\pm},
$$
where $N_+=N_-\neq 0$. Let $(K_t)=(K_t^{\alpha,\gamma})$ be the corresponding special flow over $R_\alpha$ and under $\Phi$ (see \Cref{sec:towers} for the relevant definitions on special flows). We will call these flows Kochergin flows. It was shown in \cite{Koch} that for every irrational $\alpha$ the corresponding Kochergin flow is mixing. In fact mixing was shown in a more general context of $T$ being an arbitrary interval exchange transformation and the roof function $\Phi$ having non-trivial power singularities at the discontinuities of $T$. We should say that the methods we develop here are applicable to the more general case; here we restrict to the case of the torus to reduce the technicalities. The general case will be studied in forthcoming work. Our first main result on Kochergin flows is the following. 
\begin{bigthm}\label{thm:koch1} Every Kochergin flow $(K_t)$ on $\T^2$ with one fixed point is mixing of all orders.
\end{bigthm}
It was shown in \cite{FFK} that if $\gamma$ is close to $1$  then $(K_t)$ has (countable) Lebesgue spectrum for a full measure set of frequencies. This makes the above result more interesting, as one can not use the Host criterion to deduce multiple mixing. Our second main result is polynomial mixing of Kochergin flows for a.e. irrational frequency. 
\begin{bigthm}\label{thm:koch2} 
	Let $(q_n)$ be the sequence of denominators of $\alpha$. 
	If $\alpha$ satisfies $q_{n+1}\leq C q_n\log^{2}q_n$ then the corresponding Kochergin flow $(K_t)$ is polynomially mixing of all orders.
\end{bigthm}

We remark that the set of $\alpha$ satisfying the above condition is a full measure set. 
\paragraph{Arnol'd flows on surfaces.}
Let $T$ be an ergodic IET with discontinuities $\{\beta_1,\ldots, \beta_d\}$  and let $\Phi\in \mathscr{C}^{2}(\T\setminus\{0\})$, $\Phi>0$ satisfy 
$$
\lim_{x\to \beta_i^{\pm}}\frac{\Phi''(x)}{x^{-2}}=N_i^{\pm},
$$
where $\sum N^+_i\neq \sum N^-_i\neq 0$. These flows represent smooth flows on surfaces with asymmetric logarithmic singularities and are usually called Arnol'd flows; we will denote them by $(A_t)$. It was first shown by Khanin and Sinai \cite{KS} that if $T$ is a rotation by $\alpha$ satisfying $q_{n+1}<C q_n^{1+\delta}$, and $\Phi$  has one asymmetric logarithmic singularity at $0$ then the corresponding Arnol'd flow is mixing (the diophantine condition has full measure). This result was generalized by Ulcigrai \cite{Ulc}, who has shown that under an analogous diophantine condition on $T$ the corresponding Arnol'd flow is mixing provided that the roof function has one asymmetric logarithmic singularity at $0$. Mixing for the case of general singularities for the roof function $\Phi$ was shown by the second author in \cite{Rav}. As a consequence of our techniques we show that all flows considered in \cite{Ulc}, \cite{Rav} are mixing of all orders.

\begin{bigthm}\label{thm:arn} Every mixing Arnol'd flow $(A_t)$ considered in \cite{Ulc}, \cite{Rav} is mixing of all orders.
\end{bigthm}

\paragraph{Mixing flows with symmetric singularities.}
Chaika and Wright in \cite{ChW} have shown existence of a class of minimal smooth flows on surfaces of higher genus which have symmetric logarithmic singularities and are mixing. Recall that in such case mixing typically does not hold as shown by Ulcigrai \cite{Ulc}. As a consequence of our methods we also show:
\begin{bigthm}\label{prop:ChW} Every mixing flow considered in \cite{ChW} is mixing of all orders.
\end{bigthm}

\subsection{Time-changes of unipotent flows}

Unipotent flows on quotients of semisimple Lie groups are the fundamental examples of \lq\lq uniform parabolicity\rq\rq\ in homogeneous dynamics.
Of the possible smooth perturbations which preserve the parabolic behaviour, time-changes form the simplest and most studied class. A time-change, or time-reparametrization, of a unipotent flow is obtained by rescaling the generating vector field by a (nonconstant) positive function. Despite their apparent simplicity, time-changes of unipotent flows display interesting rigidity phenomena and their ergodic properties (especially from the quantitative point of view) are far from being completely understood. Our main result in this direction shows that, under some natural conditions, they are polynomially mixing of all orders.

\begin{bigthm}\label{th:main_tc_unip}
	Let $G$ be a connected semisimple Lie group with finite centre and no compact factor,
	and let $M=\Gamma \backslash G$, where $\Gamma$ is an irreducible lattice. Any admissible time-change of a unipotent flow on $M$ is polynomially mixing of all orders.
\end{bigthm}

We refer to \Cref{sec:tc_unip_flows} for all definitions, including the notion of admissibility (see \Cref{def:admissible_tc}). We remark that, when $M$ is compact, all (normalized) smooth time changes are admissible.

Polynomial mixing of all orders for homogeneous unipotent flows in the setting of \Cref{th:main_tc_unip} was proved by Bj\"{o}rklund, Einsiedler, and Gorodnik in \cite{BEG}.
For $G = \SL_2(\R)$ and $\Gamma$ co-compact, polynomial $2$-mixing of smooth time-changes oh horocycle flows was proved in \cite{FU} and polynomial $3$-mixing in \cite{KR}. In the general setting of \Cref{th:main_tc_unip}, polynomial $2$-mixing of admissible time-changes of unipotent flows was proved in \cite{Rav}.

\subsection{Other mixing flows}
It is possible to apply our methods to other mixing flows. In particular we plan to use our approach to show mixing of all orders for mixing time changes of nilflows considered in  \cite{AFU,Ravotti,AFRU}. We also plan to use this approach to prove multiple mixing and quantitative multiple mixing for time-changes of linear flows on tori considered in \cite{Fay}. To keep the presentation of this paper more readable, we decided to study the above mentioned examples in a separate paper (but the approach will be based on methods from the present work). 

\subsection{Organization of the paper}

The remainder of the paper is organised as follows.

In \Cref{sec:GUSflows}, we introduce the notion of globally uniformly shearing (GUS) flows as an abstract notion that describes the shearing properties of many \lq\lq uniformly parabolic\rq\rq\ flows. We show that mixing GUS flows are mixing of all orders. We also provide a quantitative version of the GUS property, as well as a result on polynomial mixing of all orders. This section could serve as a \lq\lq warm up\rq\rq\ for the generalizations that will be the focus of \Cref{sec:LUSflows} and \Cref{sec:quantitative_results}.

In \Cref{sec:tc_unip_flows}, we prove \Cref{th:main_tc_unip} by showing that time-changes of unipotent flows are quantitatively GUS and satisfy the assumptions of the abstract results of \Cref{sec:GUSflows}.

\Cref{sec:towers} is an \emph{intermezzo} in which we discuss some properties of towers, which will be useful in the following sections. We also recall some basic definitions on special flows and we introduce the notion of almost partitions. 

\Cref{sec:LUSflows} and \Cref{sec:quantitative_results} contain the main abstract results of the paper: we introduce the notion of locally uniformly shearing (LUS) flows and its quantiative counterpart (QLUS flows), which aim at describing the shearing properties of many \lq\lq non-uniformly parabolic\rq\rq\ flows.
We show that mixing (respectively, polynomially mixing) LUS (respectively, QLUS) flows are mixing of all orders (respectively, polynomially mixing of all orders). To this end, we introduce the properties $\Par(k)$ and $\QPar(k)$ which, roughly speaking, imply a local decay of $k$-correlations on towers.
The core of the arguments is to show that the (Q)LUS property allows us to upgrade property (Q)Par$(k)$ to (Q)Par$(k+1)$.

\Cref{sec:Goodflows} focuses on special flows. We formulate sufficient conditions to ensure that a special flow is LUS or QLUS and we call the special flows which verify such conditions (quantitatively) good. 

In \Cref{sec:LUS_on_surfaces} and \Cref{sec:QLUS_on_surfaces}, we further give conditions to verify whether a special flow is (quantitatively) good. These conditions rely on the existence of partitions of the base of the suspension space on which more standard shearing inequalities hold (namely, lower bounds on the shearing and upper bounds on the distortion).

The subject of the last section, \Cref{sec:KFlows}, are Kochergin flows on the torus. We prove \Cref{thm:koch1} and \Cref{thm:koch2} by constructing the partitions which allow us to apply the results of the previous two sections.

Finally, the Appendices \ref{sec:appendix_a}, \ref{sec:vdc} and \ref{sec:shearing} contain the proofs of several technical lemmas.


\section{Globally uniformly shearing flows}\label{sec:GUSflows}

In this section, we prove multiple mixing for a class of flows which we call \emph{globally uniformly shearing}, which correspond to the \lq\lq uniformly parabolic\rq\rq\ behaviour. The definitions and arguments presented below will be later generalized in \Cref{sec:LUSflows}, to cover the cases where weaker assumptions hold (in \lq\lq non-uniformly parabolic\rq\rq\ setting). 

\subsection{Motivation}
The notion of global uniform shearing is an abstract formulation of the shearing mechanism that parabolic flows display. In order to motivate \Cref{def:GUS_flow}, let us recall the simplest case of this behaviour: horocycle flows on unit tangent bundles of hyperbolic surfaces with constant negative curvature. Let $X$ be the unit tangent bundle of a compact hyperbolic surface, and let $(h_t)_{t\in \R}$ and $(g_t)_{t\in R}$ denote the horocycle and geodesic flow on $X$. The well-known commutation relation $h_t \circ g_s = g_s \circ h_{e^s t}$ implies that, when $t>1$ is large and, say, $|s|\leq \sigma_t := t^{-3/4}$, for any point $x \in X$, we can bound the distance
\[
\dist \big(h_t \circ g_s (x), h_{t+st}(x) \big) \leq \sigma_t + \sigma_t^2 t \leq 2 t^{-1/2}.
\] 
Geometrically, this means that the pushed geodesic segment $s \mapsto h_t \circ g_s (x)$, for $s\in [0,\sigma_t]$, becomes sheared along the flow direction and approximates the horocycle orbit segment of length $\sigma_t \, t = t^{1/4}$ starting at $h_t(x)$ with linear speed $t$, see \Cref{fig:shear} for a schematic picture. Since the flow is uniquely ergodic, the pushed segments equidistributes in $X$. A formalization of this heuristic argument takes the name of the \lq\lq mixing via shearing\rq\rq\ method.

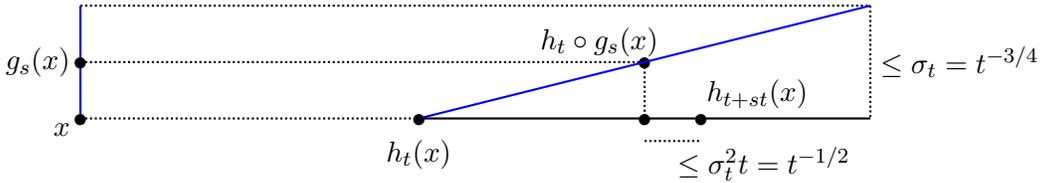
\begin{figure}[h]
	\centering
	\begin{tikzpicture}[scale=3]
		\draw[thick, blue] (0,0) -- (0,0.5);
		\foreach \Point in {(0,-0.01), (0,0.24), (1.5,-0.01), (2.5,0.24), (2.5,-0.01), (2.75,-0.01)}{
			\node at \Point {\textbullet};
		}
		\draw[thick, densely dotted] (0,0) -- (1.5,0);
		\draw[thick, densely dotted] (0,0.25) -- (2.5,0.25);
		\draw[thick, densely dotted] (0,0.5) -- (3.5,0.5);
		\draw[thick, densely dotted] (2.5,0.25) -- (2.5,0);
		\draw[thick, densely dotted] (2.5,-0.1) -- (2.75,-0.1);
		\draw[thick, densely dotted] (3.5,0.5) -- (3.5,0);
		\draw[thick, blue] (1.5,0) -- (3.5,0.5);
		\draw[thick] (1.5,0) -- (3.5,0);
		
		\draw[white] (0,-0.05)  node[anchor=east] {\textcolor{black}{$x$}};
		\draw[white] (0,0.25)  node[anchor=east] {\textcolor{black}{$g_s(x)$}};
		\draw[white] (1.5,-0.05)  node[anchor=north] {\textcolor{black}{$h_t(x)$}};
		\draw[white] (3,0)  node[anchor=south] {\textcolor{black}{$h_{t+st}(x)$}};
		\draw[white] (3.5,0.25)  node[anchor=west] {\textcolor{black}{$\leq \sigma_t = t^{-3/4}$}};
		\draw[white] (2.3,0.23)  node[anchor=south] {\textcolor{black}{$h_t \circ g_s(x)$}};
		\draw[white] (2.6,-0.2)  node[anchor=west] {\textcolor{black}{$\leq \sigma_t^2t = t^{-1/2}$}};

		\clip (-0.2,-0.3) rectangle (4.5,0.7);
	\end{tikzpicture}
	\caption{Shearing phenomenon: the pushed geodesic segment approximate an orbit segment parametrized with linear speed.}
	\label{fig:shear}
\end{figure}

A similar situation holds for time-changes of horocycle flows as well, although in this case the maps $g_s$ are not measure preserving. Since $s \leq \sigma_t$ is small, its effect on the measure is small as well. In general, we formulate the following definition. 
 
\begin{definition}\label{def:almostmp}
	A measurable map $\mapg \colon (X, \mu) \to (Y,\nu)$ between two measure spaces $ (X, \mu)$ and $ (Y,\nu)$ is called \emph{$\varepsilon$-almost measure preserving} if there exists $X_0 \subseteq X$ and $ Y_0 \subseteq Y$ with $\mu(X\setminus X_0) < \varepsilon \mu(X)$ and $\nu(Y \setminus Y_0) < \varepsilon \nu(Y)$ such that $(1-\varepsilon) \nu(A) \leq \mu(\mapg^{-1}(A)\cap X_0) \leq (1+\varepsilon) \nu(A)$ for all measurable sets $A \subseteq Y_0$. We call $X_0$ and $Y_0$ the \emph{proper domain} and the  \emph{proper codomain} of $\mapg$ respectively.
\end{definition}

The key basic property of almost measure preserving maps is the following.

\begin{lemma}\label{lem:change_variab}
	Let $\mapg \colon (X, \mu) \to (X, \mu)$ be $\varepsilon$-almost measure preserving. Then, for every $ \phi \in L^{\infty}(X)$ we have
	\[
	\left\lvert \int \phi \, \diff \mu - \int \phi \circ \mapg \, \diff \mu  \right\rvert \leq 3 \|\phi \|_{\infty} \, \varepsilon. 
	\]
\end{lemma}
\begin{proof}
	Fix $ \phi \in L^{\infty}(X)$, and let $D$ and $E$ denote the proper domain and codomain of $\mapg$ respectively. Then,
	\[
	\begin{split}
		\left\lvert \int \phi \, \diff \mu - \int \phi \circ \mapg \, \diff \mu  \right\rvert &\leq \| \phi\|_{\infty} \, \big( \mu(X \setminus E) + \mu(X \setminus D) \big)+ \left\lvert \int_E \phi \, \diff \mu - \int_D \phi \circ \mapg \, \diff \mu  \right\rvert \\
		& \leq 2\| \phi\|_{\infty} \, \varepsilon + \| \phi\|_{\infty} \, \| \mu|_E - \mapg_{\ast}( \mu|_D )\|_{\totvar},
	\end{split}
	\]
	where $\| \cdot \|_{\totvar}$ is the total variation norm. By the definition of $\varepsilon$-almost measure preserving maps, the latter is not larger than $\varepsilon$.
\end{proof}

\subsection{GUS mixing flows are mixing of all orders}

Let us now describe the general setting we work with.
Let $(X, \dist)$ be a complete locally compact, $\sigma$-compact, metric space, and let $\mu$ be a Radon probability measure on $X$.
We denote by $\Lip(X)$ the space of Lipschitz functions on $X$, and by $\Lip_c(X)$ the subspace of those with compact support. The norm $\|\cdot\|_{\Lip}$ is defined by
\[
\|\phi \|_{\Lip} = \|\phi\|_{\infty} + \sup_{x\neq y} \frac{|\phi(x) - \phi(y)|}{\dist(x,y)}.
\]
Let $\flowR$ be a measure preserving flow on $(X, \mathcal{B}, \mu)$. 
We fix some notation: henceforth, by a \emph{$k$-tuple of times} we mean a vector $\tb = (t_0, \dots, t_{k-1}) \in \R^k$ so that $0 = t_0 < t_1 < \cdots < t_{k-1}$. For any $k$-tuple of times $\tb$, we let $\Deltat = \min_{ 0 < i < k} t_i - t_{i-1}$, with the convention that $\Delta (0) = \infty$.

\begin{definition}[Globally uniformly shearing flows]\label{def:GUS_flow}
	The flow $\flowR$ is called {\em globally uniformly shearing} (GUS for short) if there exists $a \in (0,1/2]$ and, for every $\delta >0$ and $k \in \Z_{\geq 0}$, there exist a measurable set $G_{\delta} \subseteq X$, with $\mu(X\setminus G_{\delta}) \leq \delta$, and $M_{k,\delta} > \delta^{-2}$ such that for every $M \geq M_{k,\delta}$ there exist $t_{M}\geq M$ for which the following holds.
For every 
$(k+1)$-tuple of times $\tb$ with $\Deltat > t_{M}$, there exist $\delta$-almost measure preserving maps $\mapg^{\tb}_m \colon X \to X$, for $m=0,\dots, K_{\delta}= \lfloor \delta^{-1} \rfloor$, satisfying the following property
\begin{itemize}
\item[(GUS1)] for every $j=0,\dots,k$, we have
\[
\sup_{x \in G_{\delta}} \dist \left(\flow_{t_{j}}(\mapg^{\tb}_m x), \flow_{t_j +m M^a(t_j/ t_k) }(x) \right)\leq \delta.
\]
\end{itemize}
\end{definition}

Condition (GUS1) quantifies the uniform shearing effect produced by the flow $\flowR$ on the points $x$ and $\mapg^{\tb}_m(x)$. The reader can verify that the horocycle flow is GUS for any $0<a\leq 1/2$ by taking $\mapg^{\tb}_m = g_{mM^a/t_k}$; the more general situation of time changes of unipotent flows will be the subject of \Cref{sec:tc_unip_flows}.

We now prove that GUS mixing flows are mixing of all orders. 
We will use the following lemma.
\begin{lemma}[Uniform $k$-mixing]\label{lem:uniform_k_mixing}
	Assume that the flow $\flowR$ is $k$-mixing. 
	Let $\phi_0,\ldots, \phi_{k-1}\in \Lip(X)$, with $\|\phi_i\|_{\Lip} < \infty$, and assume that $\mu(\phi_j)=0$ for at least one $j \in \{0,\dots, k-1\}$. For every $\varepsilon>0$, there exists $R_{-}\geq 1$ and for every $R_{+} \geq R_{-}$ there exists $t_{\varepsilon, R_{+}} \geq 1$ so that for all $k$-tuples of times $\tb$ with $\Deltat \geq t_{\varepsilon, R_{+}}$, we have
	\[
	\sup_{\substack{0\leq |r_i|\leq R_{+} \\ R_{-} \leq |r_j| \leq R_{+}}}\left\lvert \int_{X} \prod_{i = 0}^{ k-1} (\phi_i \cdot \phi_i \circ \flow_{r_i}) \circ \flow_{t_i} \diff\mu\right\rvert \leq \varepsilon.
	\]
\end{lemma}
\begin{proof}
	For the sake of notation, let us fix $j=k-1$; the proof is identical for other choices of $j$. 
	Let $\varepsilon >0$ be fixed, and let $R_{-} \geq 1$ be so that $|\mu(\phi_{k-1} \cdot \phi_{k-1} \circ \flow_r)|\leq \varepsilon/(4\prod_{i=0}^{k-2} \|\phi_i \|_{\infty})$ for all $|r|\geq R_{-}$. Fix now $R_{+} \geq R_{-}$, and let $N \in \N$ be minimal so that 
	\[
	N \geq 4 \varepsilon^{-1} \, k \left( \prod_{i=0}^{k-1} \|\phi_i\|^2_{\Lip}\right).
	\]
	We use the $k$-mixing assumption for the finite collection of all possible $k$-tuples of functions 
	\[
	\{\phi_0 \cdot \phi_0 \circ \flow_{r_0}, \dots, \phi_{k-1} \cdot \phi_{k-1} \circ \flow_{r_{k-1}}\},
	\] 
	for all choices of $r_i \in N^{-1}\Z$, with $|r_i| \leq R_{+}$: there exist $t_{\varepsilon, R_{+}} \geq 1$ (given by the maximum of the respective quantity over all $k$-tuples of functions as above) so that for all $k$-tuples of times $\tb$ with $\Deltat \geq t_{\varepsilon, R_{+}}$, and for all $|r_i|\leq R_{+}$, $r_i \in N^{-1}\Z$, with $|r_{k-1} | \geq R_{-}$, we have
	\begin{equation}\label{eq:unif_mixing1}
		\begin{split}
			\left\lvert \int_{X} \prod_{i=0}^{k-1} (\phi_i \cdot \phi_i \circ \flow_{r_i}) \circ \flow_{t_i} \diff\mu\right\rvert &\leq \frac{\varepsilon}{4} + \left(\prod_{i=0}^{k-2} \|\phi_i \|_{\infty} \right)|\mu(\phi_{k-1} \cdot \phi_{k-1} \circ \flow_{r_{k-1}})| \leq \frac{\varepsilon}{2}.
		\end{split}
	\end{equation}
	Let now $|r_i| \leq R_{+}$ be arbitrary, with $|r_{k-1}| \geq R_{-}$,  and let $m_i/N\in N^{-1}\Z$ so that $|r_i-m_i/N|\leq N^{-1}$. By the Lipschitz property, we can bound
	\[
	\begin{split}
		&\| \phi_i \cdot \phi_i \circ \flow_{r_i} - \phi_i \cdot \phi_i \circ \flow_{m_i/N}\|_{\infty} \leq \|\phi_i\|^2_{\Lip} N^{-1}.
	\end{split}
	\]
	This implies that, up to an error bounded by
	\[
	\sum_{i=0}^{k-1} \left(\prod_{j\neq i} \|\phi_j\|^2_{\infty} \right) \|\phi_i\|^2_{\Lip} N^{-1} \leq \frac{\varepsilon}{2},
	\]
	the bound \eqref{eq:unif_mixing1} holds for arbitrary $|r_i| \leq R_{+}$, with $|r_{k-1}| \geq R_{-}$.
\end{proof}

\begin{theorem}\label{thm:multiple_mixing_GG_flows}
Let $\flowR$ be a mixing GUS flow. Then $\flowR$ is mixing of all orders.
\end{theorem}
\begin{proof}
The proof proceeds by induction: let $k\geq 2$, and assume that $\flowR$ is $k$-mixing. It is enough to prove mixing of order $k+1$ for Lipschitz functions with compact support, so let $\phi_0,\ldots, \phi_{k}\in \Lip_c(X)$ be fixed, and let $\varepsilon \in (0,1/2)$. 
For the sake of notation, in the following when writing $A \ll B$, we mean $A \leq C \cdot B$ for a constant $C\geq 0$ whose value depends only on $k$ and on the Lipschitz norm of the functions $\phi_0, \dots, \phi_k$.
We need to show that there exists $T(k+1,\varepsilon)\geq 1$ so that for all $(k+1)$-tuple of times $\tb$ with $\Deltat \geq T(k+1,\varepsilon)$ we have 
\begin{equation*}
	\left\lvert \int_{X}\prod_{i=0}^{k} \phi_i\circ \flow_{t_i} \diff\mu - \prod_{i=0}^{k}  \int_{X} \phi_i \diff\mu  \right\rvert \ll \varepsilon.
\end{equation*}
Writing $\phi_k = \phi_k^{\perp} + \mu(\phi_{k})$, with $\phi_k^{\perp} = \phi - \mu(\phi_{k}) \in \Lip(X)$ and $\|\phi_k^{\perp}\|\ll 1$, and using the $k$-mixing assumption, we can assume that $\mu(\phi_k) = 0$. Hence, it suffices to show that 
\begin{equation}\label{eq:proof_thmGG_goal}
	\left\lvert \int_{X}\prod_{i=0}^{k} \phi_i\circ \flow_{t_i}\diff\mu  \right\rvert \ll \varepsilon.
\end{equation}
Let $a\in (0,1/2]$ and fix $\delta = \varepsilon^4$. Let $M_{k, \delta} \geq \delta^{-2}$ be given by the GUS property in \Cref{def:GUS_flow}.
\sloppy We apply the uniform $k$-mixing property of \Cref{lem:uniform_k_mixing} to the $k$-tuple of functions $\phi_1, \dots, \phi_k$ (notice the shift in the indexing) with parameter $\varepsilon^2$, and we choose $R_{+}= M = \max \{M_{k, \delta}, R_{-}^{1/a}\}$: there exists $t_{\varepsilon^2, M}$ so that 
\begin{equation}\label{eq:proof_thmGG2}
	\sup_{\substack{0\leq |r_i|\leq M \\ M^a \leq r_k \leq M}}\left\lvert \int_{X} \prod_{i =1 }^k (\phi_i \cdot \phi_i \circ \flow_{r_i}) \circ \flow_{t_{i-1}} \diff\mu\right\rvert \leq \varepsilon^2,
\end{equation}
for all $k$-tuples of times $\tb $ with  $\Deltat \geq t_{\varepsilon^2,M}$.
We define $T(k+1, {\varepsilon})$ as follows: 
we consider $t_{M}$ be given by \Cref{def:GUS_flow}, and we set $T(k+1, {\varepsilon}) = \max \{ t_{M}, t_{\varepsilon^2,M} + 2 M\}$.

Let now $\tb$ be a $(k+1)$-tuple of times with $\Deltat \geq T(k+1,\varepsilon)$. 
From the GUS property, for all $m \leq K_{\delta} = \lfloor \delta^{-1}\rfloor$, we have a $\delta$-almost measure preserving map $\mapg^{\tb}_m \colon X \to X$.
We set $K = \lfloor K_{\delta} /2 \rfloor$.
By \Cref{lem:change_variab}, we have
\begin{equation}\label{eq:proof_thmGG4}
\left\lvert \int_{X}\prod_{i=0}^{k} \phi_i\circ \flow_{t_i} \diff\mu\right\rvert \ll \left\lvert \frac{1}{K} \sum_{m=1}^K \int_{X} \prod_{i=0}^{k} \phi_i\circ \flow_{t_i} \circ \mapg^{\tb}_m \diff\mu\right\rvert + \delta.
\end{equation}
Note that, by definition for $t_0=0$, we have $\dist(x, \mapg^{\tb}_m (x)) \leq \delta$.
Thus, since $\phi_0$ is Lipschitz, we obtain
\begin{equation*}
	\left\lvert \frac{1}{K} \sum_{m=1}^K \int_{X} \prod_{i=0}^{k} \phi_i\circ \flow_{t_i} \circ \mapg^{\tb}_m \diff\mu\right\rvert
	\ll \left\lvert  \int_{X} \phi_0 \cdot \left(\frac{1}{K} \sum_{m=1}^K \prod_{i=1}^{k} \phi_i\circ \flow_{t_i} \circ \mapg^{\tb}_m \right) \diff\mu\right\rvert + \delta.
\end{equation*}
We now apply the Cauchy-Schwarz Inequality and the Van der Corput Inequality (\Cref{lem:vdc} in \Cref{sec:vdc}) with $L = \lfloor \sqrt{K} \rfloor$, which together yield
\begin{multline*}
	\left\lvert \frac{1}{K} \sum_{m=1}^K \int_{X} \prod_{i=0}^{k} \phi_i\circ \flow_{t_i} \circ \mapg^{\tb}_m \diff\mu\right\rvert \\
	\ll \left[ \frac{1}{K} \sum_{m=1}^K \left( \frac{1}{L} \sum_{l=0}^L \left\lvert \int_{X}  \prod_{i=1}^{k} \phi_i\circ \flow_{t_i} \circ \mapg^{\tb}_m \cdot \phi_i\circ \flow_{t_i} \circ \mapg^{\tb}_{m+l} \diff\mu \right\rvert \right) \right]^{1/2} +\left(\frac{L}{K}\right)^{1/2} + \delta.
\end{multline*}
Isolating the term corresponding to $l=0$, since $L^{-1/2} \ll K^{-1/4} \ll \varepsilon$ and $(L/K)^{1/2} \ll K^{-1/4} \ll \varepsilon$, we obtain
\begin{equation}\label{eq:proof_thmGG3}
	\left\lvert \frac{1}{K} \sum_{m=1}^K \int_{X} \prod_{i=0}^{k} \phi_i\circ \flow_{t_i} \circ \mapg^{\tb}_m \diff\mu\right\rvert 
	\ll \sup_{\substack{0 \leq m,l \leq K \\ |m-l| \geq 1} }\left\lvert \int_{X} \prod_{i=1}^{k} \phi_i\circ \flow_{t_i} \circ \mapg^{\tb}_m \cdot \phi_i\circ \flow_{t_i} \circ \mapg^{\tb}_{l} \diff\mu \right\rvert^{1/2} + \varepsilon.
\end{equation}

Let us fix $0 \leq m,l \leq K$ with $ |m-l| \geq 1$, and let us consider the absolute value on the right hand side above. 
By (GUS1), for all $x \in G = G_{\delta}$, we have
\[
|\phi_i\circ \flow_{t_i} \circ \mapg^{\tb}_m (x)\cdot \phi_i\circ \flow_{t_i} \circ \mapg^{\tb}_{l}(x) - \phi_i\circ \flow_{t_i +m M^a (t_i/ t_k) } (x) \cdot \phi_i\circ \flow_{t_i +l M^a(t_i/ t_k) } (x)| \ll \delta= \varepsilon^4.
\]
Since $\mu(X \setminus G) \leq \delta$, We deduce
\begin{multline*}
	\left\lvert \int_{X} \prod_{i=1}^{k} \phi_i\circ \flow_{t_i} \circ \mapg^{\tb}_m \cdot \phi_i\circ \flow_{t_i} \circ \mapg^{\tb}_{l} \diff\mu \right\rvert \\
	\ll \left\lvert \int_{X} \prod_{i =1}^{k} (\phi_i \cdot \phi_i\circ \flow_{(l-m) M^a(t_i/ t_k)} )\circ \flow_{t_i + m M^a(t_i/ t_k)} \diff\mu \right\rvert + \varepsilon^4.
\end{multline*}
By invariance of $\mu$, the integrand function above can be composed by $\flow_{-\overline t}$, where ${\overline t} = t_1 + m M^a(t_1/ t_k)$; then, we obtain
\begin{equation*}
	\left\lvert \int_{X} \prod_{i=1}^{k} \phi_i\circ \flow_{t_i} \circ \mapg^{\tb}_m \cdot \phi_i\circ \flow_{t_i} \circ \mapg^{\tb}_{l} \diff\mu \right\rvert 
	\ll \left\lvert \int_{X}  \prod_{i =1}^{k} (\phi_i \cdot \phi_i\circ \flow_{(l-m) M^a(t_i/ t_k)} )\circ \flow_{t_i'} \diff\mu \right\rvert + \varepsilon^4,
\end{equation*}
where the new $k$-tuple of times $\tb'$ has components $t_i' = t_i + m M^a(t_i/ t_k)-\overline t$.
Now, we claim that the right hand side above is bounded by $\ll \varepsilon^2$. Indeed, we can use \eqref{eq:proof_thmGG2}: we have 
\[
\begin{split}
	&|(l-m) M^a(t_i/ t_k)| \leq K \sqrt{M} \leq \delta^{-1} \sqrt{M} \leq M, \qquad |l-m| M^a \geq M^a\\
	&\text{and} \qquad  |t_i' - t_j'| \geq |t_i -t_j| - 2K \sqrt{M} \geq \Deltat- 2M \geq t_{\varepsilon^2,M}.
\end{split}
\]
Therefore, the uniform $k$-mixing property in \eqref{eq:proof_thmGG2} yields
\[
\left\lvert \int_{X}  \prod_{i=1}^{k} \phi_i\circ \flow_{t_i} \circ \mapg^{\tb}_m \cdot \phi_i\circ \flow_{t_i} \circ \mapg^{\tb}_{l} \diff\mu \right\rvert \ll\varepsilon^2.
\]
Combining this bound with \eqref{eq:proof_thmGG3} and \eqref{eq:proof_thmGG4} proves \eqref{eq:proof_thmGG_goal}, and hence \Cref{thm:multiple_mixing_GG_flows}.
\end{proof}

\subsection{Quantitative versions}\label{sec:QGUS}

We include a quantitative version of the notion of globally uniformly shearing flows, which allows us to prove quantitative (higher) mixing results. 
In the remainder of this section, in addition to the previous assumptions, we also require that $X$ is a smooth manifold (possibly with boundary). Here and henceforth, we denote by $\mathscr{C}^{\infty}_c(X)$ the space of infinitely differentiable functions with compact support. 

We assume the existence of appropriate norms as follows: there exist a constant $D \geq 0$ and, for every $d \geq 0$, constants $C_d, \sigma(d)\geq 1$, and norms $\cN_d$ on the space $\mathscr{C}^{\infty}_c(X)$ so that, for all $\phi, \psi  \in \mathscr{C}^{\infty}_c(X)$, we have
\begin{enumerate}
	\item $\| \phi \|_{\mathscr{C}^n} \leq C_{n+d} \cN_{n+d}(\phi)$ for every $n \geq 0$ and $d\geq D$,
	\item $\cN_{d}(\phi \cdot \psi) \leq C_d \cN_{d+D}(\phi) \cdot \cN_{d+D}(\psi) $ for every $d \geq 0$,
	\item $\cN_{d}(\phi \circ \flow_t) \leq C_d \cN_{d}(\phi) \cdot |t|^{\sigma(d)}$ for every $d \geq 0$ and $t \in \R$.
\end{enumerate} 
If the space $X$ is compact, then one can usually take $\cN_{d} = \| \cdot \|_{\mathscr{C}^d}$ and $D=0$; in general, however, one needs to consider appropriate Sobolev norms.

We want to prove quantitative mixing estimates in the following sense. As before, for a $k$-tuple of times $\tb = (t_0, \dots, t_{k-1})$ with $0 = t_0 <t_1 < \cdots <t_{k-1}$, we set $\Deltat = \min_{ 0< i < k} t_i-t_{i-1}$. 

\begin{definition}\label{def:quantitative_mixing}
	Let $k\geq 2$ and let $\eta_k>0$.
	We say that the flow $\flowR$ is polynomially $k$-mixing with rate $\eta_k$ if there exists $d_k\geq 1$ so that, for every $\phi_0, \dots, \phi_{k-1} \in \mathscr{C}^{\infty}_c(X)$ and every $k$-tuple of times $\tb$, we have
	\[
	\left\lvert \int_{X} \prod_{i=0}^{k-1}\phi_i \circ \flow_{t_i} \diff \mu - \prod_{i=0}^{k-1} \int_{X} \phi_i \diff \mu \right\rvert \leq \left(\prod_{i=0}^{k-1} \cN_{d_k}(\phi_i)\right)   \cdot \Deltat^{-\eta_k}.
	\]
\end{definition} 
We will assume that the norms $\cN_{d_k}$ are so that $\|\cdot \|_{\mathscr{C}^1} \ll \cN_{d_2} \leq \cN_{d_k} \leq \cN_{d_{k+1}}$, and we define $\sigma_k := k \cdot \sigma(d_k+D)$.
The following lemma is an easy consequence of the properties of the Sobolev norms.

\begin{lemma}[Uniform polynomial $k$-mixing]\label{lem:uniform_polynomial_k_mixing}
	Assume that the flow $\flowR$ is polynomially $k$-mixing with rate $\eta_k$. 
	Let $\phi_0,\ldots, \phi_{k-1}\in \mathscr{C}^{\infty}_c(X)$, and assume that $\mu(\phi_j)=0$ for at least one $j \in \{0,\dots, k-1\}$. For every $1\leq R_{-}\leq R_{+}$ and for every $k$-tuple of times $\tb$ with $\Deltat\geq 1$, we have
	\[
	\sup_{\substack{0\leq |r_i|\leq R_{+} \\ R_{-} \leq |r_j| \leq R_{+}}}\left\lvert \int_{X} \prod_{i = 0}^{ k-1} (\phi_i \cdot \phi_i \circ \flow_{r_i}) \circ \flow_{t_i} \diff\mu\right\rvert \leq \left(\prod_{i=0}^{k-1} \cN_{d_k+D}(\phi_i)^2\right) \left(R_{-}^{-\eta_2} + R_{+}^{\sigma_k}\Deltat^{-\frac{\eta_k\eta_2}{ 16\sigma_k}}\right).
	\]
\end{lemma}
\begin{proof}
	Let us assume $j=k-1$, the proof being identical in the other cases. Let $\tb$ be a $k$-tuple of times and let $|r_i|\leq R_{+}$, with $|r_{k-1}|\geq R_{-}$ be fixed. We fix $\varepsilon = \Deltat^{-\eta_2\eta_k/(16\sigma_k)}$ and we apply \Cref{lem:phi_phi_orthogonal} to $\phi_{k-1} \cdot \phi_{k-1} \circ \flow_{r_{k-1}}$; using the Cauchy–Schwarz Inequality, we obtain
	\begin{multline*}
		\left\lvert \int_{X} \prod_{i = 0}^{ k-1} (\phi_i \cdot \phi_i \circ \flow_{r_i}) \circ \flow_{t_i} \diff\mu\right\rvert \leq \left\lvert \int_{X} \left(\prod_{i = 0}^{ k-2} (\phi_i \cdot \phi_i \circ \flow_{r_i}) \circ \flow_{t_i}\right) \cdot (\phi_{k-1} \cdot \phi_{k-1} \circ \flow_{r_{k-1}})^{\perp} \circ \flow_{r_{k-1}} \diff\mu\right\rvert \\
		+ \left(\prod_{i=0}^{k-2} \|\phi_i\|^2_{\infty}\right) \left(|\mu(\phi_{k-1} \cdot \phi_{k-1} \circ \flow_{r_{k-1}})| + \cN_{d_2}(\phi_{k-1} \cdot \phi_{k-1} \circ \flow_{r_{k-1}}) \varepsilon\right). 
	\end{multline*}
By polynomial $2$-mixing and the properties of the Sobolev norms, the second summand above can be bounded by 
	\begin{multline*}
	 \left(\prod_{i=0}^{k-2} \|\phi_i\|^2_{\infty}\right) \left(|\mu(\phi_{k-1} \cdot \phi_{k-1} \circ \flow_{r_{k-1}})| + \cN_{d_2}(\phi_{k-1} \cdot \phi_{k-1} \circ \flow_{r_{k-1}}) \varepsilon\right) \\
	\ll  \left(\prod_{i=0}^{k-1} \cN_{d_2+D}(\phi_i)^2\right) (R_{-}^{-\eta_2} + R_{+}^{\sigma(d_2+D)}\varepsilon),
\end{multline*}
whereas, by $k$-mixing and by \Cref{lem:phi_phi_orthogonal}, the first is bounded by 
	\begin{multline*}
	\left\lvert \int_{X} \left(\prod_{i = 0}^{ k-2} (\phi_i \cdot \phi_i \circ \flow_{r_i}) \circ \flow_{t_i}\right) \cdot (\phi_{k-1} \cdot \phi_{k-1} \circ \flow_{r_{k-1}})^{\perp} \circ \flow_{r_{k-1}} \diff\mu\right\rvert \\
	\leq \left(\prod_{i=0}^{k-1}  \cN_{d_{k}}(\phi_i \cdot \phi_i \circ \flow_{r_i}) \right) \cdot \varepsilon^{-8\sigma(d_k)/\eta_2} \Deltat^{-\eta_k} \ll  \left(\prod_{i=0}^{k-1} \cN_{d_k+D}(\phi_i)^2\right) R_{+}^{ k \sigma(d_k+D)}\Deltat^{-\eta_k/2}. 
\end{multline*}
With a little algebra, the proof is then complete.
\end{proof}

\begin{definition}[Quantitatively globally uniformly shearing flows]\label{def:QGUS_flow}
	The flow $\flowR$ is called {\em quantitatively globally uniformly shearing} (QGUS for short) if there exists $a\in (0,1/2]$ and, for every $k \in \Z_{\geq 0}$, there exists  
	$0< \gamma_k\leq 1/4$, so that for every $M\geq 1$ there exists a measurable set $G_M \subseteq X$, with $\mu(X \setminus G_M) \leq M^{-\gamma_k}$, and for any $(k+1)$-tuple of times $\tb$, with $\Deltat \geq 4M$, there exist $M^{-\gamma_k}$-almost measure preserving maps $\mapg^{\tb}_m \colon X \to X$, for $m=0, \dots, K_{M} = \lfloor M^{\gamma_k} \rfloor$, satisfying the following property
	\begin{itemize}
\item[(QGUS1)] for every $j=0, \dots, k$, we have
	\[
\sup_{x\in G_M} \dist \left(\flow_{t_j}(\mapg^{\tb}_m x), \flow_{t_j+m \cdot M^a (t_j/t_k) }(x) \right)\leq M^{-\gamma_k}.
\]
	\end{itemize}
\end{definition}

The following result says that QGUS flows which are polynomially mixing are polynomially mixing of all orders, with some explicit bound on the rates.

\begin{theorem}\label{thm:multiple_mixing_QGUS_flows}
	Let $\flowR$ be a QGUS flow, and assume it is polynomially mixing. 
	Then, for any $k\geq 2$, $\flowR$ is polynomially $k$-mixing with $d_{k+1}=d_k+D$. 
\end{theorem}
\begin{proof}
	The proof is analogous to the proof of \Cref{thm:multiple_mixing_GG_flows} and proceeds by induction: let $k\geq 2$, and assume that $\flowR$ is polynomially $k$-mixing with rate $\eta_k$, which we assume to be $\leq 1$. 
	Let $\phi_0,\ldots, \phi_{k}\in \mathscr{C}^{\infty}_c(X)$ and a $(k+1)$-tuple of times $\tb$ be fixed, with $\Deltat\geq 1$.
	Analogously as before, when writing $A \ll B$, we mean $A \leq C \cdot B$ for a constant $C\geq 0$ whose value depends only on $k$.
	
	Let $a \in (0,1/2]$ and $\gamma_k$ be given by the QGUS property. 
	Let us fix 
	\[
	\beta_k = \frac{a\eta_2^2\gamma_k}{32\sigma_k}, \qquad \text{and} \qquad M= \left(\frac{\Deltat}{4}\right)^{\frac{\eta_k\eta_2}{64\sigma_k^2}} \leq \frac{\Deltat}{4},
	\]
	and consider the set $G = G_M$ coming from \Cref{def:QGUS_flow}. 
	
	We apply \Cref{lem:phi_phi_orthogonal} to $\phi_{k}$ with $\varepsilon = M^{-\beta_k} $, so that we can replace $\phi_k$ with $\phi_k^{\perp} + \mu(\phi_k)$, up to an error $\ll \cN_{d_2}(\phi_k) \varepsilon$. Furthermore, using the polynomial $k$-mixing assumption, we can replace $\phi$ simply with $\phi_k^{\perp}$, keeping in mind that $\mu(\phi_{k}^{\perp}) = 0$, $\|\phi_k^{\perp}\|_{\infty} \ll \|\phi_k\|_{\infty}$, and, for all $r\geq 1$, we have $\cN_r(\phi_k^{\perp}) \ll \varepsilon^{-8\sigma(r)/\eta_2} \cN_{r}(\phi_k)$.
	To uniformize the notation, we call $\psi_i = \phi_i$ for $i=0,\dots, k-1$, and $\psi_k = \phi_k^{\perp}$, so that it suffices to show 
	\begin{equation}\label{eq:new_goal}
	\left\lvert \int_{X}\prod_{i=0}^{k} \psi_i\circ \flow_{t_i} \diff\mu\right\rvert  \ll   \left(\prod_{i=0}^{k} \cN_{d_k+D}(\phi_i)\right) \left(M^{-\beta_k} + \Deltat^{-\frac{a\eta_k\eta_2}{32 \sigma_k}}\right).
	\end{equation}
	
	From the QGUS property, for all $m \leq K_{M} = \lfloor M^{\gamma_k}\rfloor$, we have a $M^{-\gamma_k}$-almost measure preserving map $\mapg^{\tb}_m \colon X \to X$.	
	We set $K = \lfloor K_{M} /2 \rfloor$.
	Since $\dist(x, \mapg^{\tb}_m (x)) \leq M^{-\gamma_k}$, in the same way as we did in the proof of \Cref{thm:multiple_mixing_GG_flows}, we obtain
	\begin{equation*}
		\left\lvert \int_{X}\prod_{i=0}^{k} \psi_i\circ \flow_{t_i} \diff\mu\right\rvert  \ll  \left\lvert  \int_{X} \psi_0 \cdot \left[\frac{1}{K} \sum_{m=1}^K \prod_{i=1}^{k} \psi_i\circ \flow_{t_i} \circ \mapg^{\tb}_m \right] \diff\mu\right\rvert  + \left(\prod_{i=0}^{k} \cN_{d_2}(\phi_i)\right) M^{-\gamma_k},
	\end{equation*}
since, we recall, we have $\|\psi_0\|_{\mathscr{C}^1} \ll \cN_{d_2}(\phi_0)$ and $\|\psi_i\|_{\infty} \ll \|\phi_i\|_{\infty} \ll \cN_{d_2}(\phi_i)$ .

	We apply the Cauchy-Schwarz Inequality and the Van der Corput Inequality (\Cref{lem:vdc} in \Cref{sec:vdc}) with $L = \lfloor \sqrt{K} \rfloor$, which yield
	\begin{multline*}
		\left\lvert \int_{X}\prod_{i=0}^{k} \psi_i\circ \flow_{t_i} \diff\mu\right\rvert \ll \cN_{d_2}(\phi_0) \max_{\substack{0 \leq m,l \leq K \\ |m-l| \geq 1} }\left\lvert \int_{X} \prod_{i=1}^{k} \psi_i\circ \flow_{t_i} \circ \mapg^{\tb}_m \cdot \psi_i\circ \flow_{t_i} \circ \mapg^{\tb}_{l} \diff\mu \right\rvert^{1/2} \\ + \left(\prod_{i=0}^{k} \cN_{d_2}(\phi_i)\right) M^{-\gamma_k/4}.
	\end{multline*}
	
	Let us now fix $0\leq m,l \leq K$ with $|m-l|\geq 1$, and we consider the absolute value in the right hand side above. 
	By splitting the integral above over $G_M$, on which we can use (QGUS1) in \Cref{def:QGUS_flow}, and over $X\setminus G_M$, which has measure bounded by $M^{-\gamma_k}$, we get 
	\begin{multline*}
		\left\lvert \int_{X}  \prod_{i=1}^{k} \psi_i\circ \flow_{t_i} \circ \mapg^{\tb}_m \cdot \psi_i\circ \flow_{t_i} \circ \mapg^{\tb}_{l} \diff\mu \right\rvert \\
		\ll \left\lvert \int_{X} \prod_{i =1}^{k} (\psi_i \cdot \psi_i\circ \flow_{(l-m) \cdot M^a (t_i/t_k)} )\circ \flow_{t_i + m \cdot M^a (t_i/t_k)} \diff\mu \right\rvert +\left(\prod_{i=1}^{k} \cN_{d_2}(\phi_i)^2\right) M^{-\gamma_k/2}.
	\end{multline*}
	Note that in the previous inequalities we used the bounds 
	\[
	\|\psi_k\|_{\mathscr{C}^1} \ll \cN_{d_2}(\psi_k) \ll \varepsilon^{-8\sigma(d_2)/\eta_2} \cN_{d_2}(\phi_k) \ll M^{\gamma_k/4}\cN_{d_2}(\phi_k).
	\]
	By invariance of $\mu$, the integrand above can be composed by $\flow_{-\overline t}$, where ${\overline t} = t_1+ m \cdot M^a (t_1/t_k)$,
	so that 
	\begin{multline*}
		\left\lvert \int_{X}  \prod_{i=1}^{k} \psi_i\circ \flow_{t_i} \circ \mapg^{\tb}_m \cdot \psi_i\circ \flow_{t_i} \circ \mapg^{\tb}_{l} \diff\mu \right\rvert \\
		\ll \left\lvert \int_{X} \prod_{i =1}^{k} (\psi_i \cdot \psi_i\circ \flow_{(l-m) \cdot M^a (t_i/t_k)} )\circ \flow_{t_i'} \diff\mu \right\rvert +\left(\prod_{i=1}^{k} \cN_{d_2}(\phi_i)^2\right) M^{-\gamma_k/2},
	\end{multline*}
	where the $k$-tuple of times $\tb'$ with components $t_i' = t_i  + m \cdot M^a (t_i/t_k)-{\overline t}$ satisfies $\Deltat' \geq \Deltat - 2KM^a \geq \Deltat / 2$, since $2KM^a \leq 2M \leq \Deltat/2$.
	Note also that $|(l-m) \cdot M^a (t_i/t_k)|\leq M$ and $|(l-m)M^a| \geq M^a$.
	Thus, we can apply \Cref{lem:uniform_polynomial_k_mixing} with $R_{-} = M^a$ and $R_{+}=M$, which yields
	\begin{multline*}
		\left\lvert \int_{X}  \prod_{i=1}^{k} \psi_i\circ \flow_{t_i} \circ \mapg^{\tb}_m \cdot \psi_i\circ \flow_{t_i} \circ \mapg^{\tb}_{l} \diff\mu \right\rvert \\
		\ll \left(\prod_{i=1}^{k} \cN_{d_k+D}(\psi_i)^2\right) (M^{-a\eta_2} + M^{\sigma_k}\Deltat^{-\frac{\eta_k \eta_2}{16\sigma_k}})+\left(\prod_{i=1}^{k} \cN_{d_2}(\phi_i)^2\right) M^{-\gamma_k/2} \\
		\ll \left(\prod_{i=1}^{k} \cN_{d_{k+1}}(\phi_i)^2\right) (M^{-a\eta_2/2} + M^{\sigma_k + a\gamma_k\eta_2} \Deltat^{-\frac{\eta_k \eta_2}{16\sigma_k}}+ M^{-\gamma_k/2}).
	\end{multline*}
	Substituting the expression for $M$ proves \eqref{eq:new_goal} and hence completes the proof of 
	\Cref{thm:multiple_mixing_QGUS_flows}.
\end{proof}


\section{Time changes of unipotent flows}\label{sec:tc_unip_flows}

In this section, we are going to prove \Cref{th:main_tc_unip}, namely that time-changes of polynomially mixing unipotent flows are polynomially mixing of all orders.

\subsection{Definitions and preliminaries}

Let $G$ be a connected semisimple Lie group with finite centre and without compact factors, and let $\lieg$ be its Lie algebra. An element $U \in \lieg \setminus \{0\}$ is called \emph{unipotent} if the linear operator $\mathfrak{ad}_U = [U, \cdot] \colon \lieg \to \lieg$ is nilpotent.

Let $\Gamma \leq G$ be an irreducible lattice, and equip the smooth manifold $X= \Gamma \backslash G$ with the normalized Haar measure $\vol$. We can identify any element $V \in \lieg \setminus \{0\}$ with a left-invariant vector field on $X$; in particular, any such $V$ generates a smooth flow $(h^V_t)_{t \in \R}$ on $X$ given by $h^V_t(x)=x\, \exp(tV)$, which preserves $\vol$. If $U \in \lieg \setminus \{0\}$ is unipotent, the associated smooth flow $(h^U_t)_{t \in \R}$ is called a \emph{unipotent flow}.

Let $\alpha\colon X \to \R_{>0}$ be a positive smooth function. The vector field $U_\alpha = \frac{1}{\alpha}U$ on $X$ is parallel to $U$, hence the flow $\flowR$ generated by $U_\alpha$ has the same orbits as the unipotent flow $(h^U_t)_{t \in \R}$, but percorred at different speed. The flow $\flowR$ is the \emph{time-change of $(h^U_t)_{t \in \R}$ generated by $\alpha$}. It is not hard to see that $\flowR$ preserves the smooth measure $\mu$ defined by $\diff \mu = \alpha \diff \vol$. We will assume that $\alpha$ is normalized to have integral 1 on $X$.

\begin{definition}\label{def:admissible_tc}
	A time-change $\flowR$ of $(h^U_t)_{t \in \R}$ generated by $\alpha$ is called admissible if $\vol(\alpha) = 1$ and there exists a compact set $K \subseteq X$ such that $\alpha$ is constant on $X \setminus K$.
\end{definition}

Clearly, if $X$ is compact, any positive smooth function, normalized to have integral 1, generates an admissible time-change.

Admissible time-changes of unipotent flows as above are polynomially mixing, as proved in \cite{FU} for the classical horocycle flow on compact manifolds and in \cite{Rav} in general. 
In turn, from polynomial mixing, we can obtain $L^2$-estimates for the ergodic integrals. More precisely, the same proof as in \cite[Proposition 3.1]{Rav} gives us the following result.

\begin{proposition}\label{prop:erg_int_tc_unip}
	There exist $\eta \in (0,1)$, $C_\eta >0$, and a norm $\mathcal{N} = \mathcal{N}_d$ such that, for every $\phi \in \mathscr{C}^{\infty}_c(X)$ and for every $T \geq 1$, there exists a set $E(\phi,T) \subset X$ of measure $\mu(E(\phi,T)) \leq T^{-\eta}$ so that 
	\[
	\left\lvert \int_0^t \phi \circ \flow_r(x) \diff r - t \mu(\phi)\right\rvert \leq C_\eta \, \mathcal{N}(\phi) \, t^{1-\eta}, \qquad \text{for all $t\geq T$ and $x \in X \setminus E(\phi,T)$.} 
	\]
\end{proposition}
In the following, for simplicity, we are going to assume that $\eta \in (0,1/4)$.

Finally, we finish this subsection by discussing the shearing properties of admissible time-changes. The Jacobson-Morozov Theorem ensures the existence of a $\mathfrak{sl}_2$-triple containing $U$; in particular, there exists $Y \in \lieg$ such that $[Y,U] = U$. 
The flows $\flowR$ and $(h^Y_t)_{t \in \R}$ satisfy the following commutation relation, whose proof is contained in \Cref{sec:shearing}.

\begin{proposition}\label{prop:tc_geo_commute}
	There exists a smooth function $z \colon X \times \R \times \R \to \R$ such that 
	\[
	\flow_t \circ h^Y_s(x) = h^Y_s \circ \flow_{z(x,s,t)}(x),
	\]
	for all $x \in X$ and $t,s \in \R$. Moreover, there exists a constant $C_\alpha > 0$ such that 
	\[
	\left\lvert z(x,s,t) - t \right\rvert \leq C_{\alpha} s \, t,
	\]
	and
	\[
	\left\lvert z(x,s,t) - t - st + s\int_0^t \frac{Y \alpha }{\alpha}\circ \flow_r(x) \diff r \right\rvert \leq C_\alpha s^2 |t| (1+|st|).
	\]
\end{proposition}

In the remainder of this section, the notation $A \ll B$ stands for $A \leq C \cdot B$, where $C$ is a constant depending on $X$, $U$, and $\alpha$ only, whose value is allowed to change in different occurrencies.

\subsection{Proof of \Cref{th:main_tc_unip}}

We now show that admissible time-changes of unipotent flows are QGUS flows.

\begin{proposition}\label{prop:tc_unip_are_QGUS}
	Let $\flowR$ be an admissible time change of a unipotent flow. Then, $\flowR$ is QGUS with $a=\eta/2$ and with $\gamma_k = \eta/4$.
\end{proposition}

This immediately implies \Cref{th:main_tc_unip}.

\begin{proof}[Proof of \Cref{th:main_tc_unip}]
	The existence of the Sobolev norms satisfying the required properties is proven in \cite[Section 2.2]{BEG}.
 By \Cref{prop:tc_unip_are_QGUS}, $\flowR$ is QGUS with  polynomial rate functions. The main result in \cite{Rav} implies that it is polynomially mixing. The conclusion follows from \Cref{thm:multiple_mixing_QGUS_flows}.
\end{proof}

\begin{proof}[Proof of \Cref{prop:tc_unip_are_QGUS}]
As in the statement, we fix $a=\eta/2$ and  $\gamma_k = \eta/4$.
Let $M\geq 1$ be fixed, and consider the set $G_M = E(Y\alpha / \alpha, M)$ given by \Cref{prop:erg_int_tc_unip} for the function $Y\alpha / \alpha \in \mathscr{C}_c^{\infty}(X)$. By the aforementioned result, $\mu(G_M) \leq M^{-\eta} \leq M^{-\gamma_k}$.
Let $\tb$ be a $(k+1)$-tuple of times with $\Deltat \geq 4M$.

For any $m = 0, \dots, K=\lfloor M^{\gamma_k}\rfloor \leq M^{\eta/4}$, we define 
\[
\mapg^{\tb}_m(x) = h^Y_{mM^{\eta/2} t_k^{-1}}(x).
\]
Then, $\mapg^{\tb}_m$ is $M^{-\gamma_k}$-almost measure preserving. Indeed, for any $\phi \in L^{\infty}(X)$, we have
\begin{multline*}
\left\lvert
\int_X \phi \circ \mapg^{\tb}_m \diff \mu - \int_X \phi \diff \mu \right\rvert= 
\left\lvert \int_X \phi \circ  h^Y_{m M^{\eta/2} t_k^{-1}} \cdot \alpha \diff \vol - \int_X \phi \cdot \alpha \diff \vol  \right\rvert \\
\leq  
\| \phi\|_{\infty} \,
\int_X  \left\lvert\alpha \circ  h^Y_{-mM^{\eta/2} t_k^{-1}} - \alpha  \right\rvert \diff \vol \ll \| \phi\|_{\infty} \, M^{3\eta/4} \Deltat^{-1}  \ll  \| \phi\|_{\infty} \, M^{-\gamma_k}.
\end{multline*}
Now, let $x \in G_M$ and fix $j \in \{0, \dots, k\}$. If $j=0$, then 
\[
\dist \left(\mapg^{\tb}_m (x), x \right) \leq |mM^{\eta/2}t_k^{-1}| \leq M^{3\eta/4-1} \leq M^{-\gamma_k},
\]
so that the claim is obvious. If $j >0$, by \Cref{prop:tc_geo_commute}, we have 
\[
\left\lvert \frac{1}{t_j} \int_0^{t_j} \frac{Y\alpha}{\alpha} \circ \flow_r(x) \diff r\right\rvert \ll t_j^{-\eta},
\]
since $t_j \geq \Deltat \geq 4M$. Therefore, we obtain
\[
\begin{split}
&\dist \left(\flow_{t_j}(\mapg^{\tb}_m x), \flow_{t_j+m \cdot M^{\eta/2}(t_j/t_k) }(x) \right) \\
&\qquad  \leq \dist \left(\flow_{t_j}(\mapg^{\tb}_m x), \flow_{z\big( x,mM^{\eta/2}t_k^{-1},t_j \big)}(x) \right) + \left\lvert z \big(x,mM^{\eta/2}t_k^{-1},t_j\big) - t_j-m \cdot M^{\eta/2} (t_j/t_k)\right\rvert \\
& \qquad \ll M^{-\gamma_k} + M^{3\eta/4} t_j^{-\eta} + \frac{M^{3\eta/2}}{t_k} (1+M^{3\eta/4})\ll M^{-\gamma_k},
\end{split}
\]
since, we recall, $\eta \leq 1/4$. 
The proof is thus complete.
\end{proof}


\section{Intermezzo: towers and almost partitions}\label{sec:towers}

This section contains some basic facts about \emph{towers} and \emph{almost partitions} which we will need in the remainder of the paper.

Let $(Y, \mathcal{B}, \mu)$ be a Lebesgue measure space, where $(Y, \dist)$ is a complete metric space. 
\begin{definition}[Almost partitions]
	Let $\varepsilon \geq 0$ and let $A \subseteq Y$ be measurable. A family $\cP=\{P_a\}_{a\in \mathscr{A}}$ of measurable sets $P_a \subseteq Y$ is an \emph{$\varepsilon$-almost partition of $A$} if the sets $P_a$ are pairwise disjoint and 
	\[
	\mu \left(A \, \triangle \, \bigcup \{P_a \in \cP : P_a \cap A \neq \emptyset\}\right) <\varepsilon \mu(A).
	\]
\end{definition}
One can verify that the intersection of almost partitions is again an almost partition, see \Cref{lem:intersection_almost_partitions} in \Cref{sec:vdc}.

Let us consider an ergodic (aperiodic) flow $\flowR$ on $Y$. We define towers, which can be thought of as measurable flowboxes.  

\begin{definition}\label{def:tower} 
Let $B \in \mathcal{B}$ and $\height \geq 0$.
The set $\cT=\bigcup_{0\leq t\leq \height}\flow_t(B)$ is called a {\em tower with base $B$ and height $\height$} if, for any $t,t'\in [0,\height]$ with $t \neq t'$, we have $\flow_t(B)\cap \flow_{t'}(B)=\emptyset$. 
\end{definition}

From the definition, it is easy to see that any point $x$ in a given tower $\cT$ of base $B$ and height $\height$ can be written in a unique way as $x=\flow_t(y)$ for a point $y \in B$ and $t \in [0,\height]$.

The following classical result, due to Rokhlin \cite{Rok} in the discrete setting (see \cite{Orn} for a version for flows and \cite{Lin} for $\R^d$-actions), says that we can find towers with arbitrarily large height and measure arbitrarily close to 1.
\begin{theorem}[Rokhlin]
Let $\flowR$ be as above. For any $\varepsilon>0$ and any $\height \geq 0$, there exists a tower $\cT$ of height $\height$ and base $B \in \mathcal{B}$ such that $\mu(\cT)\geq 1-\varepsilon$. Furthermore, there exists a measure $\nu$ on $B$ such that the restriction of $\mu$ to $\cT$ is the product of $\nu$ and the Lebesgue measure on $[0,\height]$.
\end{theorem}

We will need the following important fact about towers: the intersection of any two towers can be approximated by another tower, which in particular is not empty if the intersection is sufficiently large. Let us start with a preliminary step.
\begin{lemma}\label{lem:towers1}
Let $\cT$ be any tower of base $B$ and height $\height$.
For every $\varepsilon >0$ there exists $B' \subseteq B$ so that the tower $\cT' \subseteq \cT$ of base $B'$ and height $\height$ satisfies $\mu( \cT \setminus \cT') \leq \varepsilon$ and the restriction of $\mu$ to $\cT'$ can be expressed as the product of a measure $\nu$ on $B'$ and the Lebesgue measure on $[0,\height]$.
\end{lemma}
\begin{proof}
Fix $\varepsilon >0$, and let $\cR$ be a Rokhlin tower of height $H = 4\height / \varepsilon$ so that $\mu(X \setminus \cR) \leq \varepsilon / 2$. Let $A$ denote the base of $\cR$, and define $\cR_i = \bigcup \{ \flow_t(A) : i \cdot \height \leq t \leq H-\height\}$ for $i=0,1$. Note that $\mu(\cR \setminus \cR_1) \leq 2 \height / H = \varepsilon/2$.
Defining $B' = B \cap \cR_0$ ensures that the tower $\cT'$ of height $\height$ over it is fully contained in $\cR$, hence the claim on the restriction of $\mu$ to $\cT'$ follows from Rokhlin's Theorem. Finally, since $\cT \setminus \cT' \subseteq X \setminus \cR_1$, we conclude $\mu(\cT \setminus \cT') \leq \varepsilon$.
\end{proof}

\begin{lemma}\label{lem:towers2}
Let $\cT_1,\cT_2$ be two towers of bases $B_1, B_2$ and heights $\height_1, \height_2$ respectively; let $H = \min\{ \height_1, \height_2\}$. For any $\overline{\height} \leq H$, there exists a tower $\overline{\cT} \subseteq \cT_1 \cap \cT_2$ of height $\overline{\height}$ such that 
\[
\mu \left( \cT_1 \cap \cT_2 \setminus \overline{\cT} \right) \leq 4\frac{\overline{\height}}{H}.
\]
\end{lemma}
Although in general the base of the tower $\overline{\cT}$ in the statement above might be empty, clearly this is not the case as soon as $\mu(\cT_1 \cap \cT_2 ) > 4\overline{\height} / H$.
\begin{proof}
We begin by applying \Cref{lem:towers1} to both $\cT_1$ and $\cT_2$, with $\varepsilon = \overline{\height} / H$, and let us call $\cT_i', B_i', \nu_i'$ the resulting towers, bases, and measures. 
It suffices to prove that there exists a tower $\overline{\cT} \subseteq \cT_1' \cap \cT_2'$ of height $\overline{\height}$ such that $\mu \left( \cT_1' \cap \cT_2' \setminus \overline{\cT} \right) \leq 2\frac{\overline{\height}}{H}$. For the sake of notation, in the following we suppress the $'$ symbol.

Let us define 
\[
\widetilde{B} := (B_1 \cap \cT_2) \cup (B_2 \cap \cT_1) \subset B_1 \cup B_2,
\]
and let $\nu$ be the measure on $\widetilde{B}$ given by $\nu_1$ and $\nu_2$ (note that, by definition, $\nu_1$ and $\nu_2$ coincide on $B_1 \cap B_2$, so that $\nu$ is well-defined on $\widetilde{B}$).

It is easy to see that any $x \in \cT_1 \cap \cT_2$ can be written in a unique way as $x = \flow_t(y)$ for some $y \in \widetilde{B}$ and some $0\leq t\leq H$. Thus, there exists a function $h \colon \widetilde{B} \to [0,H]$ such that
\[
\cT_1 \cap \cT_2 = \{ \flow_t(x) : x \in \widetilde{B} \text{\ and\ } t \in [0, h(x)]\}.
\]
Given $\overline{\height} \leq H$, we define
\[
A_i = \{x \in \widetilde{B} : h(x) \geq i \cdot \overline{\height} \}, \qquad \text{for\ }i = 0, \dots, K=\lfloor H / \overline{\height} \rfloor,
\]
and note that $A_K \subseteq A_{K-1} \subseteq \cdots \subseteq A_0 = \widetilde{B}$.
We claim that the tower $\overline{\cT}$ of height $\overline{\height}$ and base 
\[
\overline{A} := \bigcup_{j = 1}^{K-1} \flow_{(j-1)\height} (A_j) = A_1 \cup \flow_{\height} (A_2) \cup \cdots \cup \flow_{(K-1)\height} A_{K-1}
\]
satisfies the conclusion of the lemma. Indeed, by construction, we have
\[
\cT_1 \cap \cT_2 \setminus \overline{\cT} \subseteq \bigcup \{ \flow_t(A_j \setminus A_{j+1}) : t \in [j\cdot \overline{\height}, (j+1)\cdot \overline{\height}) \text{\ and\ } j = 0, \dots, K\},
\]
so that 
\[
\mu\left( \cT_1 \cap \cT_2 \setminus \overline{\cT} \right) \leq \sum_{j=0}^K \overline{\height} \cdot \nu(A_j \setminus A_{j+1}) = \overline{\height} \cdot \nu(\widetilde{B}).
\]
Finally, since the tower of height $H$ over $B_1 \cap \cT_2$ is contained in $\cT_1$, we have that $\nu(B_1 \cap \cT_2) \leq \mu(\cT_1) / H$, and similarly for $B_2 \cap \cT_1$. It follows that $\nu(\widetilde{B}) \leq 2/H$, and hence $\mu\left( \cT_1 \cap \cT_2 \setminus \overline{\cT} \right) \leq 2 \overline{\height} / H$, which completes the proof.
\end{proof}

We now assume that $Y$ is a \emph{suspension space}, namely we place ourselves in the following setting.
Let $(X,\dist)$ be a compact metric space, let $\nu$ be a Borel probability measure on $X$, and let $T \colon X \to X$ be an ergodic (aperiodic) probability preserving automorphism. 
Let $\roof \colon X \to \R_{>0}$ be a positive integrable function, which we assume to be normalized so that $\nu(\roof) = 1$. The space $Y = \susp$ is defined as
\[
\susp = \{ (x,r) \in X \times \R : 0\leq r \leq \roof(x) \} / \sim,
\]
where $\sim$ is the equivalence relation generated by $(x, \roof(x)) \sim (Tx,0)$.
The space $\susp$ is equipped with the probability measure $\mu$ given by the product of $\nu$ on $X$ with the Lebesgue measure on the \lq\lq vertical\rq\rq\ fibers contained in $\R$. Then, the \emph{special flow over $T$ under the roof function $\roof$} is the flow $\flowR$ on $\susp$ obtained by moving each point vertically with unit speed. 
Explicitly, we can write
\[
\flow_t(x,r) = \left( T^{N(x,r,t)}x, r+t - S_{N(x,r,t)}\roof(x)\right),
\]
where $S_n\roof(x)$ denote the Birkhoff sums of $\roof$ at $x \in \baseX$, namely
\[
S_n\roof(x) = 
\begin{cases}
\roof(x) + \roof(Tx) + \cdots + \roof(T^{n-1}x), & \text{ if } n>0,\\
0 & \text{ if } n=0,\\
- \roof(T^{-1}x) - \cdots - \roof(T^nx) & \text{ if } n<0,
\end{cases}
\]
and $N(x,r,t)$ is the unique integer satisfying
\[
0\leq r+t - S_{N(x,r,t)}\roof(x) < \roof( T^{N(x,r,t)}x).
\]
We will often write $N(x,t)$ instead of $N(x,0,t)$.
We note that the probability measure $\mu$ is invariant under $\flowR$.

We now show that we can construct Rokhlin towers with arbitrarly large height and with base contained in $X$.

\begin{lemma}\label{lem:towers3}
	Assume that $\roof \in L^p(X,\nu)$ for some $p >1$. There exists a constant $C_{\roof}\geq 1$ for which the following holds. Let $\cT$ be a tower of base $B$ and height $\height \geq 16$. There exists a tower $\cT' \subseteq \cT$ of base $B \subseteq X$ and height $\height' \geq \height/2$ such that $\mu(\cT \setminus \cT') \leq C_{\roof} \height^{-\alpha}$, where $\alpha = \frac{p-1}{p+1} \in (0,1)$.
\end{lemma}
\begin{proof}
We are going to construct the base $B'$ of the desired tower by removing part of the base of $\cT$ and by sliding the remainder part in $X$. To this end, define 
\[
C_i =\{(x,r) \in \susp : \roof(x) \geq i \cdot \height^{1-\alpha} \text{ and } (i-1) \cdot \height^{1-\alpha} \leq r < i \cdot \height^{1-\alpha} \},
\]
and let $C = \cup_{i \geq 1} C_i$. Note that for each point $z = (x,r) \in \susp \setminus C$, there exists $0\leq t(z) \leq \height^{1-\alpha}$ such that $\flow_{t(z)}(z) \in X$.
We also note that, if $(x,y_1)$ and $(x,y_2)$ both belong to $B \cap C_i$ for some $i$, then $y_1=y_2$.
Therefore, using Chebyshev's Inequality, we get
\[
\nu(B \cap C_i ) \leq \nu( \{ x : \roof(x) \geq i \cdot \height^{1-\alpha} \}) \leq (i \cdot \height^{1-\alpha})^{-p} \, \|\roof\|_p^p.
\]
Therefore, the set $B_0 = B \setminus C$ has measure $\nu(B_0) \geq \nu(B) - \zeta(p) \|\roof\|_p^p \height^{-p(1-\alpha)}$. Let $B' \subseteq X$ be the first entrance of points of $B_0$ to $X$, which, as we observed, happens before time $\height^{1-\alpha}$, and let $\cT'$ be the tower of height $\height - \height^{1-\alpha} \geq \height/2$. Since $\cT'$ is fully contained in $\cT$, it is indeed a tower. Finally,
\[
\begin{split}
	\mu(\cT \setminus \cT') \leq \nu(B \setminus B_0) \, \height + 2 \nu(B) \, \height^{1-\alpha} \leq \zeta(p) \|\roof\|_p^p \height^{1-p+p\alpha} + 2\height^{-\alpha} \leq C_{\roof} \height^{-\alpha} ,
\end{split}
\] 
where we set $C_{\roof} := \zeta(p) \|\roof\|_p^p +2$.
\end{proof}

\begin{corollary}\label{cor:towers3}
Assume that $\roof \in L^p(X,\nu)$ for some $p >1$. For every $\varepsilon >0 $ and $H\geq 1$ there exists a Rokhlin tower $\cR$ with base $B \subseteq X$ and height $\height \geq H$ such that $\mu(X \setminus \cR) \leq \varepsilon$.
\end{corollary}
\begin{proof}
Let $\varepsilon$ and $H$ be given.
Let $\cR_0$ be a Rokhlin tower of height $\height$ and base $B_0$ such that $\mu(X \setminus \cR_0)\leq \varepsilon /2$, where $\height \geq H$ satisfies 
\[
\height^{-\alpha} < \frac{\varepsilon}{2 \zeta(p) \| \roof\|_p^p + 4}, \qquad \text{and} \qquad \height > 2H.
\]
By \Cref{lem:towers3}, there exists a Rokhlin tower $\cR \subseteq \cR_0$ with base $B \subseteq X$ and height $\geq \height/2 >H$ so that $\mu(\cR_0 \setminus \cR) \leq \varepsilon/2$, which proves the corollary.
\end{proof}


\section{Locally uniformly shearing flows}\label{sec:LUSflows}
In this section, we introduce the central notion of locally uniformly shearing (LUS) flows, which is a generalization of the GUS property introduced in \Cref{sec:GUSflows} (see \Cref{rmk:comments_LUS} below). Using the same notation, we give the following definition.

\begin{definition}[Locally uniformly shearing flows] \label{def:LUS_flows}
	The flow $\flowR$ is called {\em locally uniformly shearing} (LUS for short) if there exist $a \in (0,1/2]$, $d \geq 1$, an increasing sequence $(M_s)_{s \in \N}$, and, for every $\delta >0$ and $k \in \Z_{\geq 0}$, there exists $M_{k,\delta} > \delta^{-2}$ such that for every $M_s, \height \geq M_{k,\delta}$ there exist $t_{\height, M_s}\geq \max\{M_s, \height\}$ and a $\delta$-almost partition $\{\cR_{\varrho}\}_{\varrho =1}^d$ of $X$ into towers $\cR_{\varrho}$ of heights $\height_{\varrho} \in [\height^{1/2},\height^{3/2}]$ such that, for every $(k+1)$-tuple of times $\tb$ with $\Deltat > t_{\height, M_s}$, the following conditions hold:
	\begin{enumerate}
		\item[(LUS1)] there exist $\delta$-almost measure preserving maps $\mapg^{\tb}_m \colon \cup_{\varrho=1}^d \cR_{\varrho} \to \cup_{\varrho=1}^d \cR_{\varrho}$, for $m=0,\dots, K_{\delta} = \lfloor \delta^{-1}\rfloor$, with the following properties
		\begin{enumerate}
			\item for all $x$ 
			in the proper domain of $\mapg^{\tb}_m$, we have $\dist(\mapg^{\tb}_m(x),x)\leq \delta$,
			\item for any tower $\cT$ of height $H \geq \height^{1/2}$, there exists a tower $\overline{\cT} \subseteq (\mapg^{\tb}_m)^{-1}(\cT)$ of height $\height^{1/4}$ with
			\[
			\mu \left( (\mapg^{\tb}_m)^{-1}(\cT) \setminus \overline{\cT} \right) \leq d \,\delta;
			\]
		\end{enumerate} 
		\item[(LUS2)] for every $\varrho= 1,\dots, d$, there exists a $\delta$-almost partition $\{\Tow_{\varrho, j}\}_{j=1}^k$ of $\cR_{\varrho}$ into towers $\Tow_{\varrho, j} \subseteq \cR_{\varrho}$ of the same height\footnote{In particular, this means that the basis of each towers $\Tow_{\varrho, j}$ is a subset of the basis of $\cR_{\varrho}$.} $\height_{\varrho}$, so that, for every $m=0,\dots,K_{\delta}$, we have 
		\[
	\max_{j=1, \dots, k} \sup_{x\in \Tow_{\varrho,j}} \dist \left(\flow_{t_j}(\mapg^{\tb}_m x), \flow_{t_j+m \cdot M_s^a }(x) \right)\leq \delta;
		\]
		\item[(LUS3)] for every $\varrho=1, \dots, d$, there exists a $\delta$-almost partition of $\cR_{\varrho}$ into (possibly empty) towers $R_{\varrho}(i,m,\ell) \subseteq \cR_{\varrho}$, with $\ell = 0, \dots, \delta^{-1}M-1$, of the same height\footnote{As in (LUS2), this implies that the basis of each tower $R_{\varrho}(i,m,\ell)$ is a subset of the basis of $\cR_{\varrho}$.} $\height_{\varrho}$ such that, for any $i=1, \dots, k$, any $m=0,\dots,K_{\varepsilon}$, and any $\ell$ as above, we have
		\[
		\sup_{x\in R_\varrho(i,m,\ell)} \dist \left(\flow_{t_i}(\mapg^{\tb}_m x), \flow_{t_i+ \ell \cdot \delta}(x) \right)\leq \delta.
		\]
		\end{enumerate}
\end{definition}

\begin{remark}\label{rmk:comments_LUS}
\Cref{def:LUS_flows} is a generalization of the notion of GUS flows in \Cref{sec:GUSflows}, in which we require weaker assumptions on the shear of the flow on the points $x$ and $\mapg^{\tb}_m(x)$. To explain what we mean, let us assume for simplicity that $d=1$, so that there is only one Rokhlin tower $\cR$ of height $\height$. Since we are working with towers, condition (LUS1) asks for the almost measure preserving maps $\mapg^{\tb}_m$ to \lq\lq almost preserve towers\rq\rq, which is a natural requirement. Condition (LUS2) implies that, for any $j = 1,\dots, k$, we can find a subtower $\Tow_j$ of the Rokhlin tower $\cR$ with the same height $\height$ on which we see a uniform shear. If we had only one nonempty tower $\Tow_k$ corresponding to $j=k$, we would recover (GUS1) for the time $t_k$ (note that on $\Tow_{j}$ we only control the shear at time $t_j$). In general, we need some (weak) control on the complement of the subtower $\Tow_j$, as well as for other times $t_i$. This is ensured by condition (LUS3), up to taking finer subtowers: we can partition $\cR$ into finer subtowers, still with height $\height$, on which the shear for any fixed time $t_i$ (if there is any) has a uniform upper bound and is almost independent on the point. We stress that, for fixed $i$ and $m$, some of the towers $R(i,m,\ell)$ might be empty, depending on the relative sizes of the times $t_i$: if, for example, $t_1$ is much smaller than $t_k$, it could happen that the almost partition of $\cR$ consists of a single tower $R(1,m,0)$ on which there is no shear at all (since $\ell =0$).
\end{remark}

The remainder of this section contains the main abstract result of the paper: we show how to upgrade mixing for LUS flows to mixing of all orders. 
We begin by defining the properties $\Par(k)$, which will be a useful intermediate step in proving the main theorem.

\subsection{Property $\Par(k)$}

We recall that, for the (only) 1-tuple of times $\tb = (0)$, we set $\Deltat = \infty$.

\begin{definition}\label{def:property_park}
	Let $k\geq 1$. A flow $\flowR$ on $(X,\mu)$ has {\em property $\Par(k)$} if the following holds. Let $\phi_0,\ldots, \phi_{k-1}\in \Lip(X)$, with $\|\phi_i\|_{\Lip}<\infty$ and with $\mu(\phi_j)=0$ for at least one $j \in \{0,\dots, k-1\}$, be fixed. For every $\varepsilon >0$, there exist $t_{\varepsilon}, H_{\varepsilon} \geq 1$ so that for all $k$-tuples of times $\tb$ with $\Deltat \geq t_{\varepsilon}$ and all towers $\cT \subseteq X$ with height $\height \geq H_{\varepsilon}$, we have
	\[
	\left\lvert \int_{\cT}\prod_{i=0}^{k-1} \phi_i\circ \flow_{t_i} \diff\mu\right\rvert \leq \varepsilon.
	\]
\end{definition}

Property $\Par(k)$ directly implies mixing of order $k$, as the next lemma shows.

\begin{lemma}\label{lem:Park_implies_k_mixing} 
	If $\flowR$ is mixing of order $k-1$ and satisfies property $\Par(k)$, then it is mixing of order $k$. 
\end{lemma}
\begin{proof}
	It is enough to prove mixing of order $k$ for Lipschitz functions with compact support, so let $\phi_0,\ldots, \phi_{k-1}\in \Lip_c(X)$ be fixed, and let $\varepsilon \in (0,1/2)$. 
	By writing $\phi_{k-1} = \phi_{k-1}^{\perp} + \mu(\phi_{k-1})$, where $\phi_{k-1}^{\perp} = \phi_{k-1} - \mu(\phi_{k-1})$ satisfies $\mu(\phi_{k-1}^{\perp}) = 0$ and $\|\phi_{k-1}^{\perp}\|_{\Lip}\leq 2 \|\phi_{k-1}\|_{\Lip}$, and by using the $(k-1)$-mixing assumption, we can assume that $\mu(\phi_{k-1}) = 0$.
	
	Let $t_{\varepsilon}, H_{\varepsilon} $ be given by property $\Par(k)$. Let $\cR \subseteq X$ be a Rokhlin tower with $\mu(\cR) \geq 1- (\prod_{i=0}^{k-1} \| \phi_i\|_{\infty})^{-1} \varepsilon$ and height $\geq H_{\varepsilon}$. Then, for any $k$-tuple of times $\tb$ with $\Deltat \geq t_{\varepsilon}$, we have 
	\[
	\left\lvert \int_{X}\prod_{i=0}^{k-1} \phi_i\circ \flow_{t_i} \diff\mu\right\rvert \leq \left(\prod_{i=0}^{k-1} \| \phi_i\|_{\infty} \right) \mu(X \setminus \cR) + \left\lvert \int_{\cR}\prod_{i=0}^{k-1} \phi_i\circ \flow_{t_i} \diff\mu\right\rvert \leq 2 \varepsilon,
	\]
	where we used property $\Par(k)$ applied to $\cR$.
\end{proof}

We will prove inductively that, under certain assumptions, $\Par(k)$ implies $\Par(k+1)$. The base of the induction is given by the following proposition.

\begin{proposition}\label{prop:par1} 
	If $\flowR$ is ergodic, then it satisfies $\Par(1)$.
\end{proposition}
\begin{proof} 
	Let $\phi = \phi_{0}\in \Lip(X)$ with $\mu(\phi) =0$ and $\|\phi\|_{\Lip} < \infty$. Without loss of generality, we assume that $\|\phi\|_{\infty} \leq 1$. 
	
	Fix $\varepsilon >0$. Since both the flow $\flowR$ and its inverse are ergodic, by Egorov's Theorem there exists a measurable set $E \subset X$ with $\mu(E)\leq \varepsilon/4$ and there exists $T_0 \geq 1$ so that 
	\[
	\sup_{x \in X \setminus E} \left\lvert \frac{1}{t} \int_{0}^{t} \phi \circ \flow_{\pm r}(x) \diff r \right\rvert \leq \frac{\varepsilon}{4} \qquad \text{for all $t\geq T_0$.}
	\]
	We define $H_{\varepsilon} = 4T_0/ \varepsilon$.
	Let $\cT$ be a tower of height $\height \geq H_{\varepsilon}$ and base $B$; we need to show that 
	\[
	\left\lvert \int_{\cT} \phi \diff\mu\right\rvert = \left\lvert \int_{B} \int_0^{\height} \phi \circ \flow_t (x) \diff t \diff\nu(x)\right\rvert\leq \varepsilon.
	\]
	Let $\nu$ be the measure on the base $B$ of $\cT$ (see \Cref{lem:towers1}) and consider the set $A = \{x \in B : \flow_r(x) \in E \text{ for all } r\in [0,\height]\}$. Since 
	\[
	\nu(A) \cdot \height \leq \mu(E) \leq \frac{\varepsilon}{4},
	\]
	it follows that $\nu(A) \leq \varepsilon / (4\height)$. For all $x \notin A$, there exists $r_x \in [0,\height]$ so that $y_x:= \flow_{r_x}(x) \notin E$. We claim that 
	\begin{equation}\label{eq:to_show_par1}
		\sup_{x \in B \setminus A}\left\lvert \int_0^{\height} \phi \circ \flow_t (x) \diff t \right\rvert\leq \frac{\varepsilon\cdot \height}{2}.
	\end{equation}
	Assuming the claim \eqref{eq:to_show_par1}, since $\nu(B) \leq \height^{-1}$, we conclude that  
	\[
	\left\lvert \int_{\cT} \phi \diff\mu\right\rvert \leq \nu(A) \cdot \height + \nu(B \setminus A) \sup_{x \in B \setminus A} \left\lvert  \int_0^{\height} \phi \circ \flow_t (x) \diff t \right\rvert\leq \varepsilon,
	\]
	which proves the proposition.
	
	Let us now show \eqref{eq:to_show_par1}. Fix $x\in B \setminus A$, and assume first that $r_x \leq T_0$; then
	\[
	\left\lvert \int_0^{\height} \phi \circ \flow_t (x) \diff t \right\rvert\leq r_x + \left\lvert \int_0^{\height - r_x} \phi \circ \flow_t (y_x) \diff t \right\rvert \leq T_0 + (\height - T_0) \frac{\varepsilon}{4} \leq \frac{\varepsilon \cdot \height }{2},
	\]
	where we used the fact that $\height - r_x \geq H_{\varepsilon} - T_0 > T_0$. If $T_0 < r_x \leq \height - T_0$, then
	\[
	\left\lvert \int_0^{\height} \phi \circ \flow_t (x) \diff t \right\rvert\leq \left\lvert \int_0^{r_x} \phi \circ \flow_{-t} (y_x) \diff t \right\rvert + \left\lvert \int_0^{\height - r_x} \phi \circ \flow_t (y_x) \diff t \right\rvert \leq \frac{\varepsilon}{4} (r_x + \height - r_x)\leq \frac{\varepsilon \cdot \height }{4},
	\]
	since both $r_x$ and $\height - r_x$ are larger than $T_0$. The final case $\height - T_0 < r_x \leq \height$ is treated similarly:
	\[
	\left\lvert \int_0^{\height} \phi \circ \flow_t (x) \diff t \right\rvert\leq \left\lvert \int_0^{r_x} \phi \circ \flow_{-t} (y_x) \diff t \right\rvert + \height - r_x \leq \frac{\varepsilon}{4} r_x + T_0 \leq \frac{\varepsilon \cdot \height }{2}.
	\]
	This proves our claim \eqref{eq:to_show_par1} and hence the proposition.
\end{proof}

We will also need the following strengthening of property $\Par(k)$. The proof of \Cref{lem:uniform_park} below is similar to the proof of \Cref{lem:uniform_k_mixing}, we provide it for completeness.

\begin{lemma}[Uniform $\Par(k)$]\label{lem:uniform_park}
	Assume that $\flowR$ satisfies $\Par(k)$. 
	Let $\phi_0,\ldots, \phi_{k-1}\in \Lip(X)$, with $\|\phi_i\|_{\Lip} < \infty$, and assume that $\mu(\phi_j)=0$ for at least one $j \in \{0,\dots, k-1\}$. For every $\varepsilon>0$, there exists $R_{-}\geq 1$ and, for every $R_{+} \geq R_{-}$ and $\delta>0$, there exist $t_{\delta, R_{+}}, H_{\delta, R_{+}} \geq 1$ so that for all $k$-tuples of times $\tb$ with $\Deltat \geq t_{\delta, R_{+}}$ and all towers $\cT \subseteq X$ with height $\height \geq H_{\delta, R_{+}}$, we have
	\[
	\sup_{\substack{0\leq |r_i|\leq R_{+} \\ R_{-} \leq |r_j| \leq R_{+}}}\left\lvert \int_{\cT} \prod_{i=0}^{k-1} (\phi_i \cdot \phi_i \circ \flow_{r_i}) \circ \flow_{t_i} \diff\mu\right\rvert \leq \delta + \varepsilon \mu(\cT).
	\]
\end{lemma}
\begin{proof}
Without loss of generality, let us assume $j=k-1$ (the proof is identical for other choices of $j$). 
Let $\varepsilon >0$ be fixed, and let $R_{-}>0$ be so that $|\mu(\phi_{k-1} \cdot \phi_{k-1} \circ \flow_r)|\leq \varepsilon/(\prod_{i=0}^{k-2} \|\phi_i \|_{\infty})$ for all $|r|\geq R_{-}$. 
Fix now $R_{+} \geq R_{-}$ and $\delta>0$, and let $N \in \N$ be minimal so that  $N \geq \delta^{-1} \, k ( \prod_{i=0}^{k-1} \|\phi_i\|^2_{\Lip})$.
As in \Cref{lem:uniform_k_mixing}, when writing $\psi^{\perp}$, we mean $\psi - \mu(\psi)$.
We use property $\Par(k)$ for $\delta$ and for the finite collection of all possible $k$-tuples of functions 
\[
\{\phi_0 \cdot \phi_0 \circ \flow_{r_0}, \dots, \phi_{k-2} \cdot \phi_{k-2} \circ \flow_{r_{k-2}},(\phi_{k-1} \cdot \phi_{k-1} \circ \flow_{r_{k-1}})^{\perp}\},
\] 
for all choices of $r_i \in N^{-1}\Z$, with $|r_i| \leq R_{+}$: there exist $t_{\delta, R_{+}}, H_{\delta, R_{+}} \geq 1$ (given by the maximum of the respective quantity over all $k$-tuples of functions as above) so that for all $k$-tuples of times $\tb$ with $\Deltat \geq t_{\delta, R_{+}}$, for all towers $\cT$ with height $\geq H_{\delta, R_{+}}$, and for all $|r_i|\leq R_{+}$, $r_i \in N^{-1}\Z$, with $r_{k-1} \geq R_{-}$, we have
\begin{multline}\label{eq:unif_Park1}
			\left\lvert \int_{\cT} \prod_{i=0}^{k-1} (\phi_i \cdot \phi_i \circ \flow_{r_i}) \circ \flow_{t_i} \diff\mu\right\rvert \leq \left\lvert \int_{\cT} \left(\prod_{i=0}^{k-2} (\phi_i \cdot \phi_i \circ \flow_{r_i}) \circ \flow_{t_i}\right) \cdot (\phi_{k-1} \cdot \phi_{k-1} \circ \flow_{r_{k-1}})^{\perp}\circ \flow_{t_{k-1}} \diff\mu\right\rvert \\
		+ \left(\prod_{i=0}^{k-2} \|\phi_i \|_{\infty} \right)|\mu(\phi_{k-1} \cdot \phi_{k-1} \circ \flow_{r_{k-1}})| \cdot \mu(\cT) 
		\leq  \delta + \varepsilon \mu(\cT).
\end{multline}
Let now $|r_i| \leq R_{+}$ be arbitrary, with $r_{k-1} \geq R_{-}$,  and let $m_i/N\in N^{-1}\Z$ so that $|r_i-m_i/N|\leq N^{-1}$. The Lipschitz property implies $\| \phi_i \cdot \phi_i \circ \flow_{r_i} - \phi_i \cdot \phi_i \circ \flow_{m_i/N}\|_{\infty} \leq \|\phi_i\|^2_{\Lip} N^{-1}$; hence, up to an error bounded by
\[
\sum_{i=0}^{k-1} \left(\prod_{j\neq i} \|\phi_j\|^2_{\infty} \right) \|\phi_i\|^2_{\Lip} N^{-1} \leq \delta,
\]
the bound \eqref{eq:unif_Park1} holds for arbitrary $|r_i| \leq R_{+}$, with $r_{k-1} \geq R_{-}$.
\end{proof}

The following is the key result on property $\Par(k)$.

\begin{proposition}\label{prop:park_implies_park+1} 
	Let $\flowR$ be a mixing LUS flow. Then, for every $k\geq 1$, property $\Par(k)$ implies property $\Par(k+1)$. 
\end{proposition}

The proof of \Cref{prop:park_implies_park+1} is contained in the next subsection. As an immediate corollary, we obtain the main abstract result of the paper.

\begin{theorem}\label{thm:main_abstract_result} Let $\flowR$ be a mixing LUS flow. Then $\flowR$ is mixing of all orders.
\end{theorem}
\begin{proof}
	By \Cref{prop:par1}, the flow $\flowR$ satisfies $\Par(1)$; hence, \Cref{prop:park_implies_park+1} implies that $\flowR$ satisfies $\Par(k)$ for all $k \geq 2$ as well. Mixing of all orders follows from \Cref{lem:Park_implies_k_mixing} by induction.
\end{proof}


\subsection{Proof of \Cref{prop:park_implies_park+1}}\label{sec:Proof_induction_park}

This subsection is devoted to the proof of \Cref{prop:park_implies_park+1}: we assume that the LUS flow $\flowR$ is mixing and satisfies property $\Par(k)$; we show it satisfies property $\Par(k+1)$. 
The proof follows the same lines as the one of \Cref{thm:multiple_mixing_GG_flows}, with additional complications due to the weaker assumptions.

Let $a\in (0,1/2]$, $d\geq 1$, and $(M_s)_{s\in \N}$ be given by \Cref{def:LUS_flows}.
We fix $\phi_0,\ldots, \phi_{k}\in \Lip(X)$, with finite Lipschitz norms and with $\mu(\phi_j)=0$ for at least one $j \in \{0,\dots, k\}$, and $\varepsilon \in (0,1)$.
For all $j$'s for which $\mu(\phi_j)\neq 0$, we replace $\phi_j$ with $\phi_j^{\perp} + \mu(\phi_j)$; using property $\Par(k)$, we can then assume that $\mu(\phi_i)=0$ for all $i=0,\dots, k$.
In the following when writing $A \ll B$, we mean $A \leq C \cdot B$ for a constant $C\geq 0$ whose value depends only on $d$, on $k$, and on the Lipschitz norm of the functions $\phi_0, \dots, \phi_k$.

Let $\delta_0 = \varepsilon^{4}$ and 
let $M_{k, \delta_0} \geq \delta_0^{-2}$ be given by the LUS property in \Cref{def:LUS_flows}.
Let $R_{-}\geq 1$ be given by \Cref{lem:uniform_park} applied with $\delta_0$ and for the $k$-tuple of functions $\phi_1, \dots, \phi_k$ (notice the shift in the indexing); we let 
\[
R_{+} = M = M_s \geq \max\{M_{k, \delta}, R_{-}^{1/a}\}, \qquad \text{and} \qquad \delta = \varepsilon^2/(\delta_0^{-1}M)^{k-1}.
\]
By \Cref{lem:uniform_park}, there exist $t_{\delta, R_{+}}$ and $ H_{\delta, R_{+}}$ so that 
\begin{equation}\label{eq:proof_prop_lus_2}
	\sup_{\substack{0 \leq |r_i| \leq M \\ M^a \leq |r_j| \leq M }}\left\lvert \int_{\cT} \prod_{i =1 }^k (\phi_i \cdot \phi_i \circ \flow_{r_i}) \circ \flow_{t_{i-1}} \diff\mu\right\rvert \leq \delta + \varepsilon^4 \mu(\cT),
\end{equation}
for any $j \in \{1,\dots,k\}$, for all $k$-tuples of times $\tb $ with  $\Deltat \geq t_{\delta, R_{+}}$, and for all towers $\cT$ with height $\height \geq  H_{\delta, R_{+}}$.

We set $t_{\varepsilon}$ and $ H_{\varepsilon}$ in \Cref{def:property_park} as follows: 
we let $H_{\varepsilon} = \max \{M_{k, \delta_0}, (H_{\delta, R_{+}}/ \delta^{2k})^{8}\}$, then we consider $t_{H_{\varepsilon},M}$ be given by \Cref{def:LUS_flows}, as well as the Rokhlin towers $\cR_1, \dots, \cR_d$ of bases $B_\varrho$ and heights $\height_{\varrho} \geq H_{\varepsilon}^{1/2}$, and we set $t_{\varepsilon} = \max \{ t_{H_{\varepsilon},M}, t_{\delta, R_{+}} + 2 M\}$.
Let us then fix $\height \geq H_{\varepsilon}$, a $(k+1)$-tuple of times $\tb$ with $\Deltat \geq t_{\varepsilon}$, and a tower $\cT \subseteq X$ with height $\height$; we need to show that\footnote{We recall that the implicit constant only depends on $k$ and on the Lipschitz norm of the functions $\phi_i$.}
\begin{equation}\label{eq:proof_prop_lus_goal}
	\left\lvert \int_{\cT}\prod_{i=0}^{k} \phi_i\circ \flow_{t_i} \diff\mu\right\rvert \ll \varepsilon.
\end{equation}

\Cref{def:LUS_flows} yields the existence of $\delta_0$-almost measure preserving maps $\mapg^{\tb}_m \colon \cup \cR_{\varrho} \to \cup \cR_{\varrho}$, with $m=0,\dots, K_{\delta_0} = \lfloor \delta_0^{-1}\rfloor$, satisfying (LUS1,2,3). 
Let $K = \lfloor K_{\delta} /2 \rfloor$, and note that $ \varepsilon^{-4} \ll K \ll \sqrt{M}$.
By \Cref{lem:change_variab}, we have
\[
\left\lvert \int_{\cT}\prod_{i=0}^{k} \phi_i\circ \flow_{t_i} \diff\mu\right\rvert \ll \left\lvert \frac{1}{K} \sum_{m=1}^K \int_{X} (\one_{\cT} \circ \mapg^{\tb}_m) \, \prod_{i=0}^{k} \phi_i\circ \flow_{t_i} \circ \mapg^{\tb}_m \diff\mu\right\rvert + \delta_0.
\]
Furthermore, by (LUS2), we can replace the domain of integration with the union of the towers $\Tow_{\varrho, j}$ for $j=1,\dots,k$ and $\varrho=1,\dots,d$; in particular
\begin{equation}\label{eq:proof_prop_lus_3}
	\left\lvert \int_{\cT}\prod_{i=0}^{k} \phi_i\circ \flow_{t_i} \diff\mu\right\rvert \ll \max_{\varrho =1, \dots, d} \max_{j=1, \dots, k} \left\lvert \frac{1}{K} \sum_{m=1}^K \int_{\Tow_{\varrho, j}} (\one_{\cT} \circ \mapg^{\tb}_m) \, \prod_{i=0}^{k} \phi_i\circ \flow_{t_i} \circ \mapg^{\tb}_m \diff\mu\right\rvert + \delta_0.
\end{equation}
Let us fix $\varrho \in \{1, \dots, d\}$ and $j \in \{1, \dots, k\}$; we focus on the absolute value in the right hand side above. 
By (LUS1)-(a), since $\phi_0$ is Lipschitz, we have $\|\phi_0\circ \mapg^{\tb}_m-\phi_0\|_{\infty} \leq \|\phi_0\|_{\Lip} \delta_0$; hence, 
\begin{multline*}
	\left\lvert \frac{1}{K} \sum_{m=1}^K \int_{\Tow_{\varrho, j}} (\one_{\cT} \circ \mapg^{\tb}_m) \, \prod_{i=0}^{k} \phi_i\circ \flow_{t_i} \circ \mapg^{\tb}_m \diff\mu\right\rvert \\
	\ll \left\lvert  \int_{X} \phi_0 \cdot \left[\frac{1}{K} \sum_{m=1}^K \one_{\Tow_{\varrho, j} \cap (\mapg^{\tb}_m)^{-1}(\cT)} \, \prod_{i=1}^{k} \phi_i\circ \flow_{t_i} \circ \mapg^{\tb}_m \right] \diff\mu\right\rvert + \delta_0.
\end{multline*}
An application of the Cauchy-Schwarz Inequality and of the Van der Corput Inequality (\Cref{lem:vdc} in \Cref{sec:vdc}) with $L = \lfloor \sqrt{K} \rfloor$ yields
\[
\begin{split}
	&\left\lvert \frac{1}{K} \sum_{m=1}^K \int_{\Tow_{\varrho, j}} (\one_{\cT} \circ \mapg^{\tb}_m) \, \prod_{i=0}^{k} \phi_i\circ \flow_{t_i} \circ \mapg^{\tb}_m \diff\mu\right\rvert \\
	&\ll \left[ \frac{1}{K} \sum_{m=1}^K \left( \frac{1}{L} \sum_{l=0}^L \left\lvert \int_{X} \one_{\Tow_{\varrho, j} \cap (\mapg^{\tb}_m)^{-1}(\cT) \cap (\mapg^{\tb}_{m+l})^{-1}(\cT)} \, \prod_{i=1}^{k} \phi_i\circ \flow_{t_i} \circ \mapg^{\tb}_m \cdot \phi_i\circ \flow_{t_i} \circ \mapg^{\tb}_{m+l} \diff\mu \right\rvert \right) \right]^{1/2} \\
	& \quad +\left(\frac{L}{K}\right)^{1/2} + \varepsilon^4.
\end{split}
\]
We isolate the term corresponding to $l=0$ in the inner sum: since $L^{-1/2} \ll K^{-1/4} \ll \varepsilon$ and $(L/K)^{1/2} \ll K^{-1/4} \ll \varepsilon$, we obtain
\begin{multline}\label{eq:proof_prop_lus_4}
	\left\lvert \frac{1}{K} \sum_{m=1}^K \int_{\Tow_{\varrho, j}} (\one_{\cT} \circ \mapg^{\tb}_m) \, \prod_{i=0}^{k} \phi_i\circ \flow_{t_i} \circ \mapg^{\tb}_m \diff\mu\right\rvert \\
	\ll \sup_{\substack{0 \leq m,l \leq K \\ |m-l| \geq 1} }\left\lvert \int_{X} \one_{\cT(j,m,l)} \, \prod_{i=1}^{k} \phi_i\circ \flow_{t_i} \circ \mapg^{\tb}_m \cdot \phi_i\circ \flow_{t_i} \circ \mapg^{\tb}_{l} \diff\mu \right\rvert^{1/2} + \varepsilon,
\end{multline}
where we denoted $\cT(j,m,l) = \Tow_{\varrho, j} \cap (\mapg^{\tb}_m)^{-1}(\cT) \cap (\mapg^{\tb}_{l})^{-1}(\cT)$. Let us fix $0 \leq m,l \leq K$ with $ |m-l| \geq 1$, and let us consider the absolute value on the right hand side above. 
Property (LUS1)-(b) implies that the sets $(\mapg^{\tb}_{m})^{-1}(\cT)$ and $(\mapg^{\tb}_{l})^{-1}(\cT)$ can be replaced by two towers of height $H_{\varepsilon}^{1/4}$, up to a set of measure $\ll \delta_0 = \varepsilon^4$. Therefore, by \Cref{lem:towers2}, we can replace the set $\cT(j,m,l)$ with a tower $\cT(j)$ of height $H_{\varepsilon}^{1/8}$.

By construction, on the tower $\cT(j)$ we have a good control of the factor in the integral corresponding to the time $t_j$, for which we have a uniform shear by (LUS2). In order to control the factors corresponding to the other times $t_i$, we exploit property (LUS3) and use the towers $R_{\varrho}(i,m,\ell)$.
Recall that, for fixed $i$ and $m$, the family $\{R_{\varrho}(i,m,\ell) : \ell = 0, \dots, \delta_0^{-1}M - 1 \}$ forms a $\delta_0$-almost partition of $\cR_{\varrho}$.
Then, by \Cref{lem:intersection_almost_partitions}, the family $\{R_{\varrho}(i,m,p_i) \cap R_{\varrho}(i,l,q_i): p_i,q_i = 0, \dots, \delta_0^{-1}M - 1 \}$ forms a $2\delta_0$-almost partition of the same set.
Therefore, using again \Cref{lem:intersection_almost_partitions}, the collection of towers
\[
R_{\varrho}(\pb, \qb) := \bigcap_{i \in \{1, \dots, k\} \setminus \{j\}} R_{\varrho}(i,m,p_i) \cap R_{\varrho}(i,l,q_i)
\]
over all $\pb = (p_1, \dots, p_{j-1}, p_{j+1}, \dots, p_k), \qb = (q_1, \dots, q_{j-1}, q_{j+1}, \dots, q_k) \in \Z^{k-1}$ with $0 \leq p_i, q_i \leq \delta_0^{-1}M - 1$ form a $(2^{k-1}\delta_0)$-almost partition of $\cR_{\varrho}$.
In particular,
\begin{multline*}
	\left\lvert \int_{X} \one_{\cT(j)} \, \prod_{i=1}^{k} \phi_i\circ \flow_{t_i} \circ \mapg^{\tb}_m \cdot \phi_i\circ \flow_{t_i} \circ \mapg^{\tb}_{l} \diff\mu \right\rvert \\
	\ll \left\lvert \sum_{\pb, \qb} \int_{X} \one_{\cT(j) \cap R_{\varrho}(\pb, \qb)} \, \prod_{i=1}^{k} \phi_i\circ \flow_{t_i} \circ \mapg^{\tb}_m \cdot \phi_i\circ \flow_{t_i} \circ \mapg^{\tb}_{l} \diff\mu \right\rvert + \delta_0.
\end{multline*}
For the term corresponding to the index $j$ in the product above, we have
\[
\|\phi_j\circ \flow_{t_j} \circ \mapg^{\tb}_m \cdot \phi_j\circ \flow_{t_j} \circ \mapg^{\tb}_{l} - \phi_j\circ \flow_{t_j+mM^a} \cdot \phi_j\circ \flow_{t_j+lM^a} \|_{\infty} \ll \delta_0 = \varepsilon^4,
\]
using twice (LUS2). For the other terms, we use (LUS3) instead and we get 
\[
\|\phi_i\circ \flow_{t_i} \circ \mapg^{\tb}_m \cdot \phi_i\circ \flow_{t_i} \circ \mapg^{\tb}_{l} - \phi_i\circ \flow_{t_i+\delta p_i} \cdot \phi_i\circ \flow_{t_i+\delta q_i} \|_{\infty} \ll \delta_0= \varepsilon^4.
\]
Again, since the towers $R_{\varrho}(\pb, \qb)$ form an almost partition of $\cR_{\varrho}$, we deduce
\[
\begin{split}
	&	\left\lvert \int_{X} \one_{\cT(j)} \, \prod_{i=1}^{k} \phi_i\circ \flow_{t_i} \circ \mapg^{\tb}_m \cdot \phi_i\circ \flow_{t_i} \circ \mapg^{\tb}_{l} \diff\mu \right\rvert \\
	&	\ll \sum_{\pb, \qb} \left\lvert \int_{\cT(j) \cap R_{\varrho}(\pb, \qb)} \, \left( \prod_{i \in \{1, \dots, k\} \setminus \{j\}} (\phi_i \cdot \phi_i\circ \flow_{\delta (q_i-p_i)} )\circ \flow_{t_i + \delta p_i} \right) (\phi_j \cdot \phi_j\circ \flow_{(l-m)M^a} )\circ \flow_{t_j + mM^a}\diff\mu \right\rvert \\
	& \quad + \varepsilon^4.
\end{split}
\]
By invariance of $\mu$, the integrand function above can be composed by $\flow_{-\overline t}$, where ${\overline t} = t_1+\delta p_1$, or ${\overline t} = t_1+mM^a$ in case $j=1$; then, the integral is taken over the set 
\[
\flow_{-\overline t}(\cT(j) \cap R(\pb, \qb)) = \flow_{-\overline t}(\cT(j)) \cap \flow_{-\overline t}(R(\pb, \qb)).
\] 
We note that the image of a tower under the flow is still a tower of the same height, so that the set above is an intersection of $k$ towers of heights at least $H_{\varepsilon}^{1/8}\geq H_{\delta, R_{+}} / \delta^{2k}$. By \Cref{lem:towers2}, there exists a (possibly empty) tower $\widetilde{\cT} \subseteq \flow_{-\overline t}(\cT(j) \cap R(\pb, \qb))$ of height $H_{\delta, R_{+}}$ so that 
\[
\mu(\flow_{-\overline t}(\cT(j) \cap R(\pb, \qb)) \setminus \widetilde{\cT} ) \ll \delta^2 = \varepsilon^4  (\delta_0^{-1}M)^{-2(k-1)}.
\]
Since  there are $\delta_0^{-1}M$ possible values for each $p_i$ and $q_i$, we obtain
\begin{multline*}
	\left\lvert \int_{X} \one_{\cT(j)} \, \prod_{i=1}^{k} \phi_i\circ \flow_{t_i} \circ \mapg^{\tb}_m \cdot \phi_i\circ \flow_{t_i} \circ \mapg^{\tb}_{l} \diff\mu \right\rvert \ll \\
	\sum_{\pb, \qb} \left\lvert \int_{\widetilde{\cT}} \, \left( \prod_{i \neq j} (\phi_i \cdot \phi_i\circ \flow_{\delta (q_i-p_i)} )\circ \flow_{t_i'} \right) (\phi_j \cdot \phi_j\circ \flow_{(l-m)M^a} )\circ \flow_{t_j'}\diff\mu \right\rvert 	 + \varepsilon^4,
\end{multline*}
where the $k$-tuple of  times $\tb'$ has components 
$t_i'=t_i + \delta p_i {-\overline t}$ for $i\neq j$ and
$t_j' = t_j + (l-m)M^a {-\overline t}$.
Now, we claim that the right hand side above is bounded by $\ll \varepsilon^4$. Indeed, we can use \eqref{eq:proof_prop_lus_2}: we have $|\delta (q_i-p_i)| \leq M $, and $|(l-m)M^a| \leq K_{\delta}\sqrt{M}  \leq M$, and 
\[
\begin{split}
	&|(t_i + \delta p_i {-\overline t}) - (t_{i'} + \delta p_{i'} {-\overline t})| \geq |t_i -t_{i'}| - \delta |p_i - p_{i'}| \geq t_{\varepsilon} - M \geq t_{\delta,R_{+}}, \\
	&|(t_i + \delta p_i {-\overline t}) - (t_j + mM^a{-\overline t})| \geq |t_i -t_j| - \delta |p_i| - mM^a \geq t_{\varepsilon} - M - K_{\delta} \sqrt{M}  \geq t_{\delta,R_{+}}.
\end{split}
\]
Therefore, the uniform property $\Par(k)$ in \eqref{eq:proof_prop_lus_2} yields
\[
\left\lvert \int_{X} \one_{\cT(j)} \, \prod_{i=1}^{k} \phi_i\circ \flow_{t_i} \circ \mapg^{\tb}_m \cdot \phi_i\circ \flow_{t_i} \circ \mapg^{\tb}_{l} \diff\mu \right\rvert \ll \sum_{\pb, \qb}\big(\delta + \varepsilon^4 \mu(\widetilde{\cT}) \big)  \ll (\delta_0^{-1}M)^{2(k-1)} \delta+ \varepsilon^4 \ll \varepsilon^2,
\]
where we used the fact that
\[
\sum_{\pb, \qb}\mu(\widetilde{\cT}) \leq \sum_{\pb, \qb} \mu(R(\pb, \qb)) = \mu\left(\bigcup_{\pb, \qb}R(\pb, \qb)\right) \ll 1.
\]  
Combining this bound with \eqref{eq:proof_prop_lus_4} and \eqref{eq:proof_prop_lus_3} proves \eqref{eq:proof_prop_lus_goal} and hence establishes property $\Par(k+1)$.
The proof of \Cref{prop:park_implies_park+1} is therefore complete.


\section{Quantitative results}\label{sec:quantitative_results}

In this section we discuss the quantitative analogues of the results we proved in the previous sections. Here, the assumptions are the same as in \Cref{sec:QGUS}; in particular, $X$ is a smooth manifold and $\flowR$ a smooth flow. Analogously to \Cref{def:LUS_flows}, quantitatively locally uniformly shearing flows are defined as follows.

\begin{definition}[Quantitatively locally uniformly shearing flows] \label{def:QLUS_flows}
	The flow $\flowR$ is called {\em quantitatively locally uniformly shearing} (QLUS for short) if there exist $a \in (0,1/2]$, $d \geq 1$, and, for every $k \in \Z_{\geq 0}$, there exists $0 < \gamma_k\leq 1/4$ such that for every $1\leq M \leq \height$ there exists a  $M^{-\gamma_k}$-almost partition $\{\cR_{\varrho}\}_{\varrho =1}^d$ of $X$ into towers $\cR_{\varrho}$ of heights $\height_{\varrho} \in [\height^{1/2},\height^{3/2}]$ such that, for every $(k+1)$-tuple of times $\tb$ with $\Deltat^{\gamma_k} \geq 4\height$, the following conditions hold:
	\begin{enumerate}
		\item[(QLUS1)] there exist $M^{-\gamma_k}$-almost measure preserving maps $\mapg^{\tb}_m \colon \cup_{\varrho=1}^d \cR_{\varrho} \to \cup_{\varrho=1}^d \cR_{\varrho}$, for $m=0,\dots, K_{M} = \lfloor M^{\gamma_k}\rfloor$, with the following properties
		\begin{enumerate}
			\item for all $x$ 
			in the proper domain of $\mapg^{\tb}_m$, we have $\dist(\mapg^{\tb}_m(x),x)\leq M^{-\gamma_k}$,
			\item for any tower $\cT$ of height $H \geq \height^{1/2}$, there exists a tower $\overline{\cT} \subseteq (\mapg^{\tb}_m)^{-1}(\cT)$ of height $\height^{1/4}$ with
			\[
			\mu \left( (\mapg^{\tb}_m)^{-1}(\cT) \setminus \overline{\cT} \right) \leq d \,M^{-\gamma_k};
			\]
		\end{enumerate} 
		\item[(QLUS2)] for every $\varrho= 1,\dots, d$, there exists a $M^{-\gamma_k}$-almost partition $\{\Tow_{\varrho, j}\}_{j=1}^k$ of $\cR_{\varrho}$ into towers $\Tow_{\varrho, j} \subseteq \cR_{\varrho}$ of the same height\footnote{In particular, this means that the basis of each towers $\Tow_{\varrho, j}$ is a subset of the basis of $\cR_{\varrho}$.} $\height_{\varrho}$, so that, for every $m=0,\dots,K_{M}$, we have 
		\[
		\max_{j=1, \dots, k} \sup_{x\in \Tow_{\varrho,j}} \dist \left(\flow_{t_j}(\mapg^{\tb}_m x), \flow_{t_j+m \cdot M^a }(x) \right)\leq M^{-\gamma_k};
		\]
		\item[(QLUS3)] for every $\varrho=1, \dots, d$, there exists a $M^{-\gamma_k}$-almost partition of $\cR_{\varrho}$ into (possibly empty) towers $R_{\varrho}(i,m,\ell) \subseteq \cR_{\varrho}$, with $\ell = 0, \dots, M^{1+\gamma_k}-1$, of the same height\footnote{As in (QLUS2), this implies that the basis of each tower $R_{\varrho}(i,m,\ell)$ is a subset of the basis of $\cR_{\varrho}$.} $\height_{\varrho}$ such that, for any $i=1, \dots, k$, $m=0,\dots,K_{M}$, and $\ell$, we have
		\[
		\sup_{x\in R_\varrho(i,m,\ell)} \dist \left(\flow_{t_i}(\mapg^{\tb}_m x), \flow_{t_i+ \ell \cdot M^{-\gamma_k}}(x) \right)\leq M^{-\gamma_k}.
		\]
	\end{enumerate}
\end{definition}

\subsection{Property $\QPar(k)$}

We now formulate the quantitative version of property $\Par(k)$. As before, if $\tb = (0)$ is the trivial 1-tuple, then we set $\Deltat = \infty$.

\begin{definition}\label{def:quantitative_park}
	Let $k\geq 1$. A flow $\flowR$ on $(X,\mu)$ has {\em property $\QPar(k)$} if there exist $\delta_k>0$ and $D_k\geq d_k$ for which the following holds. Let $\phi_0,\ldots, \phi_{k-1}\in \mathscr{C}^{\infty}_c(X)$ with $\mu(\phi_i)=0$ for at least one $i \in \{0,\dots, k-1\}$. 
	For every $k$-tuples of times $\tb$ with $\Deltat \geq 1$ and for every tower $\cT \subseteq X$ with height $\height \geq 1$, we have
	\[
	\left\lvert \int_{\cT}\prod_{i=0}^{k-1} \phi_i\circ \flow_{t_i} \diff\mu\right\rvert \leq \left(\prod_{i=0}^{k-1} \cN_{D_k}(\phi_i)\right) \cdot
	\min \{ \Deltat, \height\}^{-\delta_k}.
	\]
\end{definition}
A result analogous to \Cref{lem:Park_implies_k_mixing} holds: Property $\QPar(k)$ implies polynomial mixing of order $k$.

\begin{lemma}\label{lem:QPark_implies_k_mixing} 
	If $\flowR$ is $(k-1)$-mixing with rate $\eta_{k-1}$ and satisfies property $\QPar(k)$, then it is $k$-mixing with rate $\min\{\eta_{k-1}, \delta_k\eta_2/(16\sigma({D_k})) \}$ and $d_k = D_k$. 
\end{lemma}
\begin{proof}
	Let $\phi_0,\ldots, \phi_{k-1}\in \mathscr{C}^{\infty}_c(X)$ and a $k$-tuple of times $\tb$ be fixed. By \Cref{lem:phi_phi_orthogonal} applied to $\phi_{0}$ with $\varepsilon = \Deltat^{-\delta_k\eta_2/(16\sigma({D_k}))}$, and by the $(k-1)$-mixing assumption, we have
	\begin{equation*}
		\begin{split}
			&\left\lvert \int_{X} \prod_{i=0}^{k-1}\phi_i \circ \flow_{t_i} \diff \mu - \prod_{i=0}^{k-1} \mu(\phi_i) \right\rvert \\ & \qquad 
			\leq  |\mu(\phi_0)| \cdot \left\lvert \int_{X} \prod_{i=1}^{k-1}\phi_i \circ \flow_{t_i} \diff \mu - \prod_{i=1}^{k-1} \mu(\phi_i) \right\rvert 
			+\varepsilon \left(\prod_{i=1}^{k-1} \|\phi_i\|_{\infty}\right)
			+
			\left\lvert \int_{X} \phi_0^{\perp} \cdot \prod_{i=1}^{k-1}\phi_i \circ \flow_{t_i} \diff \mu \right\rvert \\
			 & \qquad \leq 2 \left(\prod_{i=0}^{k-1} \cN_{d_{k-1}}(\phi_i)\right)
			(\Deltat^{-\eta_{k-1}} + \Deltat^{-\delta_k\eta_2/(16\sigma({D_k}))}) + \left\lvert \int_{X} \phi_0^{\perp} \cdot \prod_{i=1}^{k-1}\phi_i \circ \flow_{t_i} \diff \mu \right\rvert.
		\end{split}
	\end{equation*}
	To bound the last summand above, let $\cR \subseteq X$ be a Rokhlin tower with $\mu(X \setminus \cR) \leq \Deltat^{-\delta_k}$ and height $\height \geq \Deltat$; then
	\begin{multline*}
			\left\lvert \int_{X} \phi_0^{\perp} \cdot \prod_{i=1}^{k-1}\phi_i \circ \flow_{t_i} \diff \mu \right\rvert \leq 2\left(\prod_{i=0}^{k-1} \| \phi_i\|_{\infty} \right)\Deltat^{-\delta_k} + \left\lvert \int_{\cR}\phi_0^{\perp} \cdot\prod_{i=1}^{k-1} \phi_i\circ \flow_{t_i} \diff\mu\right\rvert \\
			\leq 4 \left(\prod_{i=0}^{k-1} \cN_{D_k}(\phi_i)\right) \varepsilon^{-8\sigma({D_k})/\eta_2} \min \{ \Deltat, \height\}^{-\delta_k} \leq 4 \left(\prod_{i=0}^{k-1} \cN_{D_k}(\phi_i)\right) \Deltat^{-\delta_k/2},
	\end{multline*}
	which proves the result.
\end{proof}

To start our induction, we need to establish property $\QPar(1)$. A sufficient condition is given by the following result

\begin{proposition}\label{prop:QPar1} 
	Assume that the flow $\flowR$ is polynomially mixing with polynomial rate $\eta_2 \in (0,1)$.
	Then $\flowR$ satisfies $\QPar(1)$ with $\delta_1 = \eta_2/(8(1+\eta_2))$ and $D_1=d_2$.
\end{proposition}
\begin{proof} 
	Let us define $ \eta = \frac{\eta_2}{4(1+\eta_2)} $. We first make the following claim.
	
	\textbf{Claim.} 
	For every $\phi \in \mathscr{C}_c^{\infty}(X)$ with $\mu(\phi)=0$ and every $T\geq 1$, there exists a measurable subset $E(\phi,T) \subseteq X$ of measure $\mu(E(\phi,T))\leq T^{-\eta}$ for which we have
	\[
	\left\lvert \frac{1}{t} \int_0^t \phi \circ \flow_{\pm r}(x) \diff r \right\rvert \leq \mathcal{N}_{d_2}(\phi) \, t^{-\eta}, \qquad \text{for all $t\geq T$ and $x \in X \setminus E(\phi,T)$.} 
	\]
	Assuming the Claim, let us prove $\QPar(1)$.
	Let $\cT$ be a tower of base $B$ and height $\height$, and let $T=\sqrt{\height}$. Fix $\phi \in \mathscr{C}_c^\infty(X)$ so that $\mu(\phi)=0$; we have to show that 
	\[
	\left\lvert \int_\cT \phi \diff \mu \right\rvert \leq \cN_{d_2}(\phi) \height^{-\eta/2}.
	\]	
	We can assume that $\mu(\cT) \geq \height^{-\eta/2}$, otherwise the conclusion is clear, and note that we have $\height^{-1-\eta/2} \leq \nu(B) \leq \height^{-1}$. Let us consider the set $A = \{x \in B : \flow_r(x) \in E \text{ for all } r\in [0,\height]\}$: we have
	\[
	\nu(A) \height \leq \mu(E) \leq T^{-\eta} = \height^{-\eta/2},
	\]
	which implies $\nu(A) \leq \height^{-1-\eta/2}$. For all $x\in B \setminus A$ there exists $y_x = \flow_{r_x}(x) \notin E$, for some $r_x \in [0,\height]$. We claim that this implies that 
	\[
	\left\lvert \int_0^{\height} \phi \circ \flow_r(x) \diff r \right\rvert \leq \cN_{d_2}(\phi) \height^{1-\eta/2}.
	\]
	Indeed, if $r_x \leq T = \height^{1/2}$, then
	\[
	\left\lvert \int_0^{\height} \phi \circ \flow_r(x) \diff r \right\rvert \leq \|\phi\|_{\infty} \height^{1/2} + \left\lvert \int_0^{\height-r_x} \phi \circ \flow_r(y_x) \diff r \right\rvert \ll \, \cN_{d_2}(\phi) (\height^{1/2} + \height^{1-\eta/2} ),
	\]
	which proves the claim in this case; 
	the cases $\height^{1/2}< r_x \leq \height - \height^{1/2}$ and $\height - \height^{1/2} < r_x \leq \height$ are treated similarly as it was done in \Cref{prop:par1}.
	Therefore,
	\[
	\begin{split}
		\left\lvert \int_\cT \phi \diff \mu \right\rvert & \leq \int_A \left\lvert \int_0^{\height} \phi \circ \flow_r(x) \diff r \right\rvert \diff \nu(x) + \int_{B \setminus A} \left\lvert \int_0^{\height} \phi \circ \flow_r(x) \diff r \right\rvert \diff \nu(x) \\
		&\leq \|\phi\|_{\infty} \nu(A) \height + \nu(B) \sup_{x \in B \setminus A}\left\lvert \int_0^{\height} \phi \circ \flow_r(x) \diff r \right\rvert \ll \, \cN_{d_2}(\phi) \height^{-\eta/2},
	\end{split}
	\]
	which proves the result.
	
	\textbf{Proof of the Claim.} Let us now prove the Claim. 	The proof is essentially the same as \cite[Proposition 3.1]{Rav}, we reproduce it here for completeness.

	The first step is to obtain a $L^2$-bound for the ergodic integrals of $\phi$ for any $t\geq 1$ as follows. By invariance of $\mu$ under $\flowR$, and using Fubini's Theorem, we have
	\[
	\begin{split}
		\left\|\int_0^t \phi\circ \flow_r \diff r \right\|_2^2 &= \int_X \int_0^t \int_0^t \phi \circ \flow_s \cdot \phi \circ \flow_r \diff r \diff s \diff \mu = \int_{[0,t]^2} \int_X \phi \circ \flow_{s-r} \cdot \phi \diff \mu \diff s \diff r.
	\end{split}
	\]
	Let $A = \{(s,r) \in [0,t]^2 : |s-r| \leq t^{1/(1+\eta_2)} \}$, and note that $\Leb(A) \leq 4 t^{1+1/(1+\eta_2)}$. Then, we have 
	\begin{equation}\label{eq:poly_averages}
		\begin{split}
			\left\|\int_0^t \phi\circ \flow_r \diff r \right\|_2^2 &\leq \|\phi\|_2^2 \Leb(A) + \int_{[0,t]^2 \setminus A} \int_X \phi \circ \flow_{s-r} \cdot \phi  \diff \mu \diff s \diff r\\
			& \leq 4 C_{d_2} \norm_{d_2}(\phi)^2 t^{1+1/(1+\eta_2)} + t^2 \sup_{(s,r) \in A} \left\lvert \int_X \phi \circ \flow_{s-r} \cdot \phi \diff \mu  \right\rvert \\
			& \leq (4 C_{d_2}  + 1) \norm_{d_2}(\phi)^2 t^{2-\eta_2/(1+\eta_2)}.
		\end{split}
	\end{equation}
	
	Define $C'=(4 C_D  + 1)^{1/2}$ and $\eta = \frac{\eta_2}{4(1+\eta_2)} <1$. 
	Let now $T \geq 1$ be fixed, and let $N = \lfloor T^{\eta} \rfloor +2$. For every $n \in \N$, further define $k_n = n^{1/\eta}$ and 
	\[
	E_n := \left\{ x \in X : \left\lvert \int_0^{k_n} \phi \circ \flow_r(x) \diff r\right\rvert \geq C'\norm_{d_2}(\phi) k_n^{1-\eta} \right\}, \qquad \text{and}\qquad E(\phi, T) := \bigcup_{n \geq N} E_n.
	\] 
	By Chebyshev's Inequality and \eqref{eq:poly_averages} we can bound
	\[
	\mu(E(\phi, T)) \leq \sum_{n=N}^\infty \mu(E_n) \leq \sum_{n=N}^\infty k_n^{-2\eta} \leq (N-1)^{-1} \leq T^{-\eta}.
	\]
	Consider now $t\geq T$ and let us prove that $|\int_0^t\phi\circ\flow_r(x) \diff r| \leq C_0  \norm_{d_2}(f) t^{1-\eta}$ for all $x\in X\setminus E(\phi,T)$ for some constant $C_0>0$. Let $n(t) =  \lfloor t^{\eta} \rfloor +2 \geq N$, so that $k_{n(t)} = n(t)^{1/\eta} > t$. Moreover, note that there exists a constant $C_{\eta}>0$ depending on $\eta$ only such that 
	\[
	k_{n(t)} - t \leq (t^{\eta}+2)^{1/\eta} - t \leq C_{\eta} t^{1-\eta}.
	\]
	Therefore, for all $x \in X$, we have
	\[
	\left\lvert \int_0^{t} \phi \circ \flow_r(x) \diff r\right\rvert \leq \left\lvert \int_0^{k_{n(t)}} \phi \circ \flow_r(x) \diff r\right\rvert + C_\eta \|\phi\|_\infty t^{1-\eta}.
	\]
	If moreover $x \notin E(\phi, T)$, then $x \notin E_{n(t)}$, so that we can also bound the first summand above and get 
	\[
		\left\lvert \int_0^{t} \phi \circ \flow_r(x) \diff r\right\rvert  \leq  C' \norm_{d_2}(\phi) k_{n(t)}^{1-\eta'} + C_{\eta}  C_{d_2} \norm_{d_2}(\phi) t^{1-\eta} \leq C_0\norm_{d_2}(\phi) t^{1-\eta},
	\]
	where $C_0 = 4(C_{\eta}+1)C$.
	This proves the Claim.
\end{proof}

We will also need the following lemma. We set, inductively, $D_1=d_2$ and $D_{k+1} = D_k +D$, and $\sigma_k = k \cdot \sigma(D_{k+1})$.
\begin{lemma}[Uniform $\QPar(k)$]\label{lem:uniform_qpark}
	Assume that $\flowR$ satisfies $\QPar(k)$. 
	Let $\phi_0,\ldots, \phi_{k-1}\in \mathscr{C}_c^{\infty}(X)$, and assume that $\mu(\phi_j)=0$ for at least one $j \in \{0,\dots, k-1\}$. 
	For every $1 \leq R_{-} \leq R_{+}$,
	for every $k$-tuple of times $\tb$ with $\Deltat \geq 1$ and for every tower $\cT \subseteq X$ with height $\height$, we have
	\[
	\sup_{\substack{0\leq |r_i|\leq R_{+} \\ R_{-} \leq |r_j| \leq R_{+}}}\left\lvert \int_{\cT} \prod_{i=0}^{k-1} (\phi_i \cdot \phi_i \circ \flow_{r_i}) \circ \flow_{t_i} \diff\mu\right\rvert \leq \left(\prod_{i=0}^{k-1} \cN_{D_{k+1}}(\phi_i)^2\right) \left(R_{-}^{-\eta_2}\mu(\cT) + R_{+}^{\sigma_k} \min \{\height, \Deltat\}^{-\frac{\eta_2\delta_k}{16\sigma_k}}\right).
	\]
\end{lemma}
\begin{proof}
	Let us assume $j=k-1$, the proof being identical in the other cases. Let $\tb$ be a $k$-tuple of times, let $\cT$ be a tower of height $\height$, and let $|r_i|\leq R_{+}$, with $|r_{k-1}|\geq R_{-}$ be fixed. We fix $\varepsilon = \min \{\height, \Deltat\}^{-\eta_2\delta_k/(16\sigma_k)}$ and we apply \Cref{lem:phi_phi_orthogonal} to $\phi_{k-1} \cdot \phi_{k-1} \circ \flow_{r_{k-1}}$;
	\begin{multline*}
		\left\lvert \int_{\cT} \prod_{i = 0}^{ k-1} (\phi_i \cdot \phi_i \circ \flow_{r_i}) \circ \flow_{t_i} \diff\mu\right\rvert \leq \left\lvert \int_{\cT} \left(\prod_{i = 0}^{ k-2} (\phi_i \cdot \phi_i \circ \flow_{r_i}) \circ \flow_{t_i}\right) \cdot (\phi_{k-1} \cdot \phi_{k-1} \circ \flow_{r_{k-1}})^{\perp} \circ \flow_{r_{k-1}} \diff\mu\right\rvert \\
		+ \left(\prod_{i=0}^{k-2} \|\phi_i\|^2_{\infty}\right) \left(|\mu(\phi_{k-1} \cdot \phi_{k-1} \circ \flow_{r_{k-1}})| \cdot \mu(\cT) + \cN_{d_2}(\phi_{k-1} \cdot \phi_{k-1} \circ \flow_{r_{k-1}}) \varepsilon \right).
	\end{multline*}
	By polynomial $2$-mixing and the properties of the Sobolev norms, the second summand above is bounded by 
	\begin{multline*}
		\left(\prod_{i=0}^{k-2} \|\phi_i\|^2_{\infty}\right) \left(|\mu(\phi_{k-1} \cdot \phi_{k-1} \circ \flow_{r_{k-1}})| \mu(\cT) + \cN_{d_2}(\phi_{k-1} \cdot \phi_{k-1} \circ \flow_{r_{k-1}}) \varepsilon \right) \\
		\ll  \left(\prod_{i=0}^{k-1} \cN_{d_2+D}(\phi_i)^2\right) \left(R_{-}^{-\eta_2} \mu(\cT) + R_{+}^{\sigma(d_2+D)}\varepsilon\right),
	\end{multline*}
	whereas, by property $\QPar(k)$ and by \Cref{lem:phi_phi_orthogonal}, the first is bounded by 
	\begin{multline*}
		\left\lvert \int_{\cT} \left(\prod_{i = 0}^{ k-2} (\phi_i \cdot \phi_i \circ \flow_{r_i}) \circ \flow_{t_i}\right) \cdot (\phi_{k-1} \cdot \phi_{k-1} \circ \flow_{r_{k-1}})^{\perp} \circ \flow_{r_{k-1}} \diff\mu\right\rvert \\
		\leq \left(\prod_{i=0}^{k-1}  \cN_{D_{k}}(\phi_i \cdot \phi_i \circ \flow_{r_i}) \right) \cdot \varepsilon^{-8\sigma(D_k)/\eta_2} (\min \{\height, \Deltat\})^{-\delta_k} \\ \ll  \left(\prod_{i=0}^{k-1} \cN_{D_k+D}(\phi_i)^2\right) R_{+}^{k\sigma(D_k+D)}(\min \{\height, \Deltat\})^{-\delta_k/2}. 
	\end{multline*}
	Combining the two estimates completes the proof.
\end{proof}

We consider a flow which is polynomially mixing with rate $\eta_2$ and QLUS with rates $\gamma_k$. 
The key result on property $\QPar(k)$ is the following quantitative analogue of \Cref{prop:park_implies_park+1}.

\begin{proposition}\label{prop:qpark_implies_qpark+1} 
	 Assume that the flow $\flowR$ is polynomially mixing with rate $\eta_2$ and QLUS with rates $\gamma_k$. Then, for every $k\geq 1$, property $\QPar(k)$ implies property $\QPar(k+1)$.
\end{proposition}

The proof of \Cref{prop:qpark_implies_qpark+1} is contained in the next subsection. The main abstract result of the paper for polynomial mixing is an immediate corollary. 

\begin{theorem}\label{thm:main_abstract_result_quantitative} 
	Let $(\flow_t, X, \mu)$ be a polynomially mixing flow. If $\flowR$ is QLUS, then it is mixing of all orders with polynomial rates, with $d_{k+1} = d_k + D$. 
\end{theorem}
\begin{proof}
	By \Cref{prop:QPar1}, the flow $\flowR$ satisfies property $\QPar(1)$ with a polynomial rate $\delta_1$. Hence, \Cref{prop:qpark_implies_qpark+1} implies that $\flowR$ satisfies $\QPar(k)$ for all $k \geq 2$ as well. Polynomial mixing of all orders follows from \Cref{lem:QPark_implies_k_mixing} by induction.
\end{proof}

\subsection{Proof of \Cref{prop:qpark_implies_qpark+1} }

The proof of \Cref{prop:qpark_implies_qpark+1} is entirely analogous to the one of \Cref{prop:park_implies_park+1}. 
Let us assume that $\flowR$ is QLUS with rates $\gamma_k$ and satisfies $\QPar(k)$ with rates $\delta_k$; let us show it satisfies $\QPar(k+1)$ with some explicit rate $\delta_{k+1} >0$.

We fix $\phi_0,\ldots, \phi_{k}\in \mathscr{C}^{\infty}_c(X)$, with $\mu(\phi_i)=0$ for at least one $i \in \{0,\dots, k\}$,
a $(k+1)$-tuples of times $\tb$ with $\Deltat \geq 1$, and a tower $\cT$ with height $\height \geq 1$.
Let $a\in (0,1/2]$ be given by the QLUS property, and we fix
\[
\beta_k := \frac{a\gamma_k\eta_2^2}{2^{5+k}\sigma_{k}}, \qquad 
H:= \frac{1}{4}\min \{\Deltat^{\gamma_k}, \height \}, \qquad 
 \text{and} \qquad M := 
 H^{\frac{\eta_2\delta_k \gamma_k}{2^{12}k\sigma_k^2}}.
\] 
Using \Cref{lem:phi_phi_orthogonal} with $\varepsilon = M^{-\beta_k}$ for all $j$'s for which $\mu(\phi_j)\neq 0$ and property $\QPar(k)$, we can replace $\phi_i$ with $\phi_i^{\perp}$ for all $i \in \{0,\dots, k\}$.
Thus, it suffices to show 
\begin{equation}\label{eq:proof_prop_qlus_goal}
	\left\lvert \int_{\cT}\prod_{i=0}^{k} \phi_i^{\perp}\circ \flow_{t_i} \diff\mu\right\rvert \ll \left(\prod_{i=0}^{k} \cN_{D_{k+1}}(\phi_i)\right) \left(M^{-\beta_k} + 
	H^{-\frac{\eta_2\delta_k}{2^8\sigma_k}}\right).
\end{equation}

We apply the QLUS property with $M\leq H$ chosen above: there exist $K_M := \lfloor M^{\gamma_k}  \rfloor$ maps $\mapg^{\tb}_m \colon \cup_{\varrho=1}^d \cR_{\varrho} \to \cup_{\varrho=1}^d \cR_{\varrho}$ which are $M^{-\gamma_k} $-almost measure preserving and satisfy (a) and (b) of (QLUS1).
Let $K=\lfloor K_M/2\rfloor$. In the same way as in \S\ref{sec:Proof_induction_park}, using the towers $\Tow_{\varrho,j }$ in (QLUS2), we obtain
\begin{multline*}
	\left\lvert \int_{\cT}\prod_{i=0}^{k} \phi_i^{\perp}\circ \flow_{t_i} \diff\mu\right\rvert  \ll \max_{\substack{\varrho = 1, \dots, d \\ j=1, \dots, k}} \left\lvert  \int_{X} \phi_0^{\perp} \cdot \left[\frac{1}{K} \sum_{m=1}^K \one_{\Tow_{\varrho, j} \cap (\mapg^{\tb}_m)^{-1}(\cT)} \, \prod_{i=1}^{k} \phi_i^{\perp}\circ \flow_{t_i} \circ \mapg^{\tb}_m \right] \diff\mu\right\rvert \\ + \left(\prod_{i=0}^{k} \cN_{D_1}(\phi_i)\right) M^{-\gamma_k/2}, 
\end{multline*}
where the towers $\Tow_{\varrho,j} \subset \cR_{\varrho}$ have height $H^{1/2}$. Here, we used that $\|\phi_0^{\perp}\|_{\mathscr{C}^1} \leq \cN_{d_2}(\phi_0^{\perp}) \leq M^{8\beta_k\sigma(D_1)/\eta_2} \cN_{d_2}(\phi_0)\ll M^{\gamma_k/4} \cN_{d_2}(\phi_0)$.

We apply the Cauchy-Schwarz Inequality and the Van der Corput Inequality (\Cref{lem:vdc} in \Cref{sec:vdc}) with $L = \lfloor \sqrt{K} \rfloor$, which, similarly as in \S\ref{sec:Proof_induction_park}, yield
\begin{multline*}
	\left\lvert \int_{\cT}\prod_{i=0}^{k} \phi_i^{\perp}\circ \flow_{t_i} \diff\mu\right\rvert \ll \|\phi_0\|_{\infty} \max_{\substack{\varrho = 1,\dots,d \\ j=1, \dots, k}} \sup_{\substack{0 \leq m,l \leq K \\ |m-l| \geq 1} }\left\lvert \int_{X} \one_{\cT(j,m,l)} \, \prod_{i=1}^{k} \phi_i^{\perp}\circ \flow_{t_i} \circ \mapg^{\tb}_m \cdot \phi_i^{\perp}\circ \flow_{t_i} \circ \mapg^{\tb}_{l} \diff\mu \right\rvert^{1/2} \\ + \left(\prod_{i=0}^{k} \cN_{D_1}(\phi_i)\right) M^{-\gamma_k/4},
\end{multline*}
where we denoted $\cT(j,m,l) = \Tow_{\varrho, j} \cap (\mapg^{\tb}_m)^{-1}(\cT) \cap (\mapg^{\tb}_{l})^{-1}(\cT)$.
Using (QLUS1)-(b) and \Cref{lem:towers2}, there exists a tower $\cT(j) \subset \cT(j,m,l)$ of height $H^{1/8}$ so that 
\begin{multline}\label{eq:final_step_qpark}
	\left\lvert \int_{\cT}\prod_{i=0}^{k} \phi_i^{\perp}\circ \flow_{t_i} \diff\mu\right\rvert \ll \cN_{D_1}(\phi_0) \max_{\substack{\varrho = 1,\dots,d \\ j=1, \dots, k}} \sup_{\substack{0 \leq m,l \leq K \\ |m-l| \geq 1} }\left\lvert \int_{\cT(j)}  \prod_{i=1}^{k} \phi_i^{\perp}\circ \flow_{t_i} \circ \mapg^{\tb}_m \cdot \phi_i^{\perp}\circ \flow_{t_i} \circ \mapg^{\tb}_{l} \diff\mu \right\rvert^{1/2} \\ + \left(\prod_{i=0}^{k} \cN_{D_1}(\phi_i)\right) (M^{-\gamma_k/4} + \height^{-1/16}),
\end{multline}

Let us now fix $1\leq j \leq k$, $1\leq \varrho \leq d$, and $0\leq m,l \leq K$ with $|m-l|\geq 1$, and we consider the right hand side above. 
Property (QLUS3) implies that, for fixed $i$ and $m$, the set $R_{\varrho}(i,m,p) \cap R_{\varrho}(i,l,q)$, for $p,q < M^{1+\gamma_k}$, is an intersection of towers of heights at least $H^{1/2}$.
Thus, by \Cref{lem:intersection_almost_partitions}, as in \S\ref{sec:Proof_induction_park}, the intersections  
\[
R_{\varrho}(\pb, \qb) := \bigcap_{i \in \{1, \dots, k\} \setminus \{j\}} R_{\varrho}(i,m,p_i) \cap R_{\varrho}(i,l,q_i),
\]
over all $\pb = (p_1, \dots, p_{j-1}, p_{j+1}, \dots, p_k), \qb = (q_1, \dots, q_{j-1}, q_{j+1}, \dots, q_k) \in \Z^{k-1}$ with $0 \leq p_i, q_i \leq M^{1+\gamma_k} - 1$,
form a $M^{-\gamma_k}$-almost partition of $\cR_{\varrho}$.
Abbreviating $\cT_{\pb, \qb} := \cT(j,m,l) \cap R_{\varrho}(\pb, \qb)$ and using the property (QLUS3), we get 
\[
\begin{split}
	&	\left\lvert \int_{\cT(j)}  \prod_{i=1}^{k} \phi_i^{\perp}\circ \flow_{t_i} \circ \mapg^{\tb}_m \cdot \phi_i^{\perp}\circ \flow_{t_i} \circ \mapg^{\tb}_{l} \diff\mu \right\rvert \\
	& \qquad	\ll \sum_{\pb, \qb} \left\lvert \int_{\cT_{\pb, \qb}} \, \left( \prod_{i \neq j} (\phi_i^{\perp} \cdot \phi_i^{\perp}\circ \flow_{M^{-\gamma_k} (q_i-p_i)} )\circ \flow_{t_i + M^{-\gamma_k}   p_i} \right) (\phi_j^{\perp} \cdot \phi_j^{\perp}\circ \flow_{(l-m)M^a} )\circ \flow_{t_j + mM^a}\diff\mu \right\rvert \\
	& \qquad \qquad +\left(\prod_{i=1}^{k} \cN_{D_1}(\phi_i)^2\right) M^{-\gamma_k/2}.
\end{split}
\]
By invariance of $\mu$, the integrand above can be composed by $\flow_{-\overline t}$, where ${\overline t} = t_1+M^{-\gamma_k} p_1$, or ${\overline t} = t_1+mM^a$ in case $j=1$. Then, the set 
\[
\flow_{-\overline t}(\cT_{\pb, \qb}) = \flow_{-\overline t}(\cT(j)) \cap \flow_{-\overline t}(R(\pb, \qb))
\] 
is an intersection of $k$ towers of heights $\geq H^{1/8}$. By \Cref{lem:towers2}, there exists a (possibly empty) tower $\widetilde{\cT} \subseteq \flow_{-\overline t}(\cT_{\pb, \qb})$ of height $H^{2^{-k-3}}$, so that 
\[
\mu(\flow_{-\overline t}(\cT_{\pb, \qb}) \setminus \widetilde{\cT} ) \ll H^{-2^{-k-4}} \ll M^{-\beta_k}.
\]
Thus,
\begin{equation}\label{eq:second_goal_qpark}
\begin{split}
	&	\left\lvert \int_{\cT(j)}  \prod_{i=1}^{k} \phi_i^{\perp}\circ \flow_{t_i} \circ \mapg^{\tb}_m \cdot \phi_i^{\perp}\circ \flow_{t_i} \circ \mapg^{\tb}_{l} \diff\mu \right\rvert \\
	& \qquad	\ll \sum_{\pb, \qb} \left\lvert \int_{\widetilde{\cT}} \, \left( \prod_{i \neq j} (\phi_i^{\perp} \cdot \phi_i^{\perp}\circ \flow_{M^{-\gamma_k} (q_i-p_i)} )\circ \flow_{t_i'} \right) (\phi_j^{\perp} \cdot \phi_j^{\perp}\circ \flow_{(l-m)M^a} )\circ \flow_{t_j'}\diff\mu \right\rvert \\
	& \qquad \qquad +\left(\prod_{i=1}^{k} \cN_{D_1}(\phi_i)^2\right) M^{-\beta_k},
\end{split}
\end{equation}
where the $k$-tuple of times $\tb'$ has components $t_i'=t_i + M^{-\gamma_k} p_i -\overline{ t}$ for $i\neq j$, and $t_j' = t_j + mM^a-\overline{ t}$.

We now apply \Cref{lem:uniform_qpark} to the integral in the right hand side of \eqref{eq:second_goal_qpark}, with $R_{-} = M^a \leq R_{+} = M$, and with $\Deltat' \geq \Deltat - 2M \geq \Deltat/2$. We get 
\begin{multline*}
\left\lvert \int_{\widetilde{\cT}} \, \left( \prod_{i \neq j} (\phi_i^{\perp} \cdot \phi_i^{\perp}\circ \flow_{M^{-\gamma_k} (q_i-p_i)} )\circ \flow_{t_i'} \right) (\phi_j^{\perp} \cdot \phi_j^{\perp}\circ \flow_{(l-m)M^a} )\circ \flow_{t_j'}\diff\mu \right\rvert \\
\leq \left(\prod_{i=0}^{k-1} \cN_{D_{k+1}}(\phi_i)^2\right) M^{8\beta_k\sigma_k/\eta_2} \left(M^{-a\eta_2}\mu(\widetilde{\cT}) + M^{\sigma_k} \min \{H, \Deltat\}^{-\frac{\eta_2\delta_k}{16^2\sigma_k}}\right).
\end{multline*}
Plugging this estimate into \eqref{eq:second_goal_qpark}, we obtain
\begin{multline*}
		\left\lvert \int_{\cT(j)}  \prod_{i=1}^{k} \phi_i^{\perp}\circ \flow_{t_i} \circ \mapg^{\tb}_m \cdot \phi_i^{\perp}\circ \flow_{t_i} \circ \mapg^{\tb}_{l} \diff\mu \right\rvert \\
		\leq \left(\prod_{i=0}^{k-1} \cN_{D_{k+1}}(\phi_i)^2\right) \left(M^{-a\eta_2/2} + M^{8(k-1)\sigma_k} 
		H^{-\frac{\eta_2\delta_k}{16^2\sigma_k}} \right),
\end{multline*}
where we used the fact that there are $M^{1+\gamma_k} \leq M^2$ choices for each $p_i$ and $q_i$.
The above estimate, combined with \eqref{eq:final_step_qpark}, proves 
\eqref{eq:proof_prop_qlus_goal} and thus completes the proof of \Cref{prop:qpark_implies_qpark+1}.


\section{Good special flows}\label{sec:Goodflows}

In this section, we focus on special flows and we provide sufficient conditions for them to be LUS and QLUS.

\subsection{Definitions}

Here and henceforth, we assume that the flow $\flowR$ is a special flow over $T \colon (X, \mathcal{B}, \nu) \to (X, \mathcal{B}, \nu)$ with roof function $\roof$ as in \Cref{sec:towers} satisfying the following properties.
\begin{itemize}
	\item[(a)] \emph{Integrability}: $\roof \in L^p(\baseX, \nu)$ for some $p >1$,
	\item[(b)] \emph{Factorization}: there exists a (possibly trivial) partition of $X$ into a family of measurable sets $\{X_z : z \in Z \}$, a disintegration $\{ \nu_z : z \in Z\}$ of $\nu$ into measures $\nu_z$ so that $\nu_z(X \setminus X_z) = 0$, and a measure $\sigma$ on $Z$ such that for any measurable set $A \subseteq X$ we can write 
	\[
	\nu(A) = \int_Z \nu_z(A) \diff \sigma(z).
	\]
\end{itemize} 
The example to have in mind for (b) above is the case of a fiber bundle $X$ over $Z$ (or simply a product space). 
Actually, in the cases we have in mind, the space $X$ is partitioned into 1-dimensional sets $X_z$ (segments or circles) and $\nu_z$ is the 1-dimensional Lebesgue measure. In this article, the factorization is trivial, i.e. $Z=\{z\}$ and $X=X_z$; here, we formulate our assumptions in this generality to be able to apply our results to mixing time-changes of nilflows, which is the subject of forthcoming work.

Recalling the notion of almost partitions introduced in \Cref{sec:towers}, we formulate the following definition.
\begin{definition}[Good almost partitions]
	Let $\varepsilon \geq 0$ and let $A \subseteq X$ be measurable. A family $\cP=\{P_a\}_{a\in \mathscr{A}}$ of measurable sets $P_a \subseteq X$ is a \emph{good $\varepsilon$-almost partition of $A$} if 
	\begin{enumerate}
		\item $\cP$ is a $\varepsilon$-almost partition of $A$ (with respect to the measure $\nu$ on $X$),
		\item for each $a\in \mathscr{A}$, there exists $z \in Z$ so that $P_a \subseteq X_z$,
		\item for every $z \in Z$, the collection $\{P_a \in \cP : P_a \subseteq X_z\}$ is at most countable.
	\end{enumerate}
\end{definition}
From \Cref{lem:intersection_almost_partitions} in \Cref{sec:vdc}, it is immediate that finite intersections of good almost partitions are good almost partitions.

\begin{definition}[Good special flows]\label{def:good_sf}
The flow $\flowR$ is a {\em good special flow} if there exist $d \geq 1$, an increasing sequence $(M_s)_{s \in \N}$, and, for every $\varepsilon >0$ and $k \in \Z_{\geq 0}$, there exists $M_{k,\varepsilon} > \varepsilon^{-2005k}$ such that for every $M_s, \height \geq M_{k,\varepsilon}$ there exist $t_{\height, M_s}\geq \max\{M_s, \height\}$  and a $\varepsilon$-almost partition $\{\cR_{\varrho}\}_{\varrho=1}^d$ of $\susp$ into $d$ towers $\cR_{\varrho}$ of base $B_{\varrho} \subseteq X$ and heights $\height_{\varrho} \in [\height^{1/2},\height^{3/2}]$ for which the following holds.
For every 
$(k+1)$-tuple of times $\tb$ with $\Deltat > t_{\height, M_s}$, there exists a (possibly uncountable) family $\cP_{\tb}=\cP_{\tb}(\height,M)=\{P_a\}_{a\in \mathscr{A}}$ such that

\begin{enumerate}
\item[(G1)] we have
\[
\max_{i=0,\ldots, k}\sup_{a\in \mathscr{A}} \;  \sup_{x\in P_a} \; \max_{0 \leq t \leq M_s + \height^{3/2}} \diam( \flow_t( T^{N(x,t_i)} P_a )) < \varepsilon;
\]
\item[(G2)] $\cP_{\tb}$ is a good $\varepsilon$-almost partition of $B_\varrho$ for every $\varrho= 1,\dots, d$;
\item[(G3)] for any $a\in \mathscr{A}$, $i \in \{1,\dots, k\}$, and $x \in P_a$, define $S^{+}(a,t_i,x)\geq 0$ (respectively $S^{-}(a,t_i,x)\leq 0$) be the smallest (respectively largest) number such that 
\[
\flow_{t_i}(P_a)\subset \bigcup_{S^{-}(a,t_i,x)\leq r\leq S^{+}(a,t_i,x)} \flow_r(T^{N(x,t_i)}(P_a)).
\]
Then, the latter set is a tower and $S^{+}(a,t_i,x)-S^{-}(a,t_i,x) <M_s$.

\item[(G4)] For every $a\in \mathscr{A}$ there exists $i_a\in \{1,\ldots, k\}$ and $x_a\in P_a\subseteq X_z$ such that 
 $\flow_{t_{i_a}}(x_a,0)=(T^{N(x_a,t_{i_a})}x_a,0)$ and 
$$
\Delta S_a := S^{+}(a,t_{i_a},x_a)-S^{-}(a,t_{i_a},x_a)> \varepsilon^{1000k}M_s
$$
Moreover
\begin{equation*}
\frac{\varepsilon (1-\varepsilon) \nu_z(P_a)}{\Delta S_a } \leq  \nu_z\Big\{x\in P_a \, : \, \flow_{t_{i_a}}(x)\in \bigcup_{r\in [L,L+\varepsilon]}\flow_r(T^{N(x_a,t_{i_a})}P_a) \Big\} \leq  \frac{\varepsilon (1+\varepsilon) \nu_z(P_a)}{\Delta S_a },
\end{equation*}
for every $S^{-}(a,t_{i_a},x_a)\leq L\leq S^{+}(a,t_{i_a},x_a)-\varepsilon$, see \Cref{fig:1}.
\end{enumerate}
\end{definition}

\begin{remark}
Let us explain the intuition behind this definition.
Properties (G1) and (G2) guarantee that the family $\cP_{\tb}$ is an almost partition of the bases of the Rokhlin towers with atoms converging to points. A crucial point is to control the shearing on each atom $P_a$. Property (G1) allows us to be sure that the deformation of the set $\flow_t(P_a)$ is only due to the shearing coming from fluctuations of the ergodic sums of the roof function; in other words, there is no stretching due to the base map. 

Property (G3) provides a uniform upper bound on the shearing of the images of $P_a$ under $\flow_{t_i}$. 
Property (L5) has two parts, giving us a fine control on the shearing. The first part  implies that, for every $a\in \mathscr{A}$, the atom $P_a$ is sheared by $\flow_{t_i}$ for at least one $i \in \{1,\dots, k\}$. Finally, the crucial second part is where the adjective \lq\lq locally uniformly\rq\rq\ comes from: on the atoms of $\cP_{\tb}$, the image of the stretched curve $\flow_{t_{i_a}}(P_a)$ is almost uniformly distributed with respect to the flow direction.
\end{remark}

\begin{figure}[h]
	\centering
	\begin{tikzpicture}[scale=3.5]
		\draw (0,0) -- (1,0);
		\draw[white] (0,0.01)  node[anchor=east] {\textcolor{black}{$P_a$}};
		\draw[->] (1.3,0.5) -- (1.7,0.5);
		\draw[white] (1.5,0.51)  node[anchor=south] {\textcolor{black}{$\flow_{t_{i_a}}$}};
		\draw plot [smooth] coordinates { (2.2,-0.5) (2.4,0) (2.6,0.2) (2.8,0.6) (3,0.8) (3.2,1.2)};
		\draw[white] (3.1,1)  node[anchor=east] {\textcolor{black}{$\flow_{t_{i_a}}(P_a)$}};
		\draw (2.2,0) -- (3.2,0);
		\draw[white] (3.2,0.01)  node[anchor=west] {\textcolor{black}{$T^{N(x_a,t_{i_a})}P_a$}};
		\foreach \Point in {(0.2,-0.01), (2.4,-0.01)}{
			\node at \Point {\textbullet};
		}
		\draw[white] (0.2,-0.01)  node[anchor=north] {\textcolor{black}{$x_a$}};
		\draw[white] (2.4,-0.1)  node[anchor=west] 	{\textcolor{black}{$T^{N(x_a,t_{i_a})}(x_a)$}};
		\draw[thick, densely dotted] (2.2,-0.5) -- (2.2,1.2); 
		\draw[thick, densely dotted] (2.2,-0.5) -- (3.2,-0.5); 
		\draw[thick, densely dotted] (2.2,1.2) -- (3.2,1.2); 
		\draw[thick, densely dotted] (3.2,-0.5) -- (3.2,1.2); 
		\draw[white] (2.1,-0.5)  node[anchor=north] 	{\textcolor{black}{$S^{-}(a,t_{i_a},x_a)$}};
		\draw[white] (2.1,1.2)  node[anchor=south] 	{\textcolor{black}{$S^{+}(a,t_{i_a},x_a)$}};
		\draw[thick, blue, densely dotted] (2.2,0.5) -- (3.2,0.5); 
		\draw[thick, blue, densely dotted] (2.2,0.6) -- (3.2,0.6); 
		\draw[white] (2.2,0.65)  node[anchor=east] 	{\textcolor{blue}{$L+\varepsilon$}};
		\draw[white] (2.2,0.45)  node[anchor=east] 	{\textcolor{blue}{$L$}};
		\draw [line width=1mm, blue] (2.747,0.5) -- (2.8,0.6);
		\draw [line width=1mm, blue] (0.547,0) -- (0.6,0);
		\draw[white] (2.73,0.4)  node[anchor=west] {\textcolor{blue}{$\flow_{t_{i_a}}(A_L)$}};
		\draw[white] (0.55,0.01)  node[anchor=south] 	{\textcolor{blue}{$A_L$}};
	
		\clip (-0.5,-0.6) rectangle (3.5,1.5);
	\end{tikzpicture}
	\caption{Condition (G4) for LUS flows: for any valid choice of $L$, the relative measure of the blue set $A_L$ in $P_a$ is approximately $\varepsilon / \Delta S_a$, where $\Delta S_a := S^{+}(a,t_{i_a},x_a)-S^{-}(a,t_{i_a},x_a)$.}
	\label{fig:1}
\end{figure}
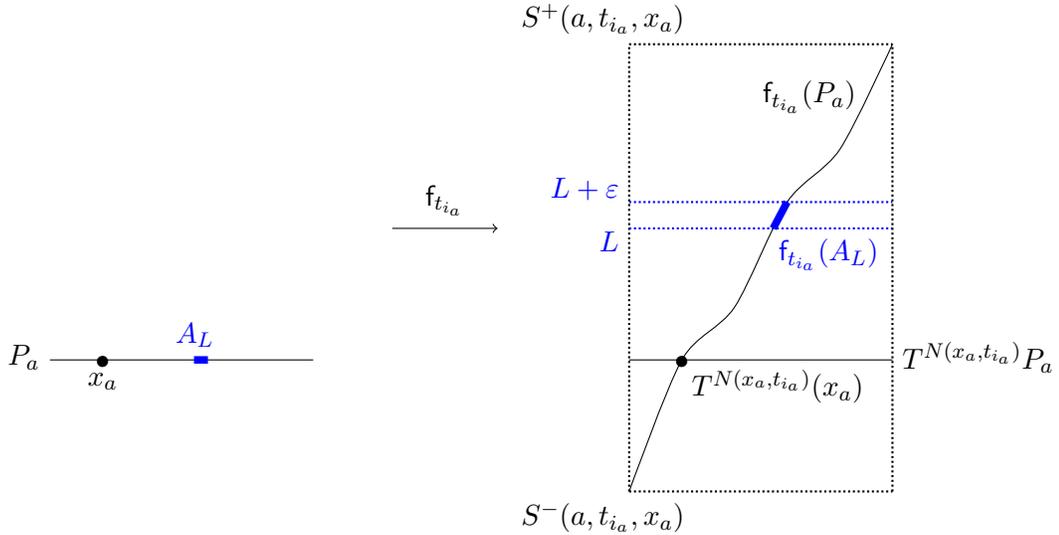

The following quantitative version corresponds to the notion of QLUS flows. 

\begin{definition}[Quantitatively good special flows]\label{def:qgood_sf}
	The flow $\flowR$ is a {\em quantitatively good special flow} if there exist $d\geq 1$ and, for every $k \in \Z_{\geq 0}$, there exists a rate $0< \varepsilon_k \leq 1/8$ and for every $1 < M \leq \height$ there exists a $M^{-\varepsilon_k}$-almost partition $\{\cR_{\varrho}\}_{\varrho=1}^d$ of $\susp$ into $d$ towers $\cR_{\varrho}$ with bases $B_{\varrho} \subset \baseX$ and heights $\height_{\varrho}\in[\height^{1/2},\height^{3/2}]$ for which the following holds.  
	For every $(k+1)$-tuple of times $\tb$ with $\Deltat^{\varepsilon_k} \geq 4\height$, there exists a (possibly uncountable) family $\cP_{\tb}=\{P_a\}_{a\in \mathscr{A}}$ such that

	\begin{enumerate}
		\item[(QG1)] we have
		\[
		\max_{i=0,\ldots, k}\sup_{a\in \mathscr{A}} \; \sup_{x\in P_a} \; \max_{0 \leq t \leq 2\height^{3/2}} \diam( \flow_t(T^{N(x,t_i)} P_a)) < M^{-\varepsilon_k};
		\]
		\item[(QG2)] $\cP_{\tb}$ is a good $M^{-\varepsilon_k}$-almost partition of $B_{\varrho}$ for every $\varrho=1,\dots, d$;
		\item[(QG3)] for any $a\in \mathscr{A}$, $i \in \{1,\dots, k\}$, and $x \in P_a$, define $S^{+}(a,t_i,x)\geq 0$ (respectively $S^{-}(a,t_i,x)\leq 0$) be the smallest (respectively largest) number such that 
		\[
		\flow_{t_i}(P_a)\subset \bigcup_{S^{-}(a,t_i,x)\leq r\leq S^{+}(a,t_i,x)} \flow_r(T^{N(x,t_i)}(P_a)).
		\]
		Then, the latter set is a tower and $S^{+}(a,t_i,x) - S^{-}(a,t_i,x) <M$.
		\item[(QG4)] For every $a\in \mathscr{A}$ there exists $i_a\in \{1,\ldots, k\}$ and $x_a\in P_a\subseteq X_z$ such that 
 $\flow_{t_{i_a}}(x_a,0)=(T^{N(x_a,t_{i_a})}x_a,0)$ and 
$$
\Delta S_a := S^{+}(a,t_{i_a},x_a)-S^{-}(a,t_{i_a},x_a)> M^{1-\varepsilon_k}.
$$
Moreover
	\begin{multline*}
			\frac{M^{-\varepsilon_k} (1-M^{-\varepsilon_k}) \nu_z(P_a)}{\Delta S_a} \leq  \nu_z\Big\{x\in P_a \, : \, \flow_{t_{i_a}}(x)\in \bigcup_{r\in [L,L+M^{-\varepsilon_k}]}\flow_r(T^{N(x_a,t_{i_a})}P_a) \Big\} \\ \leq  \frac{M^{-\varepsilon_k} (1+M^{-\varepsilon_k}) \nu_z(P_a)}{\Delta S_a},
		\end{multline*}
		for every $S^{-}(a,t_{i_a},x_a)\leq L\leq S^{+}(a,t_{i_a},x_a)-M^{-\varepsilon_k}$.
		\end{enumerate}
\end{definition}

\begin{proposition}\label{prop:good_implies_lus}
Good special flows are LUS. Furthermore, quantitatively good special flows with rate function $\varepsilon_k$ are QLUS with rate function $\gamma_k = \varepsilon_k^{1/4}$. 
\end{proposition}


\subsection{Proof of \Cref{prop:good_implies_lus}}\label{sec:contruct_mapg}

We prove that good special flows are LUS; the proof for the second claim is identical, with $M^{-\varepsilon_k}$ playing the role of $\varepsilon$.

Let $\flowR$ be a good special flow, and let $d \geq 1$ and $(M_s)_{s \in \N}$ be given by definition.
In \Cref{def:LUS_flows}, we set the same $d$ and $(M_s)_{s \in \N}$, together with $a=1/2$. 
Fix $\delta >0$ and $k \in \Z_{\geq 0}$; we let $\varepsilon = \delta^4/(20d)$ and $ M_{k, \varepsilon} \geq \varepsilon^{-2005k}$ be given by the definition of good special flows.
Fix $M_s, \height \geq M_{k, \varepsilon}$. 
By definition (see also \Cref{cor:towers3}), let $(\cR_\varrho)_{\varrho =1}^d$ be Rokhlin towers of heights $\height_\varrho \in [\height^{1/2},\height^{3/2}]$ and bases $B_\varrho \subset X$ such that $\mu(\susp \setminus \cup_\varrho \cR_\varrho) \leq \varepsilon$. 
Let us fix a $(k+1)$-tuple of times $\tb$ with $\Deltat > t_{\height, M_s}$, where $t_{\height, M_s}$ is given by \Cref{def:good_sf}. Let also $\cP_{\tb} = \{P_a\}_{a \in \mathscr{A}}$ be the associated good $\varepsilon$-almost partition of the bases $B_{\varrho}$, by (G2).
We first define $\delta^2$-almost measure preserving maps on the bases $B_\varrho$ into themselves with the aid of $\cP_{\tb}$, and then extend them to the union of the towers $\cR_\varrho$.

Fix $a \in \mathscr{A}$ and consider $i_a \in \{1,\dots, k\}$, $x_a \in P_a$, and $S_a^{\pm} = S_a^{\pm}(a,t_{i_a},x_a)$ be given by (G4) in \Cref{def:good_sf}.
For $0\leq L \leq \Delta S_a - \varepsilon$, where $\Delta S_a := S_a^{+} - S_a^{-} $, we define
\[
A_L^{a,\tb} := \left\{ x \in P_a : \flow_{t_{i_a}}(x) \in \bigcup_{r\in [S_a^{-} + L,S_a^{-} + L+\varepsilon]}\flow_r(T^{N(x_a,t_{i_a})}P_a) \right\},
\]
see \Cref{fig:1}. 
By property (G4), for all $L,L'$ as above, we have
\begin{equation}\label{eq:unifmeasureA}
1-3\varepsilon \leq \frac{\nu_z(A_L^{a,\tb})}{\nu_z(A_{L'}^{a,\tb})} \leq 1+3\varepsilon,
\end{equation}
provided that $\varepsilon <1/3$.

For $m \leq K_{\varepsilon} := \lfloor \varepsilon^{-1/4} \rfloor \leq \lfloor \delta^{-1} \rfloor$, let $N_m = m \cdot \sqrt{M_s}$. Let us also consider $0\leq \ell \leq \varepsilon^{-1} (\Delta S_a -  K_{\varepsilon} \sqrt{M_s})$ and $L = \ell \cdot \varepsilon \leq \Delta S_a -  K_{\varepsilon} \sqrt{M_s}$. By \eqref{eq:unifmeasureA}, there exists a $3\varepsilon$-almost measure preserving map
\[
\mapg^{a,\tb}_{m,L} \colon A_L^{a,\tb} \to A_{L+N_m}^{a,\tb}.
\]

\begin{lemma}
There exists a $3\varepsilon$-almost measure preserving map $\mapg^{a,\tb}_m \colon P_a \to P_a$ whose restriction to any $A_{\ell \varepsilon}^{a,\tb}$ with $0 \leq \ell \leq \varepsilon^{-1} (\Delta S_a - K_{\varepsilon}\sqrt{M_s})$ coincides with $\mapg^{a,\tb}_{m,L}$.
\end{lemma}
\begin{proof}
It suffices to show that the sets $\{A_{\ell \varepsilon}^{a,\tb} : 0\leq \ell \leq \varepsilon^{-1} (\Delta S - K_{\varepsilon}\sqrt{M_s})\}$ are pairwise disjoint and they cover $P_a$ up to a set of measure at most $\varepsilon^2 \nu_z(P_a)$; their union $D_m^{a,\tb}$ will form the proper domain of $\mapg^{a,\tb}_m$.

The fact that the sets $A_{\ell \varepsilon}^{a,\tb}$ are disjoint for different $\ell$ follows from their definition and property (G3). Using property (G4) we have 
\begin{multline*}
\nu_z \left(P_a \setminus \bigcup\{A_{\ell \varepsilon}^{a,\tb} : 0\leq \ell \leq \varepsilon^{-1} (\Delta S_a -  K_{\varepsilon} \sqrt{M_s})\} \right) \leq \sum_{\ell =  \varepsilon^{-1} (\Delta S_a -  K_{\varepsilon} \sqrt{M_s})}^{ \varepsilon^{-1} (\Delta S_a - \varepsilon)} \nu_z(A_{\ell \varepsilon}^{a,\tb})  \\ \leq  \varepsilon^{-1}  K_{\varepsilon} \sqrt{M_s}  \frac{\varepsilon (1+\varepsilon) \nu_z(P_a)}{\Delta S_a} \leq 2 \varepsilon^{-1/4 - 1000k} M_{k,\varepsilon}^{-1/2} \nu_z(P_a) \leq \varepsilon^2 \nu_z(P_a),
\end{multline*}
since $M_{k,\varepsilon} \geq \varepsilon^{-2005k}$. 
\end{proof}

We now extend $\mapg^{a,\tb}_m$ to almost measure preserving maps over the bases $B_\varrho$ as follows.

\begin{lemma}\label{lem:mapg_on_bases}
	For any $\varrho=1, \dots, d$, there exists a $3\sqrt{\varepsilon}$-almost measure preserving map $\mapg^{\varrho,\tb}_m \colon B_\varrho \to B_\varrho$ whose restriction to any $P_a$ coincides with $\mapg^{a,\tb}_m$.
\end{lemma}
\begin{proof}
	We can define $\mapg^{\tb}_m \colon \cup \cP_{\tb} \to \cup \cP_{\tb}$ by $\mapg^{a,\tb}_m$ for each atom of the almost partition $\cP_{\tb}$. The proper domain of this map is then the union $ \cup_a D_m^{a, \tb}$ of the proper domains $D_m^{a, \tb}$ of the maps $\mapg^{a,\tb}_m$. Let us show that it restrict to a $3\sqrt{\varepsilon}$-almost measure preserving map of the bases $B_\varrho$ of each Rokhlin tower $\cR_\varrho$. By (G2), we have 
	\[
	\nu\left(B_\varrho \triangle \bigcup \{P_a \in \cP_{\tb} : P_a \cap B_\varrho \neq \emptyset \} \right)<\varepsilon \nu(B_\varrho).  
	\]
	Thus, using the disintegration of $\nu$, we get
	\[
	\begin{split}
		\varepsilon \nu(B_\varrho) &\geq \nu \left(B_\varrho^c \cap \bigcup \{P_a \in \cP_{\tb} : P_a \cap B_\varrho \neq \emptyset \} \right) \\
		&\geq \int_Z \nu_z \left(B_\varrho^c \cap \bigcup \{ P_a \subseteq X_z : P_a \cap B_\varrho \neq \emptyset \text{ and } \nu_z (B_\varrho^c \cap P_a) \geq \sqrt{\varepsilon} \nu_z(P_a)  \} \right) \diff \sigma(z)\\
		& = \int_Z \sum \left\{\nu_z (B_\varrho^c \cap  P_a) \ :\ P_a \subseteq X_z, \text{ } P_a \cap B_\varrho \neq \emptyset \text{ and } \nu_z (B_\varrho^c \cap P_a) \geq \sqrt{\varepsilon} \nu_z(P_a)\right\}\\
		& \geq \sqrt{\varepsilon}\nu\left(\bigcup \{ P_a : P_a \cap B_\varrho \neq \emptyset \text{ and } \nu_z (B_\varrho^c \cap P_a) \geq \sqrt{\varepsilon} \nu_z(P_a)  \} \right);
	\end{split}
	\]
	note that the sum above is at most countable by the properties of $\cP_{\tb}$. 
	We deduce that
	\[
	\nu\left(\bigcup \{ P_a : P_a \cap B_\varrho \neq \emptyset \text{ and } \nu_z (B_\varrho \cap P_a) < (1-\sqrt{\varepsilon}) \nu_z(P_a)  \} \right) \leq \sqrt{\varepsilon} \nu(B_\varrho);
	\]
	in other words, up to removing from $B_\varrho$ a set of atoms whose union has measure at most $\sqrt{\varepsilon} \nu(B_\varrho)$, we can make sure that all atoms $P_a$ which intersect $B_\varrho$ do so in a set of relative measure at least $(1-\sqrt{\varepsilon})$. Let us call $\mathscr{A}'$ the corresponding collection of indexes of these atoms. 
	Then, the set 
	\[
	D_m^{\varrho, \tb} = \bigcup_{a \in \mathscr{A}'} D_m^{a,\tb} \cap B_\varrho \cap (\mapg^{a,\tb}_m)^{-1} (B_\varrho)
	\]
	has measure at least $(1-3\sqrt{\varepsilon})\nu(B_\varrho)$ and is a proper domain for $\mapg^{\varrho,\tb}_m \colon B_\varrho \to B_\varrho$. This proves the lemma.
\end{proof}

By the properties of towers, we can extend the maps $\mapg^{\varrho,\tb}_m$ to the whole tower $\cR_\varrho$ by setting
\begin{equation}\label{eq:towdef}
\mapg^{\varrho,\tb}_m (\flow_t(x)) = \flow_t(\mapg^{\varrho,\tb}_m (x)),
\end{equation}
for $x \in B_\varrho$ and $0\leq t \leq \height_\varrho$. In particular, since the towers $\cR_\varrho$ are disjoint, we can define a map $\mapg^{\tb}_m \colon \susp \to \susp$ by $\mapg^{\tb}_m|_{\cR_\varrho} = \mapg^{\varrho,\tb}_m$, which is $3\sqrt{\varepsilon}$-almost measure preserving. We denote by 
\[
D^{\tb}= \bigcap_{m=0}^{ K_{\varepsilon} } \bigcup_{\varrho=1}^d \bigcup_{t \in [0,\height_\varrho]} \flow_t(D_m^{\varrho,\tb})
\]
the intersection of the proper domains of $\mapg^{\tb}_m$, for all $m=0, \dots, K_{\varepsilon}$.
Note that $\mu(\susp \setminus D^{\tb}) \leq 4\sqrt{\varepsilon}  K_{\varepsilon} \leq 4 \varepsilon^{1/4} \leq \delta$.

Let us verify (LUS1): fix $0 \leq m \leq K_{\varepsilon}$ and $x \in D^{\tb}$, write $x = \flow_t(y) \in \cR_{\varrho}$ for some $0\leq t \leq \height_{\varrho}$ and $y \in D_m^{\varrho,\tb} \subseteq P_a$. Then, by (G1),
\[
\dist \left(\mapg^{\tb}_m (x), x \right)= \dist \left(\flow_t(\mapg^{\varrho, \tb}_m y), \flow_t(y) \right) \leq \diam(\flow_t(P_a)) \leq \varepsilon.
\]
This proves (LUS1)-(a); the following lemma proves (LUS1)-(b).

\begin{lemma}\label{lem:mapg_sends_towers_to_towers}
	Let 
	$\cT$ be a tower of height $H \geq \height^{1/2}$. There exists a tower $\overline{\cT} \subseteq (\mapg^{\tb}_m)^{-1}(\cT)$ of height $\height^{1/4}$ so that 
	\[
	\mu \left((\mapg^{\tb}_m)^{-1}(\cT) \setminus \overline{\cT}\right) \leq 16d \varepsilon^{1/4} \qquad \text{for all} \qquad m=0, \dots, K_{\varepsilon}.
	\]
\end{lemma}
\begin{proof}
	Since $H, \height_\varrho \geq \height^{1/2}$, by \Cref{lem:towers2}, for every $\varrho = 1,\dots,d$, there exists a tower $\cT_{\varrho} \subseteq \cT \cap \cR_{\varrho}$ of height $\height^{1/4}$ so that $\mu(\cT \cap \cR_{\varrho} \setminus \cT_{\varrho})\leq 4 \height^{-1/4}$. By construction of the map $\mapg^{\tb}_m$, see \eqref{eq:towdef}, the set $(\mapg^{\tb}_m)^{-1}(\cT_{\varrho})$ is a tower contained in $\cR_{\varrho}$ of the same height as $\cT_{\varrho}$. We then consider the tower $\overline{\cT} = (\mapg^{\tb}_m)^{-1}(\cT_1) \cup \cdots \cup (\mapg^{\tb}_m)^{-1}(\cT_d)$. By the almost measure preserving property, we have
	\[
	\begin{split}
		\mu \left( (\mapg^{\tb}_m)^{-1}(\cT) \setminus \overline{\cT} \right) &\leq   \mu\left( (\mapg^{\tb}_m)^{-1}(\cT) \setminus \bigcup_{\varrho = 1}^d \cR_{\varrho}\right) + (1+\varepsilon) \sum_{\varrho = 1}^d \mu\left( \cT \cap \cR_{\varrho} \setminus \cT_{\varrho} \right) \\
		&\leq 5 \varepsilon^{1/4}+ 4d(1+\varepsilon)\height^{-1/4} \leq 16d \varepsilon^{1/4},
	\end{split}
	\]
	which completes the proof.
\end{proof}

We group the atoms $P_a$ (and hence the maps $\mapg^{\tb}_m$ as well) into the following sets: for $j \in \{1, \dots, k\}$, let $\mathscr{A}_{j} = \{ a \in \mathscr{A} : i_a = j\}$. We then define the corresponding bases and towers
\[
B_{\varrho,j} = \bigcap_{m=0}^{ K_{\varepsilon} }D_m^{\varrho,\tb} \cap \bigcup_{a\in \mathscr{A}_{j}} P_a \subseteq B_\varrho, \qquad \text{and} \qquad \Tow_{\varrho,j} = \bigcup_{0\leq t\leq \height_\varrho} \flow_t(B_{\varrho,j}) \subseteq D^{\tb} \cap \cR_\varrho.
\]
In other words, $B_{\varrho,j}$ consists of those points in the base $B_\varrho$ which belong to the proper domain of the maps $\mapg^{\varrho,\tb}_m$ for all $m=0,\dots, K_{\varepsilon}$ and moreover for which the corresponding $i_a$ in (G4) equals $j$. Note that 
\begin{equation}\label{eq:measure_union_towj}
	\begin{split}
&\mu \left( \bigcup_{j=1}^k \Tow_{\varrho, j}\right) \geq \mu(\cR_\varrho) - \nu\left(B_\varrho \setminus \bigcup_{m=0}^{ K_{\varepsilon} }D_m^{\varrho,\tb}\right) \height_\varrho \geq \big( 1- 4\varepsilon^{1/4} \big)\mu(\cR_\varrho), \quad \text{and} \\
&\mu \left( \bigcup_{\varrho=1}^d \bigcup_{j=1}^k \Tow_{\varrho, j}\right) \geq 1 - \mu \left( \susp \setminus \bigcup_{\varrho=1}^d \cR_\varrho \right) - \sum_{\varrho=1}^{d} \mu \left( \cR_\varrho \setminus \bigcup_{j=1}^k \Tow_{\varrho, j}\right) \geq 1- 5\varepsilon^{1/4}.
	\end{split}
\end{equation}
We can now prove property (LUS2).

\begin{lemma}\label{lem:locren} Property (LUS2) is satisfied: for any $m\leq  K_{\varepsilon} $, $j  \in \{1, \dots, k\}$, and $\varrho \in \{1, \dots, d\}$ we have 
	\[
	\sup_{x\in \Tow_{\varrho,j}} \dist (\flow_{t_j}(\mapg^{\tb}_m x), \flow_{t_j+N_m}(x))\leq 3\varepsilon. 
	\]
\end{lemma}
\begin{proof} 
	Let us consider $x = \flow_H(y) \in \Tow_{\varrho,j}$, with $y\in B_{\varrho,j}$ and $H \leq \height_\varrho$. In particular, there exists $a \in \mathscr{A}_{j}$ as defined above so that $y \in P_a$, and $y \in D_m^{\varrho,\tb}$ for all $m = 0, \dots, K_{\varepsilon}$. By construction, there exists $0\leq \ell \leq \varepsilon^{-1} (\Delta S_a -  K_{\varepsilon} \sqrt{M_s})$ so that $y \in A^{a,\tb}_{\ell \varepsilon} \subset P_a$ and $\mapg^{\tb}_m y \in A^{a,\tb}_{\ell \varepsilon + N_m} \subset P_a$. Thus, there exist $r_1 \in [S_a^{-} + \ell \varepsilon, S_a^{-} + (\ell +1) \varepsilon]$ and $r_2 \in [S_a^{-} + \ell \varepsilon +N_m, S_a^{-} + (\ell +1) \varepsilon +N_m]$, where $S_a^{-} = S^{-}(a,t_{j},x_a)$, so that
	\[
	\flow_{t_{j}}(y) \in \flow_{r_1} (T^{N(x_a,t_{j})}P_a), \qquad \text{and} \qquad  \flow_{t_{j}}(\mapg^{\tb}_m y) \in \flow_{r_2} (T^{N(x_a,t_{j})}P_a).
	\]
	This, recalling the definition \eqref{eq:towdef}, implies that 
	\begin{multline*}
	\dist (\flow_{t_{j}}(\mapg^{\tb}_m x), \flow_{t_{j} + N_m}(x)) = \dist (\flow_{t_{j}+H}(\mapg^{\tb}_m y), \flow_{t_{j} + N_m + H}(y)) \\ \leq |r_2 - r_1 -N_m| + \diam (\flow_{r_2+H} (T^{N(x_a,t_{j})}P_a)) \leq 3\varepsilon,
	\end{multline*}
	where we used the fact that $r_2 + H \leq \Delta S_a + \height_\varrho \leq M_s + \height^{3/2}$, together with property (G1).
\end{proof}

We now define subtowers of the towers $\cR_\varrho$ on which we control the shear for the other times $t_i$, in order to prove (LUS3).
For any $a \in \mathscr{A}$, we fix $x_0 \in P_a$. Now, given $ i \in {1,\dots,k}$ and $0 \leq \ell \leq \varepsilon^{-1} M_s - 1$, we let 
\[
P_a(i,\ell) := \left\{ x \in P_a : \flow_{t_i}(x) \in \bigcup_{r\in [S_0^{-} + \varepsilon \ell,S_0^{-} + \varepsilon (\ell +1)]}\flow_r(T^{N(x_0,t_{i})}P_a) \right\},
\]
where $S_0^{-} = S^{-}(a,t_i,x_0)$.
Note that, by property (G3), the sets $P_a(i,0), \dots, P_a(i,\varepsilon^{-1} M_s - 1)$ form a partition of $P_a$ (note that the sets $P_a(i,\ell)$ for $\varepsilon^{-1}(S_0^+-S_0^-) \leq \ell \leq \varepsilon^{-1} M_s -1$ are actually empty). 

Given $0 \leq \ell \leq \varepsilon^{-1} M_s - 1$, we define
\[
Q_a(i, m, \ell) = \bigcup_{|\ell_2 - \ell_1| = \ell} P_a(i,\ell_1) \cap (\mapg^{\tb}_m)^{-1}(P_a(i,\ell_2)) \subseteq P_a, \qquad Q(i, m, \ell) = \bigcup_{a \in \mathscr{A}}Q_a(i, m, \ell),
\]
and
\[
R_\varrho(i, m, \ell) = \bigcup_{0\leq t \leq \height_\varrho} \flow_t(Q(i, m, \ell) \cap B_\varrho) \subseteq \cR_\varrho.
\]
\begin{lemma}\label{lem:R_varsigma_almost_partition}
For any fixed $\varrho\in \{1,\dots, d\}$, $i\in \{1,\dots, k\}$, and $m\in \{0,\dots, K_{\varepsilon}\}$, the sets $R_\varrho(i, m,0), \dots, R_\varrho(i, m,\varepsilon^{-1} M_s - 1)$ form a $4\varepsilon$-almost partition of $\cR_\varrho$.
\end{lemma}
\begin{proof}
	The sets $P_a(i,\ell)$ are pairwise disjoint for different values of $\ell$, hence so are the sets $Q_a(i,m,\ell)$. By the properties of towers, the sets $R_\varrho(i, m,0), \dots, R_\varrho(i, m,\varepsilon^{-1} M - 1)$ are also pairwise disjoint.
	
Since
\[
\bigcup_{\ell = 0}^{\varepsilon^{-1} M_s - 1} Q_a(i, m, \ell) = \bigcup_{\ell_1 = 0}^{\varepsilon^{-1} M_s - 1} P_a(i,\ell_1) \cap (\mapg^{\tb}_m)^{-1}\left(\bigcup_{\ell_2 = 0}^{\varepsilon^{-1} M_s - 1} P_a(i,\ell_2)\right) = P_a \cap (\mapg^{\tb}_m)^{-1}(P_a),
\]
the almost measure preserving property of $\mapg^{\tb}_m = \mapg^{a,\tb}_m$ on $P_a$ tells us that 
\[
\nu\left( \bigcup_{\ell = 0}^{\varepsilon^{-1} M_s - 1} Q_a(i, m, \ell)\right) \geq (1-3\varepsilon) \nu(P_a).
\]
Therefore, by property (G2),
\[
\nu\left( B_\varrho \cap \bigcup_{a \in \mathscr{A}} \bigcup_{\ell = 0}^{\varepsilon^{-1} M_s - 1} Q_a(i, m, \ell)\right) \geq \nu\left( B_\varrho \cap \bigcup_{a \in \mathscr{A}} P_a \right) - 3\varepsilon \nu(B_\varrho) \geq (1-4\varepsilon) \nu(B_\varrho),
\]
and thus
\[
\mu\left(\bigcup_{\ell = 0}^{\varepsilon^{-1} M_s - 1} R_\varrho(i, m, \ell)\right) = \height_\varrho \cdot \nu\left( B_\varrho \cap \bigcup_{a \in \mathscr{A}} \bigcup_{\ell = 0}^{\varepsilon^{-1} M_s - 1} Q_a(i, m, \ell) \right) \geq (1-4\varepsilon) \mu(\cR_\varrho),
\]
which proves the claim.
\end{proof}

On the towers $R_\varrho(i,m,\ell)$, the control the shear is as dictated by (LUS3), which completes the proof of \Cref{prop:good_implies_lus}.
\begin{lemma}\label{lem:disinter} 
	Property (LUS3) holds: for every $i \in \{1, \dots, k\}$, every $0 \leq \ell \leq \varepsilon^{-1} M - 1$, every $m \leq  K_{\varepsilon} $, and every $x\in R_\varrho(i,m, \ell)$, we have
	\[
		\dist \Big(\flow_{t_i}(\mapg^{\tb}_m x),f_{t_i+\varepsilon \ell}(x)\Big)<2 \varepsilon.
	\]
\end{lemma}
\begin{proof}
	Let $x \in R_\varrho(i,m, \ell)$, and write $x = \flow_H(y)$, with $y \in B_\varrho \cap P_a(i,\ell_1) \cap (\mapg^{\tb}_m)^{-1}(P_a(i,\ell_2))$ for some $a \in \mathscr{A}$, some $|\ell_2 - \ell_1| = \ell$, and $0\leq H \leq \height_\varrho$. Thus, there exist $0 \leq \delta_1, \delta_2 \leq \varepsilon$ so that $\flow_{t_i}(\mapg^{\tb}_m y)$ and $\flow_{t_i + \varepsilon \ell + \delta_1}(y)$ both belong to $\flow_{S^- + \varepsilon \ell_2 + \delta_2 }(T^{N(x_0,t_i)}P_a)$. Using property (G1), we conclude
	\[
	\dist \Big(\flow_{t_i}(\mapg^{\tb}_m x),f_{t_i+\varepsilon \ell}(x)\Big) \leq \varepsilon + \diam \left(\flow_{H+S^- + \varepsilon \ell_2 + \delta_2 }(T^{N(x_0,t_i)}P_a)\right) \leq 2\varepsilon.
	\]
\end{proof}


\section{LUS flows on surfaces}\label{sec:LUS_on_surfaces}

In this section $(\flow_t,\susp,\mu^\roof)$ is a special flow over an ergodic interval exchange transformation (IET) $T\colon \T\to \T$ whose discontinuities are a subset of $\{\beta_i\}_{i=1}^d$. We will assume moreover that the roof function satisfies $\Phi\in \mathscr{C}^2(\T \setminus\{\beta_1,\ldots \beta_d\})$.  Note that this allows the function $\Phi$ to have a singularity at a point which is not a discontinuity of $T$ (for example if $T$ is a rotation). We also assume that $\inf \Phi=c>0$ and for simplicity we normalize $\Phi$ so that $\mu(\Phi)=1$.

\subsection{Preliminaries}

  Let $\zeta>0$ satisfy $(1+\zeta)(1+\gamma)<2$ and  be small enough so that if
  $$
  V_R:=\Big\{x\in \T\;:\; \Phi(x)\leq R^{(1+\zeta/2)\gamma} \Big\}.
  $$
 then for sufficiently large $R$, 
  \begin{equation}\label{eq:incl}
  V_R\cap \bigcup_{i=1}^d\Big[-\frac{1}{R^{1+\zeta}}+\beta_i,\frac{1}{R^{1+\zeta}}+\beta_i\Big]=\emptyset.
  \end{equation}

  We will often abbreviate the Birkhoff sums $S_N\roof$ of $\roof$ simply with $S_N$.
  We start with the following lemma.
  
  \begin{lemma}\label{lem:nm1} For every $R>0$, any $t>0$, and for any interval $I=[a,b]\subset \T$ with $|I|<R^{-4}$  satisfying  $0\leq \max_{x\in I}N(x,t)-\min_{x\in I}N(x,t)<6c^{-1}R$,
  \begin{equation}\label{eq:nsin}
 T^{N(z,t)}I\in V_R\text{ for every }z\in I, \;\;\;\text{ and }\;\;\; \bigcup_{\ell\in [0,\max_{x\in I}N(x,t)]}T^{\ell}(I)\cap \{\beta_i\}_{i=1}^d=\emptyset,
  \end{equation}
  and $\sup_{x,y\in I}|S_{N(a,t)}(y)-S_{N(a,t)}(x)|<R$, the following holds: for every $J\subset I$, every $z\in J$, we have
  $$
  f_t(J)\subset \bigcup_{|r|<R+1}f_r(T^{N(z,t)}J).
  $$
  \end{lemma}
  \begin{proof} Let $J\subset I$ and $z\in J$. Then for $u\in J$ 
\begin{multline*}
|S_{N(z,t)}(z)-S_{N(z,t)}(u)|\leq 
|S_{N(a,t)}(z)-S_{N(a,t)}(u)| \\
+|S_{N(z,t)-N(a,t)}(T^{N(a,t)}z)-S_{N(z,t)-N(a,t)}(T^{N(a,t)}u)|.
\end{multline*}
  The first term above is bounded by $R$.  Moreover, 
\begin{multline*}
	  |S_{N(z,t)-N(a,t)}(T^{N(a,t)}z)-S_{N(z,t)-N(a,t)}(T^{N(a,t)}u)|\leq \\
   \sum_{0\leq i< N(z,t)- N(a,t)}  |\Phi(T^{N(a,t)+i}z)-\Phi(T^{N(a,t)+i}u)|.
\end{multline*}
  By \eqref{eq:nsin} and \eqref{eq:incl} it follows that for every $i\leq \max_{x\in I}N(x,t)-\min_{x\in I}N(x,t)$, we have $T^{N(a,t)+i}I\notin  \bigcup_{i=1}^d[-\frac{1}{R^{1+\zeta}}+\beta_i,\beta_i+\frac{1}{R^{1+\zeta}}]$. This implies that
  $$
   |\Phi(T^{N(a,t)+i}(z))-\Phi(T^{N(a,t)+i}(u))|\leq |\Phi'(\theta_{a,z})|\cdot |I|
  $$
   for some $\theta_{a,z}\in T^{N(a,t)+i}I$.
  By the definition of $V_R$ it follows that  $|\Phi'(\theta_{a,z})|\ll R^{(1+\zeta)(1+\gamma)}$. 
  Using also that $|N(z,t)-N(a,t)|<c^{-1}R/2$, we get 
  $$
  \sum_{0\leq i< N(z,t)- N(a,t)}  |\Phi(T^{N(a,t)+i}(z)-\Phi(T^{N(a,t)+i}(u)|\leq c^{-1}R R^{(1+\zeta)(1+\gamma)} R^{-4}<R^{-1}.
  $$
Summarizing, we have obtained $|S_{N(z,t)}(z)-S_{N(z,t)}(u)|<R+R^{-1}$.

We have
$$\flow_t(u)=\flow_{(S_{N(z,t)}(z)-S_{N(z,t)}(u))+(t-S_{N(z,t)}(z))}(T^{N(z,t)}u).$$
Since $T^{N(z,t)}I\in V_R$, it follows that $|t-S_{N(z,t)}(z))|<\Phi(T^{N(z,t)}(z))<2R(\log R)^{-1}$. Therefore, $|S_{N(z,t)}(z)-S_{N(z,t)}(u)+t-S_{N(z,t)}(z)|<R+1$. This finishes the proof, as $T^{N(z,t)}u\in T^{N(z,t)}J$.
  \end{proof}

\begin{lemma}\label{lem:nm2} There exists $\eta''>0$ such that the following holds: let $I=[a,b]\subset \T$ be an interval and assume that $N(\cdot, t)=N$ is constant on $I$. If $T^{N}I\in V_R$, then 
$$
\sup_{x,y\in I}|S_N(x)-S_N(y)|\leq 2R^{1-\eta''},
$$
\end{lemma}
\begin{proof} Note that 
\begin{multline*}
|S_N(x)-S_N(y)|=|(t-S_N(x))-(t-S_N(y))| \\ 
\leq |t-S_N(x)|+|t-S_N(y)|<\Phi(T^Nx)+\Phi(T^Ny)\leq 2R^{1-\eta''}.
\end{multline*}
This finishes the proof.
\end{proof}

\begin{lemma}\label{lem:boun}  For every $\varepsilon, \overline{M}$ there exists $\kappa=\kappa(\varepsilon,\overline{M})<\varepsilon$  such that for any $Z\subset \T$, with $\diam(Z)<\kappa$,  any $t>0$ any $x\in Z$ for which $\flow_t(x,0)\in \T$ and such that $T^{N(x,t)}\vert_{Z}$ is an isometry, we have: if $z\in Z$  and $|L|\leq 2\overline{M}$ satisfies $S_{N(x,t)}(x)-S_{N(x,t)}(z)\in [L,L+\varepsilon]$ then $\flow_tz\in \bigcup_{r\in [L,L+\varepsilon]}\flow_r(T^{N(x,t)}Z)$. On the other hand if $\flow_tz\in \bigcup_{r\in [L,L+\varepsilon]}\flow_r(T^{N(x,t)}Z)$ and 
$|S_{N(x,t)}(x)-S_{N(x,t)}(z)|<\overline{M}$ then
$$
S_{N(x,t)}(x)-S_{N(x,t)}(z)\in [L-\varepsilon^2,L+\varepsilon+\varepsilon^2].
$$

\end{lemma}
\begin{proof} 
	Note  that $\flow_t(z)=\flow_{t-S_{N(x,t)}(z)}(T^{N(x,t)}z,0)\subset \bigcup_{r\in [L,L+\varepsilon]}\flow_r(T^{N(x,t)}Z)$, since $t-S_{N(x,t)}(z)=S_{N(x,t)}(x)-S_{N(x,t)}(z)\in [L,L+\varepsilon]$; thus, the first part follows. 
	
	On the other hand assume  $\flow_tz=\flow_{r}(T^{N(x,t)}(y))$, for some $y\in Z$ and $r\in [L,L+\varepsilon]$. Note that 
$$
\flow_{S_{N(x,t)}(x)-S_{N(x,t)}(z)-r}(T^{N(x,t)}z)=\flow_{t-r}(\flow_{-S_{N(x,t)}(z)}(T^{N(x,t)}z))=\flow_{t-r}(z)=T^{N(x,t)}(y).
$$
Since $T^{N(x,t)}\vert_{Z}$ is an isometry, this implies that 
$$
\dist(\flow_{S_{N(x,t)}(x)-S_{N(x,t)}(z)-r}(T^{N(x,t)}z),T^{N(x,t)}(z))=\dist(y,z)<\diam(Z).
$$
However, $|S_{N(x,t)}(x)-S_{N(x,t)}(z)-r|<L+\overline{M}+\varepsilon$. Since the roof function $\Phi\geq c$ and the map $T$ is piecewise continuous and has no periodic points, there is  $\kappa=\kappa(\varepsilon,\overline{M})>0$ such that  if $\diam(Z)\leq \kappa$, then the only way the above inequality can hold is if  $|S_{N(x,t)}(x)-S_{N(x,t)}(z)-r|\leq\varepsilon^2$. Since $r\in [L,L+\varepsilon]$, the statement follows. This finishes the proof.
\end{proof}

  We will recall the definition of $(\varepsilon,k)$--uniformly stretching function, \cite{Koch} see also \cite{Fay}. 
\begin{definition}[{\cite[Definition 1]{Fay}}] Given $\varepsilon, K>0$ we  say that a function $g:[a,b]\to \R$ is  $(\varepsilon,K)$-uniformly stretching on $[a,b]$ if  $\sup_{[a,b]} g-\inf_{[a,b]}(g)>K$ and for any $\inf g\leq u\leq v\leq \sup g$, the set 
$$
I_{u,v}=\{x\in [a,b]\;:\; g(x)\in [u,v]\}
$$
has measure 
$$
(1-\varepsilon) \frac{v-u}{\sup g-\inf g} (b-a)<\lambda(I_{u,v})<(1+\varepsilon) \frac{v-u}{\sup g-\inf g} (b-a)
$$
\end{definition}

\begin{lemma}[{\cite[Lemma 2]{Fay}}] \label{lem:fay2} 
	Let $g:[a,b]\to \R$ be a $\mathscr{C}^2$ function if $\inf_{x\in [a,b]}|g'(x)| (b-a)>K$ and 
$\sup_{x\in [a,b]}|g''(x)|(b-a)\leq \varepsilon \inf_{x\in [a,b]}|g'(x)| $ then the function $g$ is $(\varepsilon, K)$-uniformly stretching on $[a,b]$.
\end{lemma}

\subsection{Good special flows over IETs}

The following proposition is the main technical result of this section. Let $(M_s)$ be a fixed increasing sequence.
  
  \begin{proposition}\label{prop:part1}If for any $\varepsilon>0$  there exists $M(\varepsilon)$ such that for every $M_s\geq M(\varepsilon)$ there is a $t(\varepsilon,M_s) >0$ satisfying the following: for any $t>t(\varepsilon,M_s)$ there exists an $\varepsilon$-almost partition $\{I_j^t=[a_j,b_j]\}$ of $\T$ with $\lim_{t\to \infty}\sup_j |I_j^t|=0$  satisfying the following:
  \begin{itemize}
  \item[(P1)] $T^n(I_j^t)\cap \{\beta_i\}=\emptyset$ for every $n\in [0, \max_{x\in I_j^t}N(x,t)]$;
 \item[(P2)] $T^n(I^t_j)\subset V_{M_s}$ for every $n\in [\min_{x\in I_j}N(x,t), \max_{x\in I_j}N(x,t)+c^{-1}M_s]$;
 \item[(P3)] $|S_{N(a_j,t)}(y)-S_{N(a_j,t)}(x)|<M_s$ for every $x,y\in I_j$;
 \item[(P4)] $\sup_{x\in I_j}|S_{N(a_j,t)}(a_j)-S_{N(a_j,t)}(x)|>\varepsilon^{100}M_s$;
 \item[(P5)]  for $n\in [\min_{x\in I_j}N(x,t), \max_{x\in I_j}N(x,t)]$, the function $S_n\Phi(\cdot)$ is monotone and $\Big(\varepsilon,\frac{M_s(d-c)}{6|I_j|}\Big)$-uniformly distributed on every $[c,d]\subset I_j$  with $|d-c|\geq  M_s^{-1/4}|I_j|$ .
  \end{itemize}
 Then the flow $\flowR$ is a good special flow. 
  \end{proposition}

  
\begin{proof}
Fix $\varepsilon>0$, and $k\in \N$. Let $M_{k,\varepsilon}:= \max\left\{M(\varepsilon^{100k}), e^{\varepsilon^{-100k}}\right\}$, where $M(\varepsilon)$ is given by the assumptions. Take $M_s\geq M_{k,\varepsilon}$; to simplify notation, we denote $M=M_s$. For $\height \geq M_{k,\varepsilon}$ let $\cR$ be any Rokhlin tower with base $B\subset \T$ and height $\height$ so that $\mu^\Phi(\cR)\geq 1-\varepsilon^2$. Since $B$ is measurable, there are disjoint intervals $B'_1,\ldots B'_s$ such that $\mu(B\triangle \bigcup_{\ell=1}^sB'_\ell)<\varepsilon^2\mu(B)$. So by considering the tower $\cR'$ over $\bigcup B'_\ell$ we will assume without loss of generality that $B$ is a union of $s$ intervals.

Let $t_{\height,M}:=t(\height^{-100k}\varepsilon^{100k},M^2+\height^2)$ where $t(\cdot,\cdot)$ comes from the assumptions. By increasing $t_{\height,M}$ if necessary we can assume that $\sup_j |I_j^t|<\min\Big\{\frac{ \varepsilon^{100k}}{s \height^{100k}M^4},\kappa(\varepsilon^{100k},M)\Big\}$ for any $t\geq t_{\height,M}$. We will also assume that $t_{\height,M}\geq t_{M/2}$, where $t_R$ comes from Lemma \ref{lem:nm1}.

Consider a $k$-tuple of times ${\bf t}$, $t_1<\ldots<t_k$, with $\Delta {\bf t}\geq t_{\height,M}$.  Let $(I_j^{t_i})$  be the $\height^{-100k}\varepsilon^{100k}$-almost partitions given by Proposition \ref{prop:part1} and consider the common refinement (i.e., all intersections) of the partitions $\{I_{j}^{t_i}\}_{i=1,\ldots k}$. Each atom of this common refinement is of the form $\cap_{i=1}^k I_{j_{\ell_i}}^{t_i}$. Let now $\overline{\cP}_{\bf t}$ be a family consisting only of those atoms in the common refinement  for which  $\Leb(\cap_{i=1}^k I_{j_{\ell_i}}^{t_i})\geq \varepsilon^{4k} \min_i \Leb(I^{t_i}_{j_{\ell_i}})$. Note that   by applying inductively Lemma \ref{lem:par} $k$ times to the partitions $\{I_{j}^{t_i}\}$ it follows that $\mu(\T\setminus \overline{\cP}_{\bf t})<\varepsilon^{10}$.
We also note that, for all $I \in \overline{\cP}_{\bf t}$, it follows from (P1) and (P2) that $T^n(I) \cap \{\beta_i\}_i = \emptyset$ for all $0 \leq n \leq c^{-1}(M^2+\height^2)$. In particular, for all such $n$, the restriction $T^n|_{I}$ is an isometry.

Let $R:= (c^{-1}\varepsilon^{-10}(M^2 + \height^2))^{1/\gamma}$ and $N:= \lfloor c^{-1}(M^2+\height^2)\rfloor$, and let $\cP_{\bf t}\subset \overline{\cP}_{\bf t}$ be the subfamily of intervals which satisfy $T^nI \cap V_R \neq \emptyset$ for all $0 \leq n < N$.
By definition, since we are removing all intervals which are entirely contained in $T^{-n}(\T \setminus V_R)$ for some  $n < N$, we have $\Leb(\T \setminus \cP_{\tb}) \leq \varepsilon^{10}+ N\Leb(\T \setminus V_R) \leq \varepsilon^{10}+ NR^{-\gamma} \leq 2 \varepsilon^{10}$. We furhter restrict to the subfamily of intervals which additionally intersect $B$.  We will now show that the partition $\cP_{\bf t}$ satisfies the assumptions (G1)--(G5) in the definition of good special flows. 

Let $P_a$ be an atom of $\cP_{\bf t}$ (note that this is an interval).   
For property (G1), we have to show that, for every $t\in \{ t_0 = 0, t_1,\ldots, t_k\}$, we have
\begin{equation}\label{eq:claim_prove_G1}
\max_{x\in P_a}\sup_{0\leq r\leq M+\height^{3/2}} \text{diam}\Big(\flow_r(T^{N(x,t)}P_a)\Big)<\varepsilon.
\end{equation}
Let us start from the case $t=t_0 = 0$.
Since $r\leq M+\height^{3/2}$, it follows that $N(\cdot,r)<c^{-1}(M+\height^{3/2})<N$. Then, by the mean-value theorem, 
\[
\diam (\flow_rP_a) \leq \Leb(P_a) \cdot \sup_{x \in P_a} \sup_{0\leq n <N} |S_n(\roof')(x)|\leq \Leb(P_a) \cdot N \cdot \sup_{x \in P_a} \sup_{0\leq n <N} |\roof'(T^nx)|.
\]
Since $T^n$ is an isometry on $P_a$, we have $\dist(T^nP_a,\{\beta_i\}) \geq R^{-(1+\zeta)} - \Leb(P_a) \geq R^{-(1+\zeta)}/2$, which implies that $|\roof'(T^nx)| \leq 4R^2$. Using again that $\Leb(P_a) \leq \sup_j \Leb(I_j^t)$, we deduce our claim $\diam (\flow_rP_a) < \varepsilon$.

The case $t=t_i$ for $i\geq 1$ is treated similarly.
Since $\cP_{\bf t}$ is a common refinement of the partitions  $\{I_{j}^{t_i}\}_{i=1,\ldots k}$, it is enough to show \eqref{eq:claim_prove_G1} for $P_a$ replaced by $I_{j}^{t_i}$.
Take $x\in I_j^t$ and denote $n=N(x,t)$. Note that $n\in [\min_{x\in I_j^t}N(x,t), \max_{x\in I_j^t}N(x,t)]$. 
Let $y,z\in T^n(I^t_j)$. Then by (P2) with $M^2+\height^2$ it follows that for every $m\leq N$, defined as above,  $T^m[y,z]\in V_{M^2+\height^2}$. In particular this implies that for every $m\leq N$, 
$$
|S_m(y)-S_m(z)|\leq \sum_{i=0}^m|\Phi(T^iy)-\Phi(T^iz)|\leq N \cdot \sup_{w\in V_{M^2+\height^2}} |\Phi'(w)| |I^t_j|\leq (M^2+\height^2)^4 |I^t_j|<\varepsilon^2,
$$
by the bound on $|I_j^t|$. Using this for $m=N(y,r)$ it follows that $|S_{N(y,r)}(y)-S_{N(y,r)}(z)|<\varepsilon^2$. This implies (G1). Moroever, we claim that (which will be useful in showing (G3))
\begin{equation}\label{eq:added}
 \text{ for any }x\in I_j^t\text{ any }r,r'\in [-M_s,M_s],\;\;\; \flow_r(T^{N(x,t)}I_j^t)\cap \flow_{r'}(T^{N(x,t)}I_j^t)= \emptyset.
\end{equation}
Indeed, this just follows from the fact that, by assumptions, the set $T^{N(x,t)}I_j^t$ is an interval of length going to $0$ with $t$ (since $T^{N(x,t)}$ is an isometry on it) and so, for any $y\in T^{N(x,t)}I_j^t$, the minimum over all $|\bar{r}|\leq 2M_s$ of the horizontal distances of $\flow_{\bar{r}} y$ to $y$ is bounded below by a function depending on $M_s$ and in particular is $\geq 2|I_j^t|$ for large enough $t$.

 For property (G2), we have 
$$  
\mu\Big(\bigcup_{a\in \mathscr{A}}P_a\triangle B\Big)= \mu\Big(B\setminus \bigcup_{a\in \mathscr{A}}P_a\Big)+ \mu\Big(\bigcup_{a\in \mathscr{A}}P_a\setminus B\Big).
$$
Since $B$ is a union of $s$ intervals and $\{P_a\}$ are atoms which intersect $B$, $\mu(\bigcup_{a\in \mathscr{A}}P_a\setminus B)\leq 2s \cdot \sup_j |I_j^{t_k}|<\frac{2 \varepsilon^{100k}}{\height^{100k}M^4}$. Moreover $\mu\Big(B\setminus \bigcup_{a\in \mathscr{A}}P_a\Big)\leq \mu(\T\setminus \overline{\cP}_{\bf t})<\varepsilon^{10}$. This shows (G2).

 For property (G3) we  notice that, for every $i=1,\ldots, k$, $P_a\subset I^{t_i}_{j_a}$ (as the partition was the common refinement).
We will use Lemma \ref{lem:nm1} for $R=M$ and $I=I^{t_i}_{j_a}$ where $i=1,\ldots k$ and $J=P_a\subset I$. We will check that the assumptions of the lemma are satisfied and, for simplicity, denote $t=t_i$. Note that $|I|<M^{-4}$ by the definition of $t_{\height,M}$. Moreover (P1) and (P2) together immediately imply \eqref{eq:nsin} (remember that $N(x,t)\leq c^{-1}t$ as $\Phi\geq c$). Note also that,
by P3, 
\begin{multline*}
|S_{N(x,t)-N(a_j,t)}(T^{N(a_j,t)}(x))|=|S_{N(x,t)}(x)-S_{N(a_j,t)}(x)| \\ \leq 
|S_{N(x,t)}(x)-t| +|S_{N(a_j,t)}(a_j)-t|+|S_{N(a_j,t)}(a_j)-S_{N(a_j,t)}(x)|\leq 3M, 
\end{multline*}
as $|S_{N(x,t)}(x)-t|\leq \Phi(T^{N(x,t)}(x))\leq M$ by (P2). This implies that for any $x\in P_a\subset I^t_{j_a}$,
$$
|N(x,t)-N(a_j,t)|\leq 3c^{-1}M;
$$
therefore, $|N(x,t)-N(y,t)|<6c^{-1}M$. Finally,  (P3) implies the bound on the difference of ergodic sums in  Lemma \ref{lem:nm1} and so all assumptions of the lemma are satisfied. Finally by \eqref{eq:added} the set on the RHS is indeed a tower. Therefore (G3) follows.

 For property (G4), we will apply Lemma \ref{lem:boun}. For this, we will first show that there exists $x\in P_a$ and $i\in \{1,\ldots, k\}$ such that $\flow_{t_i}(x,0)\in \T$ and such that $T^{N(x,t_i)}\vert_{P_a}$ is an isometry. By the definition of $P_a=[c,d]$, there is $i\in \{1,\ldots, k\}$ such that  $\Leb([c,d])\geq \varepsilon^{4k} |I_{j_a}^{t_i}|$. Let $I_{j_a}^{t_i}=[a,b]=I$, and we denote $t=t_i$. 
 We claim that 
 \begin{equation}\label{eq:ml1}
|S_{N(c,t)}(c)-S_{N(c,t)}(d)| \leq M+1.
 \end{equation}
Indeed,  from (P3), we get
 \begin{multline*}
 |S_{N(c,t)}(c)-S_{N(c,t)}(d)|\\ \leq
 |S_{N(c,t)}(c)-S_{N(c,t)}(d)- (S_{N(a,t)}(c)-S_{N(a,t)}(d))|+|S_{N(a,t)}(c)-S_{N(a,t)}(d)|\\ \leq
 M + |S_{N(c,t)-N(a,t)}(T^{N(a,t)}c)-S_{N(c,t)-N(a,t)}(T^{N(a,t)}d)|.
 \end{multline*}
  By (P2), it follows that $T^{N(a,t)+i}(I)\subset V_{M}$ for every $i\in \{0,\ldots, N(c,t)-N(a,t)\}$. Moreover, (P2) and the bound $|N(x,t)-N(y,t)|<6c^{-1}M$ yield 
  \[
  |S_{N(c,t)-N(a,t)}(T^{N(a,t)}c)-S_{N(c,t)-N(a,t)}(T^{N(a,t)}d)|\leq 6c^{-1}M \sup_{x\in V_M} |\Phi'(x)|(d-c)\leq 1,
  \]  
  since $\sup_{x\in V_M} |\Phi'(x)|<M^2$ and $d-c\leq b-a\leq M^{-4}$. This proves the claim \eqref{eq:ml1}.

By (P5), we know that  $S_{N(x,t)}(\cdot)$ is $(\varepsilon, M/6)$-uniformly stretching (and monotone) on $[a,b]$. This means in particular that if $J\subset [a,b]$ is an interval, then 
$S_{N(c,t)}(J)$ is an interval with length at least $\frac{M}{6} (1-\varepsilon)\Leb(J)(b-a)^{-1}$. Applying this fact to $J=[c,d]$, along with monotonicity of $S_{N(c,t)}(\cdot)$, we get by \eqref{eq:ml1} that 
$ M+1>|S_{N(c,t)}(c)-S_{N(c,t)}(d)|\geq \varepsilon^{4k} M/7$. The lower bound and (P2)  by Lemma \ref{lem:nm2} with $R=M$ imply that $N(\cdot,t)$ is not constant on $[c,d]$ (recall that $M$ is sufficiently large in terms of $\varepsilon^{-1}$). This however means that there exists $u\in [c,d]$ such that $f_{t}(u,0)=(T^{N(u,t)}u,0)$, i.e. $S_{N(u,t)}(u)=t$.
 We have 
\begin{multline*}
|S_{N(u,t)}(u)-S_{N(u,t)}(x)- (S_{N(c,t)}(u)-S_{N(c,t)}(x))| \\ \leq |S_{N(u,t)-N(c,t)}(T^{N(c,t)}u)-S_{N(u,t)-N(c,t)}(T^{N(c,t)}x)|.
\end{multline*}
Analogously to the reasoning to prove \eqref{eq:ml1}, we get that 
\[
|S_{N(u,t)-N(c,t)}(T^{N(c,t)}u)-S_{N(u,t)-N(c,t)}(T^{N(c,t)}x)|<1.
\]
Moreover by monotonicity of $S_{N(c,t)}(\cdot)$ it follows that  $|S_{N(c,t)}(u)-S_{N(c,t)}(x)|<|S_{N(c,t)}(c)-S_{N(c,t)}(d)|<M+1$. On the other hand $\varepsilon^{4k}M/7\leq |S_{N(c,t)}(c)-S_{N(c,t)}(d)|\leq |S_{N(c,t)}(u)-S_{N(c,t)}(d)|+|S_{N(c,t)}(u)-S_{N(c,t)}(d)|$, which implies that $\max_{x\in [c,d]}|S_{N(c,t)}(u)-S_{N(c,t)}(x)|\geq  \varepsilon^{4k}M/14$. Summarizing, we obtained
$$
\varepsilon^{4k}M/14< \max_{x\in [c,d]} |S_{N(u,t)}(u)-S_{N(u,t)}(x)|< M+2.
 $$
We will now apply Lemma \ref{lem:boun} with $\overline{M}= M$, $Z=P_a=[c,d]$ and $x=u$ (recall that $\flow_t(u,0)\in X$). By this lemma, for any $|L|\leq 2M$, the set $
\{x\in P_a: f_tx \in \bigcup_{r\in [L,L+\varepsilon]} f_r(T^{N(u,t)}P_a)\}$ contains the set $\{x\in P_a: |S_{N(u,t)}(u)-S_{N(u,t)}(x)|\in [L,L+\varepsilon]\}$ and is contained in the set $\{x\in P_a: |S_{N(u,t)}(u)-S_{N(u,t)}(x)|\in [L-\varepsilon^2,L+\varepsilon+\varepsilon^2]\}$. By (P5), the function $S_{N(u,t)}(\cdot)$ is $(\varepsilon, \varepsilon^{4k}M/6)$-uniformly stretching on $P_a=[c,d]$ and monotone it follows that the measure of the first set is at least 
$$
(1-\varepsilon) \frac{L+\varepsilon-L}{S_{N(u,t)}(c)-S_{N(u,t)}(d)} (d-c).
$$
On the other hand the measure of the latter set is at most 
$$
(1+\varepsilon) \frac{L+\varepsilon+\varepsilon^2- L}{S_{N(u,t)}(c)-S_{N(u,t)}(d)} (d-c).
$$ 
By monotonicity of $S_{N(u,t)}(\cdot)$ and Lemma \ref{lem:boun}, it follows that  $S_{N(u,t)}(c)=S^-({a,t,u})$ and $S_{N(u,t)}(d)S=^+({s,t,u})$. Finally, $S_{N(u,t)}(d)-S_{N(u,t)}(c)\geq \max_{x\in [c,d]} |S_{N(u,t)}(u)-S_{N(u,t)}(x)|\geq \varepsilon^{4k}M/14$. This indeed finishes the proof of (G5) and completes the proof of \Cref{prop:part1}.
\end{proof}

The main mechanism in the proof of \Cref{prop:part1} is the mixing via stretching mechanism. More precisely, we have the following result.

\begin{lemma}\label{lem:mvs} Assume that for any $\varepsilon>0$ there is $t_\varepsilon>0$ such that for every $t>t_\varepsilon$, there exists an $\varepsilon$-almost partition $\cP_t=\{J^t_i\}$ of $\T$ satisfying the following properties
\begin{enumerate}
\item[(J1)] $\sup_i|J_i^t| \to 0$ and $T^n(J_i)\cap \{\beta_i\}_{i=1}^d=\emptyset$ for every $n\in [0, \max_{x\in J_i}N(x,t)]$;
\item[(J2)] $\min_{x,y\in J_i} |S_{N(x,t)}(\Phi')(y)| |J_i|\geq K(t)$;
\item[(J3)] $\max_{x,y\in J_i} |S_{N(x,t)}(\Phi'')(y)| |J_i|\leq K(t)^{-1} \min_{x,y\in J_i} |S_{N(x,t)}(\Phi')(y)|$.
\end{enumerate}
for some function $K(t)\to \infty$, as $t\to \infty$. Then (P1)--(P5) in \Cref{prop:part1} hold. 
\end{lemma}
Properties (J1)--(J3) are the core of the mixing via shearing argument: it asserts that there is a family of curves (intervals) that almost foliate the whole space and whose length goes to $0$, (J1). Moreover these curves are stretched (sheared) (J2) and the shear is uniform~(J3).

\begin{proof}
Given $\varepsilon$, we will consider $(M_k)$  where $k$ is sufficiently large in terms of $\varepsilon$. For sufficiently large $t$ in terms of $M=M_k$ we will construct the $\varepsilon$-almost partition $\{I_j^t\}$ satisfying (P1)-(P5) by refining the partition $\cP_t=\{J^t_i\}$. Since $t$ is fixed, we will omit it from the notation. Let $\cP_t=\{J_i\}$ be the $\varepsilon^{100}$-almost partition coming from the assumptions, and let $J=[a,b]\in \cP_t$. Let $a_1>a$ be the smallest such that $|N(a_1,t)-N(a,t)|=M/2$. Define $I^J_1=[a,a_1]$; now inductively we define $a_{i+1}>a_i$ to be the smallest for which $|N(a_{i+1},t)-N(a_i,t)|=M/2$ and $I^J_{i+1}=[a_i,a_{i+1}]$. 
We continue until for some $i$ we have $|N(a_i,t)-N(x,t)|<M/2$ for all $x\in [a_i,b]$ and then we take $I^J_b=[a_i,b]$. The collection $\{I^J_i\}$ and $I^J_b$  forms a partition of $J$, and so the union over all $J\in \cP_t$ of these new intervals form an $\varepsilon^{100}$-almost partition of $\T$. We will discard those intervals $\{I^J_i\}$ which go to close to singularity in a short time after $t$. More precisely, let $Z_M^\Phi=\{(x,s): T^ix\in V_{M/2} \text{ for every } i\in [-2\varepsilon^{-1}M,2\varepsilon^{-1}M] \}$. Note that $\mu(Z_M^\Phi)\geq 1-\varepsilon^{100}$ if $M$ is large enough, since $\Phi\in L^1(\T)$. Moreover, let
\[
U_M:=\{x\in \T\, :\, (1-\varepsilon^2)n < S_n(x) < (1+\varepsilon^2)n \text{ for } n\geq M/4\} \ \ \text{and} \ \ U_M^\Phi=\{(x,s)\, :\, x\in U_M\}.
\]
By the ergodic theorem it follows that $\mu^\Phi(U_M^\Phi)\geq 1-\varepsilon^{100}$ if $M$ is large enough.

We say that  $I_i^J$ is {\em good} if $\bigcup_{r<\varepsilon}\flow_r(I_i^J)\cap (Z_M^\Phi\cap \flow_{-t}U_M^\Phi)\neq \emptyset$. Since the intervals $\{I_i^j\}$ form an $\varepsilon^{100}$-almost partition it follows that $\mu^\Phi\Big(\bigcup_{i,J}\bigcup_{r<\varepsilon}f_r(I_i^J)\Big)\geq \varepsilon(1-\varepsilon^{100})$. Thus, good intervals form an $\varepsilon^{50}$-almost partition of $X$.

 Note that each interval $J$ has one special interval $I^J_b$ where the function $N(\cdot,t)$  varies by less than $M/2$ and in fact might be constant. We have the following:
\begin{equation}\label{eq:bad}
\text{ if } I^J_b\text{ is good then }|I^J_b|\leq \varepsilon^{100} |J|
\end{equation}
Before we prove the above statement, let us show how it implies (P1)--(P5). The partition consists of good interval $\{I^J_i\}$ where we additionally discard all good intervals of the form $I^J_b$. By \eqref{eq:bad} and the above estimates on the measure of good intervals, it follows that the collection $\{I_i^t\}$ consisting of all good $\{I^J_i\}$ over all $J$ forms an $\varepsilon^{10}$-almost partition of $\T$. Note that (P1) follows from the fact that $I_i=I^J_s\subset J$ for some $J$ and we use (J1). 

For (P2) we note that since $I_i$ s good it follows that there exists $x\in I_i$ and $r<\varepsilon$ such that $\flow_{t+r}(x)\in Z_M^\Phi\cap U_M^\Phi$. By the definition, this means that $T^{N(x,t)+i}x\in V_{M/2}$ for $i\in  [-\varepsilon^{-1}M,\varepsilon^{-1}M]$. Since by definition $|N(x,t)-N(a,t)|\leq M$  for every $x\in I_i$, we have that  $T^nx\in V_{M/2}$ for every $n\in [\min_{x\in I_i}N(x,t), \max_{x\in I_i}N(x,t)+c^{-1}M]$. Moreover, since $r<\varepsilon$ this also means that $T^{w}x\in U_M$ for $w\in \{N(x,t)-1,N(x,t)\}$. In both cases, $|S_n(T^w(x))-S_n(T^{N(x,t)}(x))|<\varepsilon^2 M$ for $n\geq \varepsilon M$. In particular, for $n \geq \varepsilon M$,
\begin{equation}\label{eq:bet2}
S_n(T^{N(x,t)}(x))\in [(1-\varepsilon)n, (1+\varepsilon)n].
\end{equation}
 Since $I_i\to 0$ as $t\to \infty$ (as $I_i\subset J^t_i$), it follows that if $t$ is large enough then $T^n(I_i)\subset V_M$, by (J1). This gives (P2). 
 
 In order to verify (P3), let $I_i=[c,d]$. Notice that 
\begin{equation}\label{eq:ct}
S_{N(c,t)}(x)-S_{N(c,t)}(c)= S_{N(c,t)-N(x,t)}(T^{N(x,t)}x)+(t-S_{N(c,t)}(c))+(S_{N(x,t)}(x)-t).
\end{equation}
By (P2), we have $|t-S_{N(c,t)}(c)|<\Phi(T^{N(c,t}c)\leq \varepsilon^{100}M$ with the same bound for $S_{N(x,t)}(x)-t$.  By \eqref{eq:bet2}, \[
S_{N(c,t)-N(d,t)}(T^{N(x,t)}x)\geq S_{M/2}(T^{N(x,t)}x)\geq (1-\varepsilon)M/2,
\]
and, also by (P2), 
\[
S_{N(c,t)-N(x,t)}(T^{N(x,t)}x)\leq S_{M}(T^{N(x,t)}x)\leq (1+\varepsilon)M.
\]
Since $|N(c,t)-N(d,t)=M/2$ and $|N(c,t)-N(x,t)|\leq M/2$, the above inequalities immediately imply (P3) and (P4). 

Let us verify (P5), for which we will use Lemma \ref{lem:fay2}. Let  $x_0\in [c,d]$ be such that $\inf_{x\in [c,d]} |S_{N(c,t)}(\Phi')(x)|=|S_{N(c,t)}(\Phi')(x_0)|$. Since $|S_{N(c,t)}(\Phi)(d)-S_{N(c,t)}(\Phi)(c)|\geq M/2$ it follows that one of $|S_{N(c,t)}(\Phi)(x_0)-S_{N(c,t)}(\Phi)(c)|$ and $|S_{N(c,t)}(\Phi)(x_0)-S_{N(c,t)}(\Phi)(d)|$ is at least $M/4$. Without loss of generality, we assume it is the first one; then,
\begin{multline*}
M/4\leq |S_{N(c,t)}(\Phi)(x_0)-S_{N(c,t)}(\Phi)(c)|\leq |S_{N(c,t)}(\Phi')(x_0)|(x_0-c)+ |S_{N(c,t)}(\Phi'')(\theta)|(x_0-c^2) \\ \leq (1+K(t)^{-1})|S_{N(c,t)}(\Phi')(x_0)|(x_0-c).
\end{multline*}
This implies that $\inf_{x\in [c,d]} |S_{N(c,t)}(\Phi')(x)| |I_i|\geq M/5$ and, furthermore, using (P2), we get $\inf_{x\in [c,d]} |S_{N(y,t)}(\Phi')(x)|\geq \inf_{x\in [c,d]} |S_{N(c,t)}(\Phi')(x)|+{\rm O}(M^2)$. Summarizing, 
\[
\inf_{x,y\in [c,d]} |S_{N(y,t)}(\Phi')(x)| |I_i|\geq M/6.
\]
Moreover, if $[u,w]\subset [c,d]$, then $\inf_{x,y\in [u,w]} |S_{N(y,t)}(\Phi')(x)|(w-u)\geq \inf_{x,y\in [c,d]} |S_{N(y,t)}(\Phi')(x)| |I_i|\cdot \frac{w-u}{|I_i|}\geq \frac{M}{6}\cdot \frac{w-u}{|I_i|}$. Finally notice that (J3) holds in particular on subintervals of $J_i$ and so, in particular, it holds on $I_i$ or the interval $[u,w]\subset I_i$. This implies (P5). \\

It remains to show \eqref{eq:bad}. Let $I^J_b=[u,v]$. Then analogoulsy to \eqref{eq:ct} and the reasoning below it, 
$$
|S_{N(u,t)}(x)-S_{N(u,t)}(u)- S_{N(u,t)-N(x,t)}(T^{N(x,t)}x)|<2\varepsilon^{100}M,
$$
where $|S_{N(u,t)-N(x,t)}(T^{N(x,t)}x)|<(1+\varepsilon)M/2$ by (P2) (see the reasoning below \eqref{eq:ct}). Then, the mean value theorem implies that $\inf_{x\in[u,v]}|S_{N(u,t)}(\Phi')(x)| (v-u)\leq M$. This means that $|I^J_b|\leq \Big(\min_{x\in[u,v]}|S_{N(u,t)}(\Phi')(x)|\Big)^{-1}M$. 
\sloppy On the other hand $|J|\geq K(t)\Big(\min_{x\in[u,v]}|S_{N(u,t)}(\Phi')(x)|\Big)^{-1}$. If $t$ is sufficiently large in terms of $M$, then $\varepsilon^{100}K(t)\geq M$. This finishes the proof of \eqref{eq:bad}.
\end{proof}

It remains to show that the assumptions of \Cref{lem:mvs} hold in the classes of mixing  flows under consideration. For most examples, the conditions of \Cref{lem:mvs} have already been shown to hold and we will just quote and comment below. The one exception are Kochergin flows. In \cite{Koch}, the author shows a result of similar flavor, but the partitions are constructed for a fixed length of ergodic sums. We will therefore show how to construct these partitions for Kochergin flows in a separate subsection (see Section \ref{sec:koch}). This is one of the more technical parts of the paper. Before we do that let us comment on existence of stretching partitions for other flows on surfaces (with logarithmic singularities).

\subsection{Mixing flows on surfaces.}

In this short subsection, we discuss proofs of \Cref{thm:arn} and \Cref{prop:ChW}. We need to show that in those cases the assumptions of \Cref{lem:mvs} are fulfilled, i.e. there are partitions $\{I_j^t\}$ which satisfy (J1)--(J3).
Let us first consider Arnol'd flows that were considered in \cite{Ulc} and \cite{Rav}. 
 Properties (J1)--(J3) follow from 4 in Proposition 5 and Proposition 6 in \cite{Ulc} (see also Lemma 6.5. and Proposition 6.4. in \cite{Rav}). This proves \Cref{thm:arn}. For \Cref{prop:ChW}, Properties (J1)--(J3) are shown to hold in Theorem 6.1. of \cite{ChW}.\\

\section{QLUS flows on surfaces}\label{sec:QLUS_on_surfaces}

In this section we will make the arguments from the previous section quantitative. The goal will be to give conditions depending on ergodic sums of derivatives, that will guarantee the (polynomial) QLUS property and, as a consequence, also polynomial mixing of higher orders. We start with a quantitative version of \Cref{lem:boun}. We will assume in this section that the special flow $\flowR$ satisfies a diophantine condition, i.e. for every $\eta>0$ there exists $\eta'>0$ such that for any sufficiently large $t$,
\begin{equation}\label{eq:dioph}
	\sup_{x\in X^\Phi}\sup_{t^{-2\eta}<r<4t^{\eta}} \dist(\flow_rx,x)>t^{-\eta'}.
\end{equation}
The above condition is satisfied for example for special flows over a diophantine irrational rotation.

\begin{lemma}\label{lem:qboun} For every $t>0$ sufficiently large, for any $Z\subset X$ with $\diam(Z)\leq t^{-\eta'}$, and for any $x\in Z$ for which $\flow_t(x,0)\in \T$ and such that $T^{N(x,t)}\vert_{Z}$ is an isometry, we have: if $z\in Z$, $|L|\leq t^{\eta}$ and $\varepsilon\geq t^{-\eta}$ satisfies $S_{N(x,t)}(x)-S_{N(x,t)}(z)\in [L,L+\varepsilon]$ then $\flow_tz\in \bigcup_{r\in [L,L+\varepsilon]}\flow_r(T^{N(x,t)}Z)$. On the other hand if $\flow_tz\in \bigcup_{r\in [L,L+\varepsilon]}\flow_r(T^{N(x,t)}Z)$ and 
	$|S_{N(x,t)}(x)-S_{N(x,t)}(z)|<t^{\eta'}$ then
	$$
	S_{N(x,t)}(x)-S_{N(x,t)}(z)\in [L-\varepsilon^2,L+\varepsilon+\varepsilon^2].
	$$
\end{lemma}
\begin{proof} The proof follows the same lines as the proof of Lemma \ref{lem:boun}. 
	Note  that $\flow_t(z)=\flow_{t-S_{N(x,t)}(z)}(T^{N(x,t)}z,0)\subset \bigcup_{r\in [L,L+\varepsilon]}\flow_r(T^{N(x,t)}Z)$, as $t-S_{N(x,t)}(z)=S_{N(x,t)}(x)-S_{N(x,t)}(z)\in [L,L+\varepsilon]$ and so the first part follows. On the other hand, assume  $\flow_tz=f_{r}(T^{N(x,t)}(y))$, for some $y\in Z$ and $r\in [L,L+\varepsilon]$. Note that 
	$$
	\flow_{S_{N(x,t)}(x)-S_{N(x,t)}(z)-r}(T^{N(x,t)}z)=\flow_{t-r}(\flow_{-S_{N(x,t)}(z)}(T^{N(x,t)}z))=\flow_{t-r}(z)=T^{N(x,t)}(y).
	$$
	Since $T^{N(x,t)}\vert_{Z}$ is an isometry, this implies that 
	$$
	\dist(\flow_{S_{N(x,t)}(x)-S_{N(x,t)}(z)-r}(T^{N(x,t)}z),T^{N(x,t)}(z))=\dist(y,z)\leq \diam(Z)<t^{-\eta'}.
	$$
	However, $|S_{N(x,t)}(x)-S_{N(x,t)}(z)-r|<L+t^{\eta}+2t^{-\eta}<3t^{\eta}$. By \eqref{eq:dioph} the only way the above inequality can hold is if  $|S_{N(x,t)}(x)-S_{N(x,t)}(z)-r|\leq t^{-2\eta}$. Since $r\in [L,L+\varepsilon]$, the statement follows. This finishes the proof.
\end{proof}

\begin{proposition}\label{prop:qpart1} If there exists $0<\xi<1/100$ such that, for any $t$ sufficiently large and any $1\leq M\leq \height\leq  t^\xi$, there are finitely many  disjoint Rokhlin towers $\cR_1\ldots, \cR_d\subset \T^\Phi$ with bases (which are intervals) $B_1,\ldots, B_d\subset \T$ satisfying $ |B_i|\in [\height^{-1-\xi},\height^{-1+\xi}]$ and with $\mu^\Phi(\T^\Phi\setminus \bigcup_i\cR_i)<M^{-\xi}$ and, moreover, for every $t$ sufficiently large there exists a $t^{-10\xi}$-almost partition  $\{I_j^t=[a_j,b_j]\}$ of $\T$ with $\lim_{t\to \infty}\sup_j t^{1-10\xi} |I_j^t|=0$  satisfying the following:
	\begin{itemize}
		\item[(QP1).] $T^n(I_j^t)\cap \{\beta_i\}=\emptyset$ for every $n\in [0, \max_{x\in I_j^t}N(x,t)]$;
		\item[(QP2)] for $D\in\{M,t^\xi\}$, we have $T^n(I^t_j)\subset V_{D}$ for every $n\in [\min_{x\in I_j}N(x,t), \max_{x\in I_j}N(x,t)+c^{-1}D]$;
		\item[(QP3)] $|S_{N(a_j,t)}(y)-S_{N(a_j,t)}(x)|<M$ for every $x,y\in I_j$;
		\item[(QP4)] $\sup_{x\in I_j}|S_{N(a_j,t)}(a_j)-S_{N(a_j,t)}(x)|>M^{1-\xi}$;
		\item[(QP5)] \sloppy for $n\in [\min_{x\in I_j}N(x,t), \max_{x\in I_j}N(x,t)]$ the function $S_n\Phi(\cdot)$ is monotone and $\Big(M^{-1},\frac{M^{1-\xi}(d-c)}{6|I_j|}\Big)$-uniformly stretching on every $[c,d]\subset I_j$  with $|d-c|\geq  M^{-1/4}|I_j|$;
	\end{itemize}
	then, the flow $\flowR$ is a quantitatively good special flow. 
\end{proposition}
Notice that the two parts, namely the one concerning Rokhlin towers and the one on conditions (QP1)--(QP5), are independent. We can apply the two parts of the proposition at different times. This is what we will do in the proof.
\begin{proof}
	Let us define $\varepsilon_k=\frac{\xi}{100^{k^2}}$. Given $1<M\leq \height$, we apply the first part of the assumptions (the one on Rokhlin towers) for $t, \height, M$: let $\{\cR_i\}$ be the corresponding Rokhlin towers. Then, if $h_i$ is the height of $\cR_i$, either $h_i |B_i|\geq \height^{-\xi}$, or we can discard the tower $\cR_i$ from the collection. If $h_i|B_i|\geq \height^{-\xi}$ then $h_i\in [\height^{1-2\xi},\height^{1+\xi}]$, and so these towers satisfy the corresponding property in the definition of quantitatively good special flows.
	
	Consider a $k$-tuple of times ${\bf t}$, $t_1<\ldots<t_k$, with $\Deltat^{\varepsilon_k}\geq \height$.  In fact we will only use the fact that $t_1^{\frac{\xi}{100k^2}}\geq \height$. Let $(I_j^{t_i})$  be the $t_i^{-10\xi}$-almost partitions given by the second part of the assumptions of Proposition \ref{prop:qpart1} applied to $t_i$ with $i=1,\ldots, k$.  Consider the common refinement (i.e., all intersections) of the partitions $\{I_{j}^{t_i}\}_{i=1,\ldots k}$. Each atom of this common refinement is of the form $\cap_{i=1}^k I_{j_{\ell_i}}^{t_i}$. Let now $\overline{\cP}_{\bf t}$ be a family consisting only of those atoms in the common refinement  for which  $\Leb(\cap_{i=1}^k I_{j_{\ell_i}}^{t_i})\geq M^{-\xi'} \min_i \Leb(I^{t_i}_{j_{\ell_i}})$, where $\xi'=\xi'(\xi,k)$ is sufficiently large in terms of $\xi$ and $k$. Note that  by applying inductively Lemma \ref{lem:par} $k$ times to the partitions $\{I_{j}^{t_i}\}$ it follows that $\mu(\T\setminus \overline{\cP}_{\bf t})<M^{-4k\xi}$ (using the fact that $\xi'$ is sufficiently large in terms of $\xi$ and $k$). 
	Furthermore, for all $I \in \overline{\cP}_{\bf t}$, it follows from (QP1) and (QP2) that $T^n(I) \cap \{\beta_i\}_i = \emptyset$ for all $0 \leq n \leq 2\height^2$. In particular, for all such $n$, the restriction $T^n|_{I}$ is an isometry.

	Let $R:= (2c^{-1}\height^{2+4k\xi})^{1/\gamma}$ and $N= \lfloor 2c^{-1}\height^2\rfloor$, and let $\cP_{\bf t}\subset \overline{\cP}_{\bf t}$ be the subfamily of intervals which satisfy $T^nI \cap V_R \neq \emptyset$ for all $0 \leq n < N$.
	By definition, since we are removing all intervals which are entirely contained in $T^{-n}(\T \setminus V_R)$ for some  $n < N$, we have $\Leb(\T \setminus \cP_{\tb}) \leq M^{-4k\xi}+ N\Leb(\T \setminus V_R) \leq M^{-4k\xi}+ NR^{\gamma} \leq 2 M^{-4k\xi}$. We furhter restrict to the subfamily of intervals which additionally intersect at least one of the $B_i$, $i=1,\ldots,d$.  We will now show that the partition $\cP_{\bf t}$ satisfies the assumptions (QG1)--(QG4) in the definition of good special flows. 
		
		Let $P_a$ be an atom of $\cP_{\bf t}$, which is an interval.   
		For property (QG1), we have to show that, for every $t\in \{ t_0 = 0, t_1,\ldots, t_k\}$, we have
		\begin{equation}\label{eq:claim_prove_QG1}
			\max_{x\in P_a}\sup_{0\leq r\leq 2\height^{3/2}} \text{diam}\Big(\flow_r(T^{N(x,t)}P_a)\Big)<M^{-\varepsilon_k}.
		\end{equation}
		Let us start from the case $t=t_0 = 0$.
		Since $r\leq 2\height^{3/2}$, it follows that $N(\cdot,r)<2c^{-1}\height^{3/2}<N$. Then, by the mean-value theorem, 
		\[
		\diam (\flow_rP_a) \leq \Leb(P_a) \cdot \sup_{x \in P_a} \sup_{0\leq n <N} |S_n(\roof')(x)|\leq \Leb(P_a) \cdot N \cdot \sup_{x \in P_a} \sup_{0\leq n <N} |\roof'(T^nx)|.
		\]
		Since $T^n$ is an isometry on $P_a$, we have $\dist(T^nP_a,\{\beta_i\}) \geq R^{-(1+\zeta)} - \Leb(P_a) \geq R^{-(1+\zeta)}/2$, which implies that $|\roof'(T^nx)| \leq 4R^2$. Using again that $\Leb(P_a) \leq \sup_j \Leb(I_j^t) < t^{-1+20\xi}$, we deduce our claim $\diam (\flow_rP_a) < M^{-\varepsilon_k}$.

	The case $t=t_i$ for $i\geq 1$ is treated similarly.
		Since $\cP_{\bf t}$ is a common refinement of the partitions  $\{I_{j}^{t_i}\}_{i=1,\ldots k}$, it is enough to show \eqref{eq:claim_prove_QG1} for $P_a$ replaced by $I_{j}^{t_i}$.
	Take $x\in I_j^t$ and denote $n=N(x,t)$. Note that $n\in [\min_{x\in I_j^t}N(x,t), \max_{x\in I_j^t}N(x,t)]$. Since $r\leq 2\height^2$ it follows that $N(\cdot,r)<2c^{-1}\height^2$. Let $y,z\in T^n(I^t_j)$. Then by (QP2) we have that $T^m[y,z]\notin V_{t^\xi}$ for every $m\leq c^{-1}t^\xi$. In particular, this implies that for every $m\leq c^{-1}t^{\xi}$
	$$
	|S_m(y)-S_m(z)|\leq \sum_{i\leq m}|\Phi(T^iy)-\Phi(T^iz)|\leq c^{-1}t^{\xi} \sup_{w\in V_{t^\xi}} |\Phi'(w)| |I^t_j|\leq c^{-1}t^{\xi}t^{(1+\gamma)\xi} |I^t_j|<t^{-\xi},
	$$
	by the bound  $|I_j^t|<t^{-1+10\xi}$. Using this for $m=N(y,r)$, we obtain $|S_{N(y,r)}(y)-S_{N(y,r)}(z)|<t^{-\xi}$, which implies (QG1).

Moreover we claim that (which will be useful in proving (QG3)),	
	\begin{equation}\label{eq:qadded}
 \text{ for any }x\in I_j^t\text{ any }r,r'\in [-M,M],\;\;\; \flow_r(T^{N(x,t)}I_j^t)\cap \flow_{r'}(T^{N(x,t)}I_j^t)= \emptyset.
\end{equation}
Indeed, this just follows from the fact that, by assumptions, the set $T^{N(x,t)}I_j^t$ is an interval of length $t^{-1+10\xi'}$ (since $T^{N(x,t)}$ is an isometry on it) and so, for any $y\in T^{N(x,t)}I_j^t$, the minimum over all $|\bar{r}|\leq 2M$ of the horizontal distances of $\flow_{\bar{r}} y$ to $y$ is bounded below by $M^{-1-\xi}$ and in particular is $\geq 2|I_j^t|$ for large enough $t$.

For property (QG2), for every $B=B_i$ with $i=1,\ldots, d$, we have that $B$ is an interval of length $\geq \height^{-1-\xi}$ and $\{P_a\}$ are also intervals of length $\leq t^{-1+10\xi}$. Therefore 
\begin{multline*}
	\mu\Big(\bigcup \{P_a : P_a\cap B \neq \emptyset\}\triangle B\Big) \leq 	\mu\Big(\bigcup_{a\in \mathscr{A}}(P_a\cap B)\triangle B\Big)+2 \cdot \sup_j |I_j^{t_k}| \\ \leq \mu\Big(B\setminus \bigcup_{a\in \mathscr{A}}(P_a\cap B)\Big)+ \mu\Big(\bigcup_{a\in \mathscr{A}}(P_a\cap B)\setminus B\Big)+2 \cdot \sup_j |I_j^{t_k}|
\end{multline*}
The second summand is trivially equal to $0$. 
	Moreover, since, $\mu(B \setminus \bigcup_{a\in \mathscr{A}}(P_a\cap B))\leq M^{-4k\xi} + 2 \cdot \sup_j |I_j^{t_k}|< M^{-4k\xi} + 2t^{-1+10\xi}$, we get $\mu\Big(\bigcup \{P_a : P_a\cap B \neq \emptyset\}\triangle B\Big)< M^{-4k\xi} + 4t^{-1+10\xi}<M^{-\varepsilon_k}$.
This shows (QG2).

	For property (QG3), we notice that for every $i=1,\ldots, k$, $P_a\subset I^{t_i}_{j_a}$ (as the partition was the common refinement).
	We will use Lemma \ref{lem:nm1} for $R=M$, $I=I^{t_i}_{j_a}$ where $i=1,\ldots k$ and $J=P_a\subset I$. We will check that the assumptions of the lemma are satisfied and for simplicity denote $t=t_i$. Note that $|I|<t^{-1+10\xi}$. Moreover (QP1) and (QP2) together immediately imply \eqref{eq:nsin}. Note also that,
	by (QP3), 
	\begin{multline*}
|S_{N(x,t)-N(a_j,t)}(T^{N(a_j,t)}(x))|=|S_{N(x,t)}(x)-S_{N(a_j,t)}(x)| \\ \leq
|S_{N(x,t)}(x)-t| +|S_{N(a_j,t)}(a_j)-t|+|S_{N(a_j,t)}(a_j)-S_{N(a_j,t)}(x)|\leq 3M,
	\end{multline*}
	as $|S_{N(x,t)}(x)-t|\leq \Phi(T^{N(x,t)}(x))\leq M$ by (QP2). This implies that, for any $x\in P_a\subset I^t_{j_a}$, we have
	$$
	|N(x,t)-N(a_j,t)|\leq 3c^{-1}M,
	$$
	so that $|N(x,t)-N(y,t)|<6c^{-1}M$. Finally, (QP3) implies the bound on the difference of ergodic sums in  Lemma \ref{lem:nm1} and so all assumptions of the lemma are satisfied. Moreover the fact that the set on the RHS of (QG3) is a tower follows from \eqref{eq:qadded}. Therefore (QG3) follows.
	
	For property (QG4), we will apply Lemma \ref{lem:qboun}. For this, we will first show that there exists $x\in P_a$ and $i\in \{1,\ldots, k\}$ such that $\flow_{t_i}(x,0)\in \T$ and such that $T^{N(x,t_i)}\vert_{P_a}$ is an isometry. By the definition of $P_a=[c,d]$, there is $i\in \{1,\ldots, k\}$ such that  $\Leb([c,d])\geq M^{-\xi'} |I_{j_a}^{t_i}|$. Let $I_{j_a}^{t_i}=[a,b]=I$, and we denote $t=t_i$.  Using (QP4), we get
	\begin{multline}\label{eq:qml1}
		|S_{N(c,t)}(c)-S_{N(c,t)}(d)|\leq\\ 
		|S_{N(c,t)}(c)-S_{N(c,t)}(d)- (S_{N(a,t)}(c)-S_{N(a,t)}(d))|+|S_{N(a,t)}(c)-S_{N(a,t)}(d)|\leq\\
		M + |S_{N(c,t)-N(a,t)}(T^{N(a,t)}c)-S_{N(c,t)-N(a,t)}(T^{N(a,t)}d)|.
	\end{multline}
	By (QP2), it follows that $T^{N(a,t)+i}(I)\subset V_{M}$ for every $i\in \{0,\ldots, N(c,t)-N(a,t)\}$. Moreover, by (QP2) and the bound $|N(x,t)-N(y,t)|<6c^{-1}M$, we also have that 
	\[
	|S_{N(c,t)-N(a,t)}(T^{N(a,t)}c)-S_{N(c,t)-N(a,t)}(T^{N(a,t)}d)|\leq 6c^{-1}M \sup_{x\in V_{M}} |\Phi'(x)|(d-c)\leq t^{-\xi},
	\]
	since $\sup_{x\in V_{M}} |\Phi'(x)|<M^{2}$ and $d-c\leq b-a\leq t^{-1+10\xi}$.

	By (QP5), we know that  $S_{N(x,t)}(\cdot)$ is $(M^{-1}, M^{1-\xi}/6)$-uniformly stretching (and monotone) on $[a,b]$. This means in particular that if $J\subset [a,b]$ is an interval, then 
	$S_{N(c,t)}(J)$ is an interval with length at least $\frac{M^{1-\xi}}{6} (1-M^{-1})\Leb(J)(b-a)^{-1}$. Using this with $J=[c,d]$ along with monotonicity of $S_{N(c,t)}(\cdot)$ we get by \eqref{eq:qml1} that 
	$ M+1>|S_{N(c,t)}(c)-S_{N(c,t)}(d)|\geq \frac{M^{1-\xi-\xi'}}{7}$. The lower bound and (QP2) imply, using  Lemma \ref{lem:nm2} with $R=M$ (if $\xi$ is small enough),  that $N(\cdot,t)$ is not constant on $[c,d]$. This however means that there exists $u\in [c,d]$ such that $\flow_{t}(u,0)=(T^{N(u,t)}u,0)$, i.e. $S_{N(u,t)}(u)=t$.
	We have 
	\begin{multline*}
	|S_{N(u,t)}(u)-S_{N(u,t)}(x)- (S_{N(c,t)}(u)-S_{N(c,t)}(x))| \\ \leq  |S_{N(u,t)-N(c,t)}(T^{N(c,t)}u)-S_{N(u,t)-N(c,t)}(T^{N(c,t)}x)|.
	\end{multline*}
	Analogously to the reasoning below \eqref{eq:qml1}, we obtain
	\[
	|S_{N(u,t)-N(c,t)}(T^{N(c,t)}u)-S_{N(u,t)-N(c,t)}(T^{N(c,t)}x)|<1.
	\] 
	Moreover the monotonicity of $S_{N(c,t)}(\cdot)$ implies that  $|S_{N(c,t)}(u)-S_{N(c,t)}(x)|<|S_{N(c,t)}(c)-S_{N(c,t)}(d)|<M+1$. On the other hand $\frac{M^{1-\xi-\xi'}}{7}\leq |S_{N(c,t)}(c)-S_{N(c,t)}(d)|\leq |S_{N(c,t)}(u)-S_{N(c,t)}(d)|+|S_{N(c,t)}(u)-S_{N(c,t)}(d)|$, which gives us $\max_{x\in [c,d]}|S_{N(c,t)}(u)-S_{N(c,t)}(x)|\geq  \frac{M^{1-\xi-\xi'}}{14}$. Summarizing, 
	$$
	\frac{M^{1-\xi-\xi'}}{14}< \max_{x\in [c,d]} |S_{N(u,t)}(u)-S_{N(u,t)}(x)|< M+1.
	$$
	We will now apply Lemma \ref{lem:qboun} with $Z=P_a=[c,d]$ and $x=u$ (recall that $\flow_t(u,0)\in X$). By this lemma, for any $|L|\leq 2M<t^\xi$, the set $
	\{x\in P_a: f_tx \in \bigcup_{r\in [L,L+M^{-\varepsilon_k}]} f_r(T^{N(u,t)}P_a)\}$, contains the set $\{x\in P_a: |S_{N(u,t)}(u)-S_{N(u,t)}(x)|\in [L,L+\varepsilon_k(M)]\}$ and is contained in the set $\{x\in P_a: |S_{N(u,t)}(u)-S_{N(u,t)}(x)|\in [L-M^{-2\varepsilon_k},L+M^{-\varepsilon_k}+M^{-2\varepsilon_k}]\}$. By (QP5), the function $S_{N(u,t)}(\cdot)$ is $(M^{-1},\frac{M^{1-\xi-100k\xi}}{7})$-uniformly stretching on $P_a=[c,d]$ and monotone; thus, it follows that the measure of the first set is at least 
	$$
	(1-M^{-\varepsilon_k}) \frac{L+M^{-\varepsilon_k}-L}{S_{N(u,t)}(c)-S_{N(u,t)}(d)} (d-c).
	$$
	On the other hand the measure of the second set is at most 
	$$
	(1+M^{-\varepsilon_k}) \frac{L+M^{-\varepsilon_k}+M^{-2\varepsilon_k}- L}{S_{N(u,t)}(c)-S_{N(u,t)}(d)} (d-c).
	$$ 
	By monotonicity of $S_{N(u,t)}(\cdot)$ and Lemma \ref{lem:boun} it follows that  $S_{N(u,t)}(c)=S^-({a,t,u})$ and $S_{N(u,t)}(d)S=^+({s,t,u})$. Finally, $S_{N(u,t)}(d)-S_{N(u,t)}(c)\geq \max_{x\in [c,d]} |S_{N(u,t)}(u)-S_{N(u,t)}(x)|\geq \frac{M^{1-\xi-100k\xi}}{7} $. This indeed finishes the proof of (QG4) and completes the proof.
\end{proof}

\begin{lemma}\label{lem:qmvs} If there exists $\xi'>0$ such that for any $t$ sufficiently large there exists   finitely many  disjoint Rokhlin towers $\cR_1\ldots, \cR_d\subset \T^\Phi$ with bases (which are intervals) $B_1,\ldots, B_d\in \T$ satisfying $ |B_i|\in [t^{-\xi'},t^{-\xi'+\xi'^2}]$ with $\mu^\Phi(\T^\Phi\setminus \bigcup_i\cR_i)<t^{-\xi'}$ and a $t^{-100\xi'}$-almost partition $\cP_t=\{J^t_i\}$ of $\T$ where the intervals $J_i$ satisfy:
	\begin{enumerate}
		\item [(QJ0)] for every $M$ sufficiently large, if $T^ix\notin V_M$ for $i\leq M$, then $S_M(x)\in [(1-\varepsilon)M,(1+\varepsilon)M]$;
		\item[(QJ1)] $\sup_i t^{1-\xi'} |J_i^t| \to 0$ and $T^n(J_i)\cap \{\beta_i\}_{i=1}^d=\emptyset$ for every $n\in [0, \max_{x\in J_i}N(x,t)]$;
		\item[(QJ2)] $\min_{x,y\in J_i} |S_{N(x,t)}(\Phi')(y)| |J_i|\geq t^{\xi'}$;
		\item[(QJ3)] $\max_{x,y\in J_i} |S_{N(x,t)}(\Phi'')(y)| |J_i|\leq t^{-\xi'}\min_{x,y\in J_i} |S_{N(x,t)}(\Phi')(y)|$.
	\end{enumerate}
	Then (QP1)--(QP5) in Proposition \ref{prop:qpart1} hold. 
\end{lemma}

Similarly to Proposition \ref{prop:qpart1} the lemma has two parts, namely the one concerning Rokhlin towers and the one on conditions (QJ0)--(QJ3), which are independent. We can apply the two parts of the lemma at different times. This is what we will do in the proof.

\begin{proof}Given $1\leq M\leq \height \leq t^{\xi'}$, let us first apply the first part of the lemma for $t=\height^{\frac{-1}{\xi'}}$. This gives us Rohklin towers $\{\cR_i\}$ with $|B_i|\in [\height^{-1},\height^{-1+\delta}]$ and with $\mu^\Phi(\T^\Phi\setminus \bigcup_i\cR_i)<\height^{-1}$. So these towers satisfy the corresponding conditions in Proposition \ref{prop:qpart1}.\\
	
	We will now construct a the partition $\{I_j^t\}$ satisfying (QP1)--(QP5) by refining the partition $\cP_t=\{J^t_i\}$. Since $t$ is fixed, we will omit it from the notation. Let $\cP_t=\{J_i\}$ be the $t^{-100\xi'}$-almost partition coming from the assumptions. Let $J=[a,b]\in \cP_t$. Let $a_1>a$ be the smallest such that $|N(a_1,t)-N(a,t)|=M/2$. Define $I^J_1=[a,a_1]$. Now inductively we define $a_{i+1}>a_i$ to be the smallest for which $|N(a_{i+1},t)-N(a_i,t)|=M/2$ and $I^J_{i+1}=[a_i,a_{i+1}]$. 
	We continue until for some $i$ we have $|N(a_i,t)-N(x,t)|<M/2$ for all $x\in [a_i,b]$, and then we take $I^J_b=[a_i,b]$. The collection $\{I^J_i\}$ and $I^J_b$  forms a partition of $J$, and so the union over all $J\in \cP_t$ of these new intervals form an $t^{-100\xi'}$-almost partition of $\T$. We will discard those intervals $\{I^J_i\}$ which go to close to singularity in a short time after $t$. More precisely, let $Z_M^\Phi=\{(x,s): T^ix\in V_{M/2} \text{ for every } i\in [-M^{1+\xi'},M^{1+\xi'}] \}$. Note that by the assumptions on $\Phi$ it follows that $\mu(Z_M^\Phi)\geq 1-M^{1-\xi'}$ if $\xi'$ is small enough. We say that  $I_i^J$ is {\em good} if $I_i^J\cap Z_M^\Phi\neq \emptyset$. Note that each interval $J$ has one special interval $I^J_b$ where the function $N(\cdot,t)$  varies by less than $M/2$ and in fact might be constant. We have the following:
	\begin{equation}\label{eq:qbad}
		\text{ if } I^J_b\text{ is good then }|I^J_b|\leq M^{-10\xi'} |J|
	\end{equation}
	Before we prove the above statement let us show how it implies (QP1)--(QP5). The partition consists of good interval $\{I^J_i\}$ where we additionally discard all good intervals of the form $I^J_b$. By \eqref{eq:bad} and the above estimates on the measure of good intervals it follows that the collection $\{I_i^t\}$ consisting of all good $\{I^J_i\}$ over all $J$ forms an $M^{-9\xi'}$-almost partition of $\T$. Note that (QP1) follows from the fact that $I_i=I^J_s\subset J$ for some $J$ and we use (QJ1). 
	
	For (QP2) we note that, since $I_i$ is good, it follows that there exists $x\in I_i$ such that $\flow_tx\in Z_M^\Phi$. By the definition, this means that $T^{N(x,t)+i}x\in V_{M/2}$ for $i\in [-M^{1+\xi'},M^{1+\xi'}]$; since  $|N(x,t)-N(a,t)|\leq M$  for every $x\in I_i$, it follows that  $T^nx\in V_{M/2}$ for every $n\in [\min_{x\in I_i}N(x,t), \max_{x\in I_i}N(x,t)+c^{-1}M]$. Since $t^{-1+\xi'}|I_i|\to 0$ as $t\to \infty$ (as $I_i\subset J^t_i$), we have that if $t$ is large enough, then (QJ1) implies that $T^n(I_i)\subset V_M$. This gives (QP2). 
	
	In order to show (QP3), let $I_i=[c,d]$. Notice that 
	\begin{equation}\label{eq:qct}
		S_{N(c,t)}(x)-S_{N(c,t)}(c)= S_{N(c,t)-N(x,t)}(T^{N(x,t)}x)+(t-S_{N(c,t)}(c))+(S_{N(x,t)}(x)-t).
	\end{equation}
	By (QP2), it follows that $|t-S_{N(c,t)}(c)|<\Phi(T^{N(c,t}c)\leq M^{1-\xi'}$ with the same bound for $S_{N(x,t)}(x)-t$.
	Then, (QJ0) implies that 
	\[
	\begin{split}
		&S_{N(c,t)-N(d,t)}(T^{N(x,t)}x)\geq S_{M/2}(T^{N(x,t)}x)\geq (1-\varepsilon)M/2 \qquad \text{and} \\
		&S_{N(c,t)-N(x,t)}(T^{N(x,t)}x)\leq S_{M}(T^{N(x,t)}x)\leq (1+\varepsilon)M.
	\end{split}
	\]
	Since $|N(c,t)-N(d,t)|=M/2$ and $|N(c,t)-N(x,t)|\leq M/2$, the above inequalities immediately imply (QP3) and (QP4). 
	
	To show (QP5), we will use Lemma \ref{lem:fay2}. Let  $x_0\in [c,d]$ be such that $\inf_{x\in [c,d]} |S_{N(c,t)}(\Phi')(x)|=|S_{N(c,t)}(\Phi')(x_0)|$. \sloppy Since $|S_{N(c,t)}(\Phi)(d)-S_{N(c,t)}(\Phi)(c)|\geq M/2$ it follows that one of $|S_{N(c,t)}(\Phi)(x_0)-S_{N(c,t)}(\Phi)(c)|$ and $|S_{N(c,t)}(\Phi)(x_0)-S_{N(c,t)}(\Phi)(d)|$ is at least $M/4$. Without loss of generality, we assume it is the first one. Then, we get
\begin{multline*}
	M/4\leq |S_{N(c,t)}(\Phi)(x_0)-S_{N(c,t)}(\Phi)(c)|\leq |S_{N(c,t)}(\Phi')(x_0)|(x_0-c)+ |S_{N(c,t)}(\Phi'')(\theta)|(x_0-c^2) \\ \leq (1+K(t)^{-1})|S_{N(c,t)}(\Phi')(x_0)|(x_0-c).
\end{multline*}
	This implies that $\inf_{x\in [c,d]} |S_{N(c,t)}(\Phi')(x)| |I_i|\geq M/5$. By (QP2), we also deduce that $\inf_{x\in [c,d]} |S_{N(y,t)}(\Phi')(x)|\geq \inf_{x\in [c,d]} |S_{N(c,t)}(\Phi')(x)|+{\rm O}(M^2)$. Summarizing, 
	\[
	\inf_{x,y\in [c,d]} |S_{N(y,t)}(\Phi')(x)| |I_i|\geq M/6.
	\]
Furthermore, if $[u,w]\subset [c,d]$, then 
\[
\inf_{x,y\in [u,w]} |S_{N(y,t)}(\Phi')(x)|(w-u)\geq \inf_{x,y\in [c,d]} |S_{N(y,t)}(\Phi')(x)| |I_i|\cdot \frac{w-u}{|I_i|}\geq \frac{M}{6}\cdot \frac{w-u}{|I_i|}.
\]
Finally notice that (QJ3) holds on subintervals of $J_i$ and so, in particular, it holds on $I_i$ or the interval $[u,w]\subset I_i$. This implies (QP5). \\
	
	It remains to show \eqref{eq:bad}. Let $I^J_b=[u,v]$; then, analogoulsy to \eqref{eq:qct} and the reasoning below it, we have
	$$
	|S_{N(u,t)}(x)-S_{N(u,t)}(u)- S_{N(u,t)-N(x,t)}(T^{N(x,t)}x)|<2M,
	$$
	where $|S_{N(u,t)-N(x,t)}(T^{N(x,t)}x)|<M$ by (QP2) (see the reasoning below \eqref{eq:qct}). Together with the mean value theorem, this implies that $\inf_{x\in[u,v]}|S_{N(u,t)}(\Phi')(x)| (v-u)\leq M$. This means that $|I^J_b|\leq \Big(\min_{x\in[u,v]}|S_{N(u,t)}(\Phi')(x)|\Big)^{-1}M$. On the other hand $|J|\geq t^{\xi'}\Big(\min_{x\in[u,v]}|S_{N(u,t)}(\Phi')(x)|\Big)^{-1}$. This finishes the proof of \eqref{eq:qbad}.
\end{proof}

By the above \Cref{lem:qmvs}, \Cref{prop:qpart1}, and \Cref{prop:good_implies_lus}, in order to show that a flow is polynomially QLUS, it is enough to construct partitions $\cP_t$ that satisfy the assumptions of the above Lemma. We will do it for Kochergin flows on $\T^2$ under a diophantine condition in the next section.

\section{Kochergin flows on the torus}\label{sec:KFlows}

This section focuses on Kochergin flows on the torus; we are going to prove \Cref{thm:koch1} and \Cref{thm:koch2}. After all the work we have done so far, the crucial part that is left to do is to verify the assumptions of \Cref{lem:mvs} and \Cref{lem:qmvs} respectively.

\subsection{Stretching partitions for Kochergin flows}\label{sec:koch}

The goal of this subsection is to construct partitions satisfying (J1)--(J3) for Kochergin flows on $\T^2$ with a singularity at $0$. It follows from \cite{Koch} that all such flows are mixing. By construction, the aformentioned partitions we will then show that all these flows are mixing of all orders. Recall that $R_\alpha$ is a rotation by $\alpha$ on $\T$ which has the sequence of denominators $(q_n)_{n \in \N}$. 
We will now construct the $\varepsilon$-almost partitions of $\T$ which will have the properties (J1)--(J3). To start, let $\kappa(t)$ be a decreasing positive function with $\lim_{t\to +\infty }\kappa(t)=0$ (sufficiently slowly). The function $\kappa$ will depend on $\alpha$ and the roof function $\Phi$ only, and we will make it more precise in the course of the proof. Fix $t>0$ and  let  $n\in \N$ be unique such that $q_n\leq (1+\kappa(t))t\leq q_{n+1}$. \\

{\bf Preliminary partitions of $\T$:}\\
We will consider $2$ types of partitions, depending on the relation between $q_n$ and $q_{n+1}$:
\begin{itemize}
\item[R1.] If $\kappa(t)^{1/10}q_{n+1}\leq q_n$ we  partition $\T$  by the points $\{-i\alpha\}_{i=0}^{q_{n+1}-1}$. This partitions $\T$ into intervals of size $\sim\frac{1}{q_{n+1}}$. If $q_{n+1}<2\varepsilon^{-200}q_n$, we further refine this partition so that the new intervals have length $\in [\frac{\varepsilon^{200}}{q_n},\frac{2\varepsilon^{200}}{q_n}]$. If $q_{n+1}\geq 2\varepsilon^{-200}q_n$ we stick to the initial intervals. In both cases we call this partition $\cR_t$.

\item[R2.] If $\kappa(t)^{1/10}q_{n+1}> q_n$, then let $k$ be the smallest number such that $kq_n> (1+\kappa(t))t$. \sloppy We then consider the partition $\cR_t$ given by the points $\{-i\alpha\}_{i=0}^{K_n-1}$, where $K_n=\min(kq_n,q_{n+1})$. 
Assume, without loss of generality, that $q_n\alpha>0$.
If $K_n=q_{n+1}$, then clearly the partition is given by $\{-i\alpha\}_{i=0}^{q_{n+1}-1}$.  Otherwise, if $K_n=kq_n$, then the partition $\cR_t$ consists of $(k-1)q_n$ intervals of size $\|q_n\alpha\|\sim \frac{1}{q_{n+1}}$ and $q_n$ intervals of size $\|q_{n-1}\alpha\|-(k-1)\|q_n\alpha\|\sim \frac{1}{q_n}- \frac{(k-1)}{q_{n+1}}$. If $K_n=kq_n$ with $k\leq \varepsilon^{10} \frac{q_{n+1}}{q_n}$, then the union of  all the short intervals of size $\sim \frac{1}{q_{n+1}}$ has measure $\leq 2\varepsilon^{10}$ and we discard all of them. We denote the new  partition consisting only of the $q_n$ long intervals still by $\cR_t$. If $K_n=kq_n$ with $k>\varepsilon^{10} \frac{q_{n+1}}{q_n}$, then we discard all the atoms of the partition $\cR_t$ whose right endpoint is $-i\alpha$ with $i\in [(1-\varepsilon^5)kq_n,kq_n]$. Note that the total measure of such discarded intervals is $\leq \varepsilon^5 k q_n \frac{2}{q_{n+1}}\leq 2\varepsilon^5$. Moreover if the $q_n$ long intervals have length $\leq \varepsilon^{5}q_n^{-1}$ then we also discard them. In this case we also denote the new partial partition by $\cR_t$. \\

\end{itemize}

For $I=[a,b]\in \cR_t$ we consider $I'=[a+\varepsilon^{5}(b-a), b-\varepsilon^{5}(b-a)]$ and we call this new partial partition $\cR'_t$ note that $|I'|\geq (1-2\varepsilon^5)|I|$, and so the collection $\cR'_t$ still forms an $\varepsilon^4$-almost partition of $\T$.

Both in R1 and R2, the main point of the constructed partitions is \eqref{eq:ndis} below, i.e. for the time $(1+\kappa(t))t$ these intervals don't see the singularity. To explain the dichotomy in case R2, it is best to think about $q_{n+1}$ being a Liouvillean time, i.e. $q_{n+1}$ much larger than $q_n$. Then for $t\sim Kq_n$ we have two types of intervals: $Kq_n-q_n$ intervals of size $\sim 1/q_{n+1}$  and $q_n$ intervals of size $\sim \frac{1}{q_n}-\frac{K}{q_{n+1}}$. Now depending on the size of $K$, the collection of short intervals is negligible or not. These are the two sub-cases for R2.

One crucial property of the initial partitions (both in cases R1 and R2) is the following:
\begin{equation}\label{eq:ndis}
\text{ for every } I\in \cR_t\;\; \text{ and every }\;\; n\leq (1+\kappa(t))t,\;\;\; 0\notin R_\alpha^n I.
\end{equation}
Moreover let $\Erg(\kappa):=\{x\in \T\;:\; \frac{1}{N}S_N(x)\in (1-\kappa^2(N),1+\kappa^2(N))\}$. By the ergodic theorem, there exists a function $\kappa$ with $\kappa(t)\to 0$ as $t\to\infty$ and such that $\mu(\Erg)\geq 1-\varepsilon^{10}$. In what follows, we will only consider intervals $I\in \cR'_t$ for which $I\cap \Erg\neq \emptyset$. By definition this means that for sufficiently large $N$, there exists $z\in I$ such that 
\begin{equation}\label{eq:ergs}
S_N(\Phi)(z)\in (1-\kappa^2(N),1+\kappa^2(N))\cdot N.
\end{equation} 
Notice that, by the lower bound on the measure of $\Erg$, it follows that those intervals in $\cR'_t$ for which \eqref{eq:ergs} hold form an $\varepsilon^2$-almost partition of $\T$.\\

{\bf Finer partial partitions of $\T$:}\\
We will refine the above partitions $\cR'_t$ to discard sets where the first derivative is too small. To make the presentation clearer, we will split it into three cases:\\

{\it Case 1: R1 holds}\\
Recall that in this case $\cR'_t$ is given by $q_{n+1}$ intervals coming from the partition $\{-i\alpha\}_{i=0}^{q_{n+1}-1}$. Moreover from each such interval we trimmed the $\varepsilon^{5}$ proportion coming from the boundaries, also  if $q_{n+1}$ was small with respect to $q_n$, then we further partitioned each of these intervals so that the length is $\sim \frac{\varepsilon^{200}}{q_{n}}$.

 Let $\ell\in \N$ be unique such that $q_\ell \leq 2\kappa(t) t\leq q_{\ell+1}$. Note that, since we are in case R1, it follows that $\ell<n$. Consider the sets $B=B_t=B_1\cup B_2$, where 
\begin{equation}\label{eq:b12}
B_1=R_\alpha^{-(1+\kappa(t))t}\Big(\bigcup_{i=-q_{\ell}}^{q_{\ell}-1}[-\frac{\varepsilon^{100}}{q_\ell},\frac{\varepsilon^{100}}{q_\ell}]\Big) \;\;\text{ and  }\;\; B_2= R_\alpha^{-(1+\kappa(t))t}\Big(\bigcup_{i=-q_{\ell+1}}^{q_{\ell+1}-1}[-\frac{\varepsilon^{100}}{q_{\ell+1}},\frac{\varepsilon^{100}}{q_{\ell+1}}]\Big)
\end{equation}
Note that $B$ is a union of intervals and $\mu(B)\leq 8\varepsilon^{100}$. We will consider only those $I\in \cR'_t$ for which $I\cap B^c\neq \emptyset$. Such intervals from $\cR'_t$ still give us $\varepsilon^2$-almost partition of $\T$. 
Since we are in case R1 and we discarded the boundaries of the intervals, by ES1 in Lemma \ref{lem:ergs} we have that, 
for every $\theta\in \T$,
\begin{equation}\label{eq:sec}
\frac{1}{8} q_n^{2+\gamma}- Cq_n<S_{(1+\kappa(t))t}(\Phi'')(\theta)<2\varepsilon^{-15}q_{n+1}^{2+\gamma}
\end{equation}
 \sloppy Take $I\in \cR'_t$, so that $I\cap B^c\neq \emptyset$. Let $x_I$ be such that $\inf_{x\in I} |S_{(1+\kappa(t))t}(\Phi')(x)|=  S_{(1+\kappa(t))t}(\Phi')(x_I)$ and consider $I\setminus J=I_o\cup I_s$, where $J$ is an interval centered at $x_I$ and of length $2\varepsilon^5|I|$. By the lower bound on the second derivative, the function $S_{(1+\kappa(t))t}(\Phi')(\cdot)$ is increasing on $I$ and, if $S_{(1+\kappa(t))t}(\Phi')(x_I)\neq 0$, then $x_I$ is the left endpoint of $I$ and $I_o=\emptyset$. By the mean value theorem, for $x\in I$,
 $$
 S_{(1+\kappa(t))t}(\Phi')(x)=S_{(1+\kappa(t))t}(\Phi')(x_I)+S_{(1+\kappa(t))t}(\Phi'')(x-x_I).
$$
This, by the lower bound on the second derivative, implies that if $|x-x_I|\geq \varepsilon^5|I|$, then  
$|S_{(1+\kappa(t))t}(\Phi')(x)|\geq \frac{1}{9}q_n^{2+\gamma}\varepsilon^5 |I|$. Since $|I|\geq \frac{1}{2}\min(\frac{1}{2q_{n+1}},\frac{\varepsilon^{200}}{q_n} )$, we get
$$
\inf_{x\in I_o\cup I_s}|S_{(1+\kappa(t))t}(\Phi')(x)|\geq  \frac{1}{9}q_n^{2+\gamma}\varepsilon^5 \frac{\kappa(t)^{1/10}}{q_n}\geq \varepsilon^6 q_n^{1+\gamma}\kappa(t)^{1/10}.
$$

We will show next that 
\begin{equation}\label{eq:lboun}
\inf_{x,y\in I_o\cup I_s}|S_{N(y,t)}(\Phi')(x)|\geq  q_n^{1+\gamma}\kappa(t).
\end{equation}
Assume by contradiction that this is not the case; namely, there exists $x,y\in I_o\cup I_s$ such that $|S_{N(y,t)}(\Phi')(x)|< q_n^{1+\gamma} \kappa(t)$. Denote $N=N(y,t)$. Notice that, for every $z\in I$, 
 $$
|S_{N}(\Phi)(z)- S_N(\Phi)(x)|\leq |S_N(\phi')(x)||I|+|S_N(\phi'')(\theta)||I|^2<q_n^{\gamma+\varepsilon}.
$$
Therefore, using this for $z\in I$ which satisfies \eqref{eq:ergs}, we get
$$
(1-\kappa^2(N))N+q_n^{\gamma+\varepsilon}\geq S_{N}(\Phi)(z) +q_n^{\gamma+\varepsilon} \geq S_N(\Phi)(x)\geq S_{N(y,t)}(\Phi)(y)- q_n^{\gamma+\varepsilon}\geq t-2q_n^{\gamma+\varepsilon},
$$
which implies in particular that $N\geq (1-\kappa(t)/2)t$.
Let us remark that 
\begin{multline*}
	S_N(\Phi')(x)=S_{(1+\kappa(t))t}(\Phi')(x)+ S_{N-(1+\kappa(t))t}(\Phi')(R_\alpha^{(1+\kappa(t))t}x) \\ \geq \varepsilon^6 q_n^{1+\gamma}\kappa(t)^{1/10}-S_{N-(1+\kappa(t))t}(\Phi')(R_\alpha^{(1+\kappa(t))t}x),
\end{multline*}
and $|N-(1+\kappa(t))t|<2\kappa(t)t$. Recall that $2\kappa(t)t\in [q_{\ell},q_{\ell+1}]$. If $q_{\ell+1}<\kappa(t)^{-1/5}q_n$, then we use the first part of ES3 and $I\cap B_2\neq \emptyset$, to get that  $|S_{N-(1+\kappa(t))t}(\Phi')(R_\alpha^{(1+\kappa(t))t}x)|<\varepsilon^{-200}q_{\ell+1}^{1+\gamma}\leq \varepsilon^7 q_n^{1+\gamma}\kappa(t)^{1/10}$. If $q_{\ell+1}\geq \kappa(t)^{-1/5}q_n$, then $2\kappa(t)t\leq \varepsilon^{200}q_{\ell+1}$ and then using the second part of ES3 and  $I\cap B_1\neq \emptyset$,  we obtain 
$$
|S_{N-(1+\kappa(t))t}(\Phi')(x)|< \varepsilon^{-200} (N-(1+\kappa(t))t)q_{\ell}^\gamma\leq \kappa(t)^{1/2} q_n^{1+\gamma}.
$$
In both cases we get a contradiction, which implies that \eqref{eq:lboun} holds.

Note that by \eqref{eq:lboun} it follows that by subdividing $I_o$ and $I_s$ into intervals $\{K^I_\ell\}$ of size $\sim \frac{1}{q_{n}^{1+\gamma/2}}$, we get 
$$
\inf_{x,y\in K_\ell}|S_{N(y,t)}(\Phi')(x)|\cdot |K_\ell|\geq \frac{1}{2}q_n^{\gamma/2} \kappa(t)\geq K(t), 
$$
and by \eqref{eq:sec}
$$
\inf_{x,y\in K_\ell}|S_{N(y,t)}(\Phi'')(x)|\cdot |K_\ell|\leq K(t)^{-1}
\inf_{x,y\in K_\ell}|S_{N(y,t)}(\Phi')(x)|,
$$
if $K(t)$ goes sufficiently slowly to $\infty$. This means that the collection of intervals $\{K_\ell^{I}\}$ forms an $\varepsilon^2$-almost partition of $\T$ which satisfies (J1)--(J3) and finishes the proof in case R1.\\

{\it Case 2: R2 holds and  $k\leq \varepsilon^{10}q_{n+1}/q_n$.}\\
In this case the partition $\cR'_t$ consists of $q_n$ intervals whose right endpoint in of the form $-i\alpha$, for $i=0,\ldots, q_n-1$. Moreover, we have trimmed the $\varepsilon^{5}$-neighborhood of the boundaries of these intervals.  Let $\ell\in \N$ be unique such that $q_\ell \leq 2\kappa(t) t\leq q_{\ell+1}$. If $\ell<n$,  consider the sets $B=B_t=B_1\cup B_2$, where 
\begin{equation}\label{eq:b123}
B_1=R_\alpha^{-(1+\kappa(t))t}\Big(\bigcup_{i=-q_{\ell}}^{q_{\ell}-1}[-\frac{\varepsilon^{100}}{q_\ell},\frac{\varepsilon^{100}}{q_\ell}]\Big) \;\;\text{ and  }\;\; B_2= R_\alpha^{-(1+\kappa(t))t}\Big(\bigcup_{i=-q_{\ell+1}}^{q_{\ell+1}-1}[-\frac{\varepsilon^{100}}{q_{\ell+1}},\frac{\varepsilon^{100}}{q_{\ell+1}}]\Big)
\end{equation}
Notice that $B$ is a union of intervals and $\mu(B)\leq 8\varepsilon^{100}$. We will consider only those $I\in \cR'_t$ for which $I\cap B^c\neq \emptyset$. If $\ell=n$, we take $B=\emptyset$. Since we have discarded $\varepsilon^{5}$-neighborhood of the boundaries of the intervals and since $k\leq \varepsilon^{10}q_{n+1}/q_n$, it follows that  $\{x+i\alpha\}_{i=0}^{kq_n} \cap [-\frac{\varepsilon^5}{2q_n},\frac{\varepsilon^5}{2q_n}]=\emptyset$. Using the fact that $S_m(\Phi'')$ is positive for large enough $m$, ES1 and ES2 (where we split $kq_n$ into orbits of length $q_n$) imply that
\begin{equation}\label{eq:up}
\frac{1}{10}k q_n^{2+\gamma}<S_{(1+\kappa(t))t}(\Phi'')(\theta)<2k\varepsilon^{-15}q_n^{2+\gamma}.
\end{equation}
The proof is now similar to the analogous reasoning in case R1.  Take $I\in \cR'_t$ such that $I\cap B^c\neq \emptyset$. Let $x_I$ be such that $\inf_{x\in I} |S_{(1+\kappa(t))t}(\Phi')(x)|=  S_{(1+\kappa(t))t}(\Phi')(x_I)$ and consider $I\setminus J=I_o\cup I_s$, where $J$ is an interval centered at $x_I$ and length $2\varepsilon^5|I|$. Notice that, by the lower bound on the second derivative, the function $S_{(1+\kappa(t))t}(\Phi')(\cdot)$ is increasing on $I$ and if $S_{(1+\kappa(t))t}(\Phi')(x_I)\neq 0$, then $x_I$ is the left endpoint of $I$ and $I_o=\emptyset$. By the mean value theorem, for $x\in I$,
 $$
 S_{(1+\kappa(t))t}(\Phi')(x)=S_{(1+\kappa(t))t}(\Phi')(x_I)+S_{(1+\kappa(t))t}(\Phi'')(x-x_I).
$$
This, by the lower bound on the second derivative, implies that, if $|x-x_I|\geq \varepsilon^5|I|$, then  
$|S_{(1+\kappa(t))t}(\Phi')(x)|\geq \frac{1}{10}k q_n^{2+\gamma}\varepsilon^5 |I|$. Since $|I|\geq \frac{1}{4q_n}$,  we get
$$
\inf_{x\in I_o\cup I_s}|S_{(1+\kappa(t))t}(\Phi')(x)|\geq  \frac{1}{40}kq_n^{1+\gamma}\varepsilon^5\geq \varepsilon^6 k q_n^{1+\gamma}.
$$
We will show next that 
\begin{equation}\label{eq:lboun2}
\inf_{x,y\in I_o\cup I_s}|S_{N(y,t)}(\Phi')(x)|\geq  \varepsilon^{8} kq_n^{1+\gamma}.
\end{equation}
The proof is very similar to the proof of \eqref{eq:lboun}. Assume by contradiction that this is not the case; i.e., there exists $x,y\in I_o\cup I_s$ such that $|S_{N(y,t)}(\Phi')(x)|< \varepsilon^8 kq_n^{1+\gamma}$. Denote $N=N(y,t)$. For every $z\in I$, we have
 $$
|S_{N}(\Phi)(z)- S_N(\Phi)(x)|\leq |S_N(\phi')(x)||I|+|S_N(\phi'')(\theta)||I|^2<4\varepsilon^{-15}kq_n^\gamma 
$$
Therefore, using this for $z\in I$ which satisfies \eqref{eq:ergs}, we get
\begin{multline*}
	(1-\kappa^2(N))N+4\varepsilon^{-15}kq_n^\gamma\geq S_{N}(\Phi)(z) +4\varepsilon^{-15}kq_n^\gamma \geq S_N(\Phi)(x) \\ \geq S_{N(y,t)}(\Phi)(y)- 4\varepsilon^{-15}kq_n^\gamma\geq t-q_n^{\gamma+\varepsilon},
\end{multline*}
which implies in particular that $N\geq (1-\kappa(t)/2)t$.
Notice that 
\begin{multline*}
S_N(\Phi')(x)=S_{(1+\kappa(t))t}(\Phi')(x)+ S_{N-(1+\kappa(t))t}(\Phi')(R_\alpha^{(1+\kappa(t))t}x) \\ \geq \varepsilon^6 kq_n^{1+\gamma}-S_{N-(1+\kappa(t))t}(\Phi')(R_\alpha^{(1+\kappa(t))t}x),
\end{multline*}
and $|N-(1+\kappa(t))t|<2\kappa(t)t$. Recall that $2\kappa(t)t\in [q_{\ell},q_{\ell+1}]$. 

If $\ell<n$ and $q_{\ell+1}<\kappa(t)^{-1/5}q_n$, then we use the first part of ES3 and $I\cap B_2\neq \emptyset$, to get that  $|S_{N-(1+\kappa(t))t}(\Phi')(R_\alpha^{(1+\kappa(t))t}x)|<\varepsilon^{-200}q_{\ell+1}^{1+\gamma}\leq \varepsilon^{-200}q_n^{1+\gamma}$. If $\ell=n$ or $q_{\ell+1}\geq \kappa(t)^{-1/5}q_n$, then $2\kappa(t)t\leq \varepsilon^{200}q_{n+1}$ and then using the second part of ES3 and  $I\cap B_1\neq \emptyset$,  it follows that 
$$
|S_{N-(1+\kappa(t))t}(\Phi')(R_\alpha^{(1+\kappa(t))t}x)|< \varepsilon^{-200} (N-(1+\kappa(t))t)q_n^\gamma\leq \kappa(t)^{1/2} kq_n^{1+\gamma}.
$$
In both cases we get a contradiction which implies that \eqref{eq:lboun2} holds.

By  the upper bound in \eqref{eq:up} and monotonicity, subdividing $I_o$ and $I_s$ into intervals $\{K^I_\ell\}$ of size $\sim \frac{1}{q_{n}^{1+\gamma/2}}$, we get 
$$
\inf_{x,y\in K_\ell}|S_{N(y,t)}(\Phi')(x)|\cdot |K_\ell|\geq \varepsilon^{8}k q_n^{\gamma/2}\geq K(t), 
$$
and 
$$
\inf_{x,y\in K_\ell}|S_{N(y,t)}(\Phi'')(x)|\cdot |K_\ell|\leq K(t)^{-1}
\inf_{x,y\in K_\ell}|S_{N(y,t)}(\Phi')(x)|,
$$
if $K(t)$ goes sufficiently slowly to $\infty$. This means that the collection of intervals $\{K_\ell^{I}\}$ forms an $\varepsilon^2$-almost partition of $\T$ which satisfies (J1)--(J3) and finishes the proof in case R2.\\

{\it Case 3: R2 holds and  $k\geq \varepsilon^{10}q_{n+1}/q_n$.}\\
In this case, by the construction of $\cR'_t$, we have two types of intervals in the partition for which the considerations will be different. The first type of intervals are those with length $\sim \frac{1}{q_{n+1}}$ and we have $(k-\varepsilon^5k) q_n$ of them. Their right endpoint is at $-i\alpha$ with $i\leq (1-\varepsilon^5)kq_n$. The second type are  $q_n$ intervals of length $\sim \frac{1}{q_{n}}-\frac{k-1}{q_{n+1}}$. Let $I\in\cR'_t$ be an interval of the first type. For every $\theta\in I$, we get 
$S_{(1+\kappa(t))t}(\Phi'')(\theta)<S_{q_{n+1}}(\Phi'')(\theta)<\varepsilon^{-15}q_{n+1}^{2+\gamma}$, where the last inequality follows from ES2 and the fact that we trimmed the $\varepsilon^5$-boundaries of the initial partition $\cR_t$. Let the right endpoint of $I$ be $-i\alpha$ with $i<(1-\varepsilon^5)kq_n$. Since $(1+\kappa(t))t\geq (k-1)q_n\geq (1-\varepsilon^5)kq_n$, by the definition of $k$, we get $S_{(1+\kappa(t))t}(\Phi'')(\theta)\geq \Phi''(\theta+ i\alpha)\geq \frac{1}{4}q_{n+1}^{2+\gamma}$. Summarizing,
\begin{equation}\label{eq:up2}
\frac{1}{4}q_{n+1}^{2+\gamma}<S_{(1+\kappa(t))t}(\Phi'')(\theta)<\varepsilon^{-15}q_{n+1}^{2+\gamma}.
\end{equation}
Let $x_I$ be such that $\inf_{x\in I} |S_{(1+\kappa(t))t}(\Phi')(x)|=  S_{(1+\kappa(t))t}(\Phi')(x_I)$ and consider $I\setminus J=I_o\cup I_s$, where $J$ is an interval centered at $x_I$ and length $2\varepsilon^5|I|$. Notice that, by the lower bound on the second derivative, the function $S_{(1+\kappa(t))t}(\Phi')(\cdot)$ is increasing on $I$ and, if $S_{(1+\kappa(t))t}(\Phi')(x_I)\neq 0$, then $x_I$ is the left endpoint of $I$ and $I_o=\emptyset$. By the mean value theorem, for $x\in I$,
 $$
 S_{(1+\kappa(t))t}(\Phi')(x)=S_{(1+\kappa(t))t}(\Phi')(x_I)+S_{(1+\kappa(t))t}(\Phi'')(x-x_I).
$$
This, by the lower bound on the second derivative, implies that, if $|x-x_I|\geq \varepsilon^5|I|$, then  
$|S_{(1+\kappa(t))t}(\Phi')(x)|\geq \frac{1}{4}q_{n+1}^{2+\gamma}\varepsilon^5 |I|$. Since $|I|\geq \frac{1}{4q_{n+1}}$,  it follows that 
$$
\inf_{x\in I_o\cup I_s}|S_{(1+\kappa(t))t}(\Phi')(x)|\geq  \frac{1}{16}\varepsilon^5 q_{n+1}^{1+\gamma}.
$$
We will show next that 
\begin{equation}\label{eq:lboun3}
\inf_{x,y\in I_o\cup I_s}|S_{N(y,t)}(\Phi')(x)|\geq  \varepsilon^8 q_{n+1}^{1+\gamma}.
\end{equation}
We will do it at the end of the proof.
Note that by  the upper bound in \eqref{eq:up2} and monotonicity, subdividing $I_o$ and $I_s$ into intervals $\{K^I_\ell\}$ of size $\sim \frac{1}{q_{n+1}^{1+\gamma/2}}$, we get 
$$
\inf_{x,y\in K_\ell}|S_{N(y,t)}(\Phi')(x)| \cdot |K_\ell|\geq \varepsilon^{8}k q_{n+1}^{\gamma/2}\geq K(t), 
$$
and 
$$
\inf_{x,y\in K_\ell}|S_{N(y,t)}(\Phi'')(x)| \cdot |K_\ell|\leq K(t)^{-1}
\inf_{x,y\in K_\ell}|S_{N(y,t)}(\Phi')(x)|,
$$
if $K(t)$ goes sufficiently slowly to $\infty$.\\

It remains to deal with the second type of intervals; i.e., the $q_n$ intervals of length $\sim \frac{1}{q_{n}}-\frac{k-1}{q_{n+1}}$. Note that by the definition of $\cR_t$ and $\cR'_t$ we only consider them if their length is $\geq \varepsilon^5q_n^{-1}$. Hence, let $I'\in \cR'_t$ be an interval with $|I'|\geq \varepsilon^5 q_n$. Here, the reasoning is analogous to the corresponding one in R2. We present the full proof for completeness. 
Since we have discarded $\varepsilon^{5}$-neighborhood of the boundaries of the intervals, we have $\{x+i\alpha\}_{i=0}^{kq_n} \cap [-\varepsilon^{10}q_n^{-1},\varepsilon^{10}q_n^{-1}]=\emptyset$.  Therefore, using the fact that $S_m(\Phi'')$ is positive for large enough $m$, it follows by ES1 and ES2 (where we split $kq_n$ into orbits of length $q_n$), that for $\theta\in I'$
\begin{equation}\label{eq:up4}
\frac{1}{10}k q_n^{2+\gamma}<|S_{(1+\kappa(t))t}(\Phi'')(\theta) | <2k\varepsilon^{-30}q_n^{2+\gamma}.
\end{equation}
Let $x_I$ be such that $\inf_{x\in I'} |S_{(1+\kappa(t))t}(\Phi')(x)|=  S_{(1+\kappa(t))t}(\Phi')(x_{I'})$ and consider $I'\setminus J=I_o\cup I_s$, where $J$ is an interval centered at $x_I$ and length $2\varepsilon^5|I'|$. Notice that by the lower bound on the second derivative, the function $S_{(1+\kappa(t))t}(\Phi')(\cdot)$ is increasing on $I'$ and, if $S_{(1+\kappa(t))t}(\Phi')(x_{I'})\neq 0$, then $x_{I'}$ is the left endpoint of $I$ and $I_o=\emptyset$. By the mean value theorem, for $x\in I$,
 $$
 S_{(1+\kappa(t))t}(\Phi')(x)=S_{(1+\kappa(t))t}(\Phi')(x_{I'})+S_{(1+\kappa(t))t}(\Phi'')(x-x_{I'}).
$$
This, by the lower bound on the second derivative, implies that if $|x-x_{I'}|\geq \varepsilon^5|I'|$, then  
$|S_{(1+\kappa(t))t}(\Phi')(x)|\geq \frac{1}{10}k q_n^{2+\gamma}\varepsilon^5 |I'|$. Since $|I|\geq \frac{\varepsilon^5}{q_n}$,  we get
$$
\inf_{x\in I_o\cup I_s}|S_{(1+\kappa(t))t}(\Phi')(x)|\geq  \frac{1}{40}kq_n^{1+\gamma}\varepsilon^5\geq \varepsilon^6 k q_n^{1+\gamma}.
$$
We will show next that 
\begin{equation}\label{eq:lboun4}
\inf_{x,y\in I_o\cup I_s}|S_{N(y,t)}(\Phi')(x)|\geq  \varepsilon^{8} kq_n^{1+\gamma}.
\end{equation}
Note that by  the upper bound in \eqref{eq:up4} and monotonicity  it follows that by subdividing $I_o$ and $I_s$ into intervals $\{L^{I'}_\ell\}$ of size $\sim \frac{1}{q_{n}^{1+\gamma/2}}$, we get 
$$
\inf_{x,y\in K_\ell}|S_{N(y,t)}(\Phi')(x)| \cdot |K_\ell|\geq \varepsilon^{8}k q_n^{\gamma/2}\geq K(t), 
$$
and 
$$
\inf_{x,y\in K_\ell}|S_{N(y,t)}(\Phi'')(x)| \cdot |K_\ell|\leq K(t)^{-1}
\inf_{x,y\in K_\ell}|S_{N(y,t)}(\Phi')(x)|,
$$
if $K(t)$ goes sufficiently slowly to $\infty$. To finish, the collections $\{K^I_\ell\}$ and $\{L^{I'}_\ell\}$ considered for all $I$ and $I'$ in the partition  form an $\varepsilon^2$-almost partition of $\T$ which satisfies (J1)--(J3) and this finishes the proof in case R2. 

It remains to show \eqref{eq:lboun3} and \eqref{eq:lboun4}. First, analogously to \eqref{eq:lboun} and \eqref{eq:lboun2}  we show that $N=N(y,t)\geq (1-\kappa(t))t$. Then we proceed analogously to the proof of \eqref{eq:lboun} and \eqref{eq:lboun2}. We only need to bound $S_{N-(1+\kappa(t))t}(\Phi')(R^{(1+\kappa(t))t}x)$. Here we do not have the set $B$, but for both type of intervals, we notice that the orbit 
$\{R^{(1+\kappa(t))t}x\}_{N-(1+\kappa(t))t}^0$ is disjoint with the set $[-\frac{\varepsilon^{500}}{q_n},\frac{\varepsilon^{500}}{q_n}]$. This implies using the second part of ES3 that 
$$
 S_{N-(1+\kappa(t))t}(\Phi')(R^{(1+\kappa(t))t}x)\leq \varepsilon^{-1000}2\kappa(t)t q_n^\gamma\leq 4\varepsilon^{-1000}\kappa(t)kq_n^{1+\gamma},
$$
which is enough to get a contradiction with \eqref{eq:lboun3} and \eqref{eq:lboun4}. 
This finishes the proof.

\subsection{Polynomially stretching partitions for a class of Kochergin flows on the torus.}
In this section we will construct Rokhlin towers and partitions that appear in Lemma \ref{lem:qmvs}. Recall that, by assumption of \Cref{thm:koch2}, $\alpha$ satisfies $q_{n+1}\leq Cq_n (\log ^2q_n)$. We will first contruct the Rokhlin towers. In fact we will do it with $d=2$. Let $t>0$ and let $n$ be unique with $q_n\leq t^{4\xi'}<q_{n+1}$

 Recall that for every rotation there is a sequence of disjoint towers $\cT^1_n, \cT^2_n$ that covers $\T$ given by the following (we assume in what follows that $n$ is even):  $\cT^1_n$ is a tower of height $q_n$ and base $B^1_n=[0,R_\alpha^{q_{n-1}}(0)]$ and $\cT^2_n$ is a tower of height $q_{n-1}$ and base $B^2_n=[R_{\alpha}^{q_n}(0),0]$. Let $\tilde{B}_1=[\frac{1}{q_n^{1+\eta}}, R_\alpha^{q_{n-1}}(0)-\frac{1}{q_n^{1+\eta}}]$ and $\tilde{B}_2=[R_{\alpha}^{q_n}(0)+\frac{1}{q_n^{1+\eta}},-\frac{1}{q_n^{1+\eta}}]$. We will now consider towers $\cR_1$, $\cR_2$ for the flow $\flowR$ with bases $\tilde{B}_1$ and  $\tilde{B}_2$, and heights $h_1=q_{n}-q_n^{\gamma+\eta}$ and $h_2=q_{n-1}-Cq_n^{\gamma+\eta}$. Note that $\min_{i=1,2} |\tilde{B}_i|\geq t^{-5\xi'}$. 
 
 To show that these are indeed towers (namely, different levels are disjoint), it is enough to show that $h_1\leq \min_{x\in \tilde{B}_1} S_{q_n}(x)$  and  $h_2\leq \min_{x\in \tilde{B}_2} S_{q_{n-1}}(x)$. Note that, 
by ES4 in Lemma \ref{lem:ergs}, we have $ |S_{q_n}(x)-q_n|\leq \Phi(x_{min}^{q_n})+ Cq_n^{\gamma}$. Since we trimmed the boundaries of $B^1_n$, it follows that $x_{min}^{q_n}\geq \frac{1}{q_n^{1+\eta}}$. Therefore, $|S_{q_n}(x)-q_n|\leq q_n^{\gamma+\eta}$. The reasoning for $h_2$ is analogous. It remains to notice that $\mu(\cR_1\cup \cR_2)\geq 1- q_n\frac{4}{q_n^{1+\eta}}=1-q_n^{-\eta}$, which satisfies the assumptions from the lemma for $\eta>\xi'$. This finishes the proof of the first part. Note that the first and second part of the lemma are independent so we now move to the second part. First note that (QJ0) holds by ES0 and the second part of ES4 in Lemma \ref{lem:ergs}. So we will now construct partitions $\cP_t$ which satisfy (QJ1)--(QJ3).

Given $t>0$ and a small $\xi>0$ let $n$ be unique with $q_n\leq t+ Ct^{\gamma+\xi}<q_{n+1}$. Let $\cP''_t$ be the partition given by the points $\{-i\alpha\}_{i=0}^{q_{n+1}-1}$. Let $I'=[a',b']$ be an interval of this partition (recall that $|I'|\sim \frac{1}{q_{n+1}}$). Let $I=[a,b]:=[a'+\frac{1}{q_{n+1}^{1+\xi}},-\frac{1}{q_{n+1}^{1+\xi}}+b']$ and let $\cP'_t$ be the partition given by the intervals $I$.  First notice that $\mu(\bigcup_{I\in \cP_t}I)\geq 1-q_{n+1}^{-\xi}$ so if $\xi>200\xi'$ then $\cP'_t$ is indeed an $t^{-200\xi'}$ almost partition of $\cT$. Moreover, $t^{1-\xi'}|I|\leq 4q_n^{1-\xi'}\frac{1}{3q_{n+1}}< \frac{1}{q_{n+1}^{\xi'}}\to 0$ as $t\to \infty$
 by ES0 in Lemma \ref{lem:ergs} and $\max_{x\in I} N(x,t)< t+ Ct^{\gamma}<q_{n+1}$. This implies that (QJ1) holds by the definition of the partition $\cP'_t$.  Let $k$ be unique such that $q_k<Ct^\gamma<q_{k+1}$. We will only consider those intervals $I\in \cP_t$  which satisfy
\begin{equation}\label{eq:nc}
R_{\alpha}^{t+s}I\cap \T\setminus \left[-\frac{1}{q_{k+1}^{1+\xi}},\frac{1}{q_{k+1}^{1+\xi}}\right], \qquad \text{ for every } s\leq q_{k+1}.
 \end{equation}

Note that the measure of such good intervals is at least $1-q_{k+1}^{-\xi}\geq 1-t^{-\xi\gamma/2}$ and so indeed we can only consider such intervals

Our final partition $\cP_t$ is be a refinement (up to a negligible set) of the partition $\cP'_t$ build as follows:  take $I\in \cP'_t$ and let  $x_I\in I$ be such that $\min_{x\in I}S_{t}(\Phi')(x)=S_{t}(\Phi')(x_I)$ (here we write $S_t$ with the understanding $S_{[t]}$). Consider $I\setminus J=I_o\cup I_s$, where $J$ is an interval centered at $x_I$ and length $\frac{1}{q_{n+1}^{1+\xi}}$. By ES2. in Lemma \ref{lem:ergs},  for every $k\leq q_{n+1}$
\begin{equation}\label{eq:qsec}
S_k(\Phi'')(x)\leq 8q_n^{2+\gamma+4\xi},
\end{equation}
since we trimmed the boundaries of the intervals $I'$ and so $x_{min}^{q_{n+1}}\geq \frac{1}{q_{n+1}^{1+\xi}}$. 
 Notice that, since $S_t(\Phi'')(\cdot)>0$,  the function $S_{t}(\Phi')(\cdot)$ is increasing on $I$ and if $S_{t}(\Phi')(x_I)\neq 0$, then $x_I$ is the left endpoint of $I$ and $I_o=\emptyset$. By the mean value theorem, for $x\in I$,
 $$
 S_{t}(\Phi')(x)=S_{t}(\Phi')(x_I)+S_{t}(\Phi'')(x-x_I).
$$
This, by ES1 in Lemma \ref{lem:ergs}, implies that if $|x-x_I|\geq \frac{1}{q_{n+1}^{1+\xi}}$, then  
$|S_{t}(\Phi')(x)|\geq \frac{1}{9}q_n^{2+\gamma-\xi}\frac{1}{q_{n+1}^{1+\xi}}$. It follows that $\inf_{x\in I_o\cup I_s}|S_{t}(\Phi')(x)|\geq  \frac{1}{9}q_n^{1+\gamma-3\xi}$.
We will show next that 
\begin{equation}\label{eq:qlboun}
\inf_{x,y\in I_o\cup I_s}|S_{N(y,t)}(\Phi')(x)|\geq \frac{1}{9}q_n^{1+\gamma-4\xi}.
\end{equation}
Note that \eqref{eq:qlboun} holds then  we can further subdivide the intervals $I_o$ and $I_s$ into intervals of length $J_i$ of length $\sim \frac{1}{q_n^{1+\gamma-5\xi}}$ so that 
$\inf_{x,y\in J_i}|S_{N(y,t)}(\Phi')(x)||J_i|\geq \frac{1}{9}q_n^{1+\gamma-4\xi} |J_i|\geq q_n^\xi\geq t^{\xi/2}$ and 
$\sup_{x,y\in J_i}|S_{N(y,t)}(\Phi'')(x)|J_i|\leq t^{-\xi/2} \inf_{x,y\in I_o\cup I_s}|S_{N(y,t)}(\Phi')(x)|$. This means that the intervals $J_i$ satisfy (QJ2) and (QJ3), and hence finish the proof.\\

We only need to show \eqref{eq:qsec}. We have 
$$
S_{N(y,t)}(\Phi')(x)=S_{t}(\Phi')(x)+S_{N(y,t)-t}(\Phi')(R_\alpha^tx).
$$
Note that, by ES0, we have $|N(y,t)-t|\leq Ct^\gamma<q_{k+1}$. Therefore, using ES3 in Lemma \ref{lem:ergs} along with \eqref{eq:nc}, it follows that $|S_{N(y,t)-t}(\Phi')(R_\alpha^tx)|\leq q_{k+1}^{(1+\gamma)(1+\xi)}|\leq t^{\gamma(1+\gamma)(1+\xi)+3\xi}\leq q_n^{\gamma(1+\gamma)(1+\xi)+4\xi}$. This together with the lower bound on $S_{t}(\Phi')(x)$ implies that \eqref{eq:qlboun} holds. This finishes the proof.

\subsection{Polynomial mixing for a class of Kochergin flows on $\T^2$}

The results in the previous subsection, together with \Cref{prop:good_implies_lus}, imply that Kochergin flows satisfying the assumption of \Cref{thm:koch2} are QLUS. In order to apply our main quantitative abstract result \Cref{thm:main_abstract_result_quantitative}, and hence prove \Cref{thm:koch2}, it remains to show that these flows are polynomially mixing. We do it in this section.

We let $M= \T^2 \setminus \{0\}$ and we denote by $\mu$ the Lebesgue measure on $M$. We recall two results on polynomial mixing for Kochergin flows on the torus with one singularity at $0$. 

\begin{theorem}[{\cite[Theorem 1.1]{Fay2}}] Let $\gamma\leq 2/5$ and let $\alpha$ satisfy $q_{n+1}\leq C_\alpha q_n^{1+\gamma/100}$.  Let $(K_t)_{t\in \R}$ be the corresponding Kochergin flow. There exists $\eta>0$ such that for any two rectangles $A,B\subset M$, we have $|\mu(A\cap K_{t}B)-\mu(A)\mu(B)|\leq t^{-\eta}$, for sufficiently large $t$.
\end{theorem}
Using this theorem and an approximation argument, one could deduce a statement on decay of correlations for smooth functions. However, by using the above theorem as it is, one wouldn't be able to get an explicit dependence on the $\mathscr{C}^k$-norms of the smooth functions. Nonetheless, it seems possible to strengthen the conclusion of the theorem above to say that it holds uniformly for all rectangles with side lengths $\geq t^{-\eta'}$, and then deduce the stronger statement (with control on the $\mathscr{C}^k$-norms). We do not pursue this strategy here. The main result that is useful in our setting is the following.
\begin{theorem}[{\cite[Proposition 3.2]{FFK}}]\label{thm:FFK} 
	Let $\alpha$ satisfy $q_{n+1}\leq C_\alpha q_n(\log q_n)^{1+\xi}$ and let $\gamma=1-\eta$ with $\eta$ sufficiently small. There exists a constant $C_0>0$ so that, for any $\phi_0, \phi_1\in \mathscr{C}_c^2(M)$ with $\phi_0$ being a smooth coboundary for the flow $(K_t)_{t\in \R}$, we have 
$$
|\langle \phi_0, \phi_1 \circ K_t\rangle - \mu(\phi_0)\mu(\phi_0)|<C_0 \, \|\phi_0\|_{\mathscr{C}^2} \, \|\phi_1\|_{\mathscr{C}^2} \, t^{-1/3}.
$$
\end{theorem}
In \cite[Proposition 3.2]{FFK} there is a \lq\lq bad set\rq\rq\ of measure $q_n^{1/2-6\eta}$ and for our purposes we just estimate correlations trivially on this bad set. In \cite{FFK}, the authors wanted to push the correlation exponent to $1/2$ to get square integrability. In our case, any positive exponent suffices and so we just use $1/3$. In fact, the methods of \cite{FFK} should apply to any $\gamma>0$ to get (under the diophantine condition) the above result with some $\eta=\eta(\gamma)>0$ instead of $1/3$. In our case it is important that we have polynomial decay of correlations for all smooth functions. We will use the above theorem to deduce polynomial decay (with a smaller exponent for all smooth functions); the second ingredient is the following result on the decay of ergodic averages in $L^2$. 

\begin{lemma}\label{lem:L2_decay_koch} 
	Under the assumptions of Theorem \ref{thm:FFK}, there exist  $C_1,\eta_1>0$ such that, for any $\phi \in \mathscr{C}_c^1(M)$ and for any $t\geq 1$, we have
\[
\left\| \frac{1}{t} \int_{0}^{t} \phi \circ K_r \diff r- \mu(\phi) \right\|_2 \leq C_1 \, \|\phi\|_{\mathscr{C}^1} \, t^{-\eta_1}.
\]
\end{lemma}
\begin{proof} Let $f$ denote the roof function, which, for simplicity, we assume satisfies $\Leb(f)=1$. The proof then follows from the considerations in \cite{DFK}. First, using the decomposition (4.5) in \cite{DFK} for  $\bar{x}=(x,w)\in \T^2$ and the fact the $\phi$ has compact support, 
$$
 \int_{0}^{t} \phi \circ K_r(\bar{x}) \diff r=\sum_{i=0}^{N(x,w,t)-1}\psi(x+i\alpha)+ {\rm O}(1), 
$$ 
where $\psi(x)=\int_0^{f(x)} \phi \circ K_r(x,0) \diff r$. Moreover, since $\phi$ has compact support, $\psi\in \mathscr{C}^1(\T)$. In particular, by the diophantine assumption on $\alpha$, $|\sum_{i=0}^{N(x,w,t)-1}\psi(x+i\alpha)- N(x,w,t)\Leb(\psi)|={\rm O}(t^{\epsilon})$. Notice that, by Fubini, $\Leb(\psi)=\mu(\phi)$. By \cite[Lemma 8]{DFK}, it follows that $|S_{N(x,w,t)}(x)-N(x,w,t)|\ll f(x_{min,N(x,w,t)})+{\rm O}(N^\gamma\log^5 N)$. This implies that 
\[
|N(x,w,t)-t|<f(x+N(x,w,t)\alpha)+w+ f(x_{min,N(x,w,t)})+{\rm O}(N^\gamma\log^5 N).
\]
Consider only the points $x$ such that  $x+i\alpha\cap [-\frac{1}{t^{1+\delta}},\frac{1}{t^{1+\delta}}]=\emptyset$ for $i\leq (\min f)^{-1}t$. Then the right hand side of the last equation is $\ll t^{1-\eta}$ (for $\delta$ sufficiently small). It remains to notice that the set of $x$ which  do not satisfy this has measure $\ll t^{-\delta}$. This finishes the proof.
\end{proof}

We can now show that the Kochergin flows we consider are polynomially mixing.

\begin{proposition}
	Under the assumptions of Theorem \ref{thm:FFK}, there exist $C_2,\eta_2>0$ such that, for any $\phi_0, \phi_1 \in \mathscr{C}_c^2(M)$ and for any $t\geq 1$, we have
	\[
	|\langle \phi_0, \phi_1 \circ K_t\rangle - \mu(\phi_0)\mu(\phi_1)|<C_2 \, \|\phi_0\|_{\mathscr{C}^2} \, \|\phi_1 \|_{\mathscr{C}^2} \, t^{-\eta_2}.
	\]
\end{proposition}
\begin{proof}
	We proceed by contradiction: let us fix $C_2 = 4\max\{C_0, C_1\}$ and $\eta_2 = \min\{1/12, \eta_1/8\}$; then, there exist $\phi_0, \phi_1 \in \mathscr{C}_c^2(M)$ and $T\geq 1$ so that 
	\[
	|\langle \phi_0, \phi_1 \circ K_T \rangle - \mu(\phi_0)\mu(\phi_1)| \geq C_2 \, \|\phi_0\|_{\mathscr{C}^2} \, \|\phi_1 \|_{\mathscr{C}^2} \, T^{-\eta_2}.
	\]
	Up to replacing $\phi_{0}$ with $-\phi_{0}$, we can assume that $\langle \phi_0, \phi_1 \circ K_T \rangle - \mu(\phi_0)\mu(\phi_1)$ is positive. We define $P(t) : = \langle \phi_0, \phi_1 \circ K_t - \mu(\phi_1) \rangle$ and we notice that, since the flow is smooth, the function $P(t)$ is differentiable and $P'(t) = \langle \phi_{0}, U\phi_1 \circ K_t\rangle = -\langle U\phi_{0}, \phi_1 \circ K_t\rangle$, where we denoted by $U$ the derivative in the flow direction. By \Cref{thm:FFK} and our choices of $C_2$ and $\eta_2$, for all $t\in [T,T+T^{1/4}]$, we deduce
	\[
	P(t) \geq P(T) - T^{1/4} \sup_{\theta \in [T,T+T^{1/4}]}|P'(\theta)| \geq \frac{C_2}{2} \, \|\phi_0\|_{\mathscr{C}^2} \, \|\phi_1 \|_{\mathscr{C}^2} T^{-\eta_2}.
	\]
	Therefore, 
	\begin{multline*}
	\frac{C_2}{2} \, \|\phi_0\|_{\mathscr{C}^2} \, \|\phi_1 \|_{\mathscr{C}^2} T^{-\eta_2} \leq \frac{1}{T^{1/4}} \int_{0}^{T^{1/4}} P(T+t)\diff t \\ 
	= \Big\langle \phi_0 \circ K_{-T}, \frac{1}{T^{1/4}} \int_{0}^{T^{1/4}} \phi_1 \circ K_t \diff t - \mu(\phi_1) \Big\rangle \leq \|\phi_{0}\|_2 \cdot \left\| \frac{1}{T^{1/4}} \int_{0}^{T^{1/4}} \phi_1 \circ K_t \diff t- \mu(\phi_1) \right\|_2. 
	\end{multline*}
By \Cref{lem:L2_decay_koch}, the right hand side above is bounded by $C_1 \|\phi_{0}\|_2 \|\phi_1 \|_{\mathscr{C}^1} T^{-\eta_1/4}$. Since $C_1 \leq C_2/4$, we then obtain the inequality $T^{-\eta_2} \leq T^{-\eta_1/4}/2$, which contradicts $T\geq 1$.
\end{proof}

\appendix

\section{Denjoy-Koksma type inequalities and a combinatorial lemma}\label{sec:appendix_a}
Let $x_{min}^N=\min_{0\leq j<N}\|x+j\alpha\|$. Let $\kappa(t)$ be a monotone positive function that goes to $0$ as $t\to \infty$. In what follows $\kappa(\cdot)$ depends on $\alpha$ and $\Phi$ only.
\begin{lemma}\label{lem:ergs} The following estimates on ergodic sums hold:
\begin{enumerate}
\item[ES0] For $N$ sufficiently large, $S_N(x)\geq (1-\kappa(N))N$. Moreover if $\alpha$ satisfies  satisfies $q_{n+1}\leq Cq_n(\log q_n)^2$, then $N(x,t)<t+Ct^\gamma.$
\item[ES1] for every $x\in \T$ and every $N\in [q_n,q_{n+1}]$,
$$
\frac{1}{8} q_n^{2+\gamma}- Cq_n<S_N(\Phi'')(x)<8q_{n+1}^{2+\gamma}+ \Phi''(x_{min}^{q_{n+1}}).
$$
\item[ES2] for every $x\in \T$ and every $n\in \N$,
$$
\frac{1}{8} q_n^{2+\gamma}<|S_{q_n}(\Phi'')(x)-\Phi''(x_{min}^{q_{n}})|<8q_{n}^{2+\gamma}.
$$
\item[ES3] Let $|N|\in [q_n,q_{n+1}]$,. Then for every $x\in \T$,
$$
|S_N(\Phi')(x)|< C\left(q_{n+1}^{1+\gamma}+ \Phi'(x^{q_{n+1}}_{min})\right)
$$
Moreover if $N\leq \varepsilon^{200}q_{n+1}$ and $x\notin \Big(\bigcup_{i=-q_{n}}^{q_{n}-1}[-\frac{\varepsilon^{100}}{q_n},\frac{\varepsilon^{100}}{q_n}]\Big)$, then 
$$
|S_N(\Phi')(x)|< \varepsilon^{-200} Nq_n^\gamma.
$$
\item[ES4] for every $n\in \N$ and every $x\in X$, $|S_{q_n}(\Phi)(x)-q_n|< \Phi(x_{min}^{q_n})+Cq_n^{\gamma}$. Moreover if $\alpha$ satisfies $q_{n+1}\leq Cq_n(\log q_n)^2$, then for $M\in [q_n,q_{n+1}]$.
$$|S_{M}(\Phi)(x)-M|< \Big(\Phi(x_{min}^{M})+CM^\gamma\Big)\log ^2M.$$
\end{enumerate}
\end{lemma}
\begin{proof} Most of these properties are standard and follow from truncating the function close to the singularity and the using the Denjoy-Koksma inequality. ES1, ES2, first part of ES3 and first part of ES4 follow from Lemma 3.1 in \cite{FFK}. 
For the second part of ES3 if $N<q_n$, then we just use the first part of ES3 (since then $N\in [q_k,q_{k+1}]$ with $k<n$). Let $\ell$ be the smallest with $N<\ell q_n$. By the assumption $N\leq \varepsilon^{200}q_{n+1}$ it follows that $\ell\leq 2\varepsilon^{200}\frac{q_{n+1}}{q_n}$. By splitting $\Phi'$ into two monotone branches $\Phi'_1$ and $\Phi'_2$, we can bound $|S_N(\Phi')(x)|\leq |S_{\ell q_n}(\Phi_1')(x)|+|S_{\ell q_n}(\Phi_2')(x)|$. Since $x\notin \Big(\bigcup_{i=-q_{\ell}}^{q_{\ell}-1}[-\frac{\varepsilon^{100}}{q_\ell},\frac{\varepsilon^{100}}{q_\ell}]\Big)$, it follows that for every $\ell'\leq \ell$, $R_\alpha^{\ell'q_n}x\notin \Big(\bigcup_{i=-q_{\ell}}^{q_{\ell}-1}[-\frac{\varepsilon^{100}}{2q_\ell},\frac{\varepsilon^{100}}{2q_\ell}]\Big)$. So using this and splitting the above sums into sums of length $q_n$, we get by Lemma 3.1. in \cite{FFK} that
$$
S_{\ell q_n}(\Phi_i')(x)|\leq \ell q_n^{1+\gamma} +\ell \varepsilon^{-200+\delta} q_n^{1+\gamma}.
$$
Using the fact that $(\ell-1)q_n<N$ the second part of ES3 follows.\\
Second part of ES4 follows from using the Ostrovski expansion, i.e. writing $M=\sum_{j\leq K(M)} b_jq_j$, with $b_j\leq\frac{q_{j+1}}{q_j}$. Then we split the ergodic sums into pieces of length $q_j$ use the first part of ES4 and the diophantine assumption on $\alpha$. 

For the first part of ES0. we use Ostrovski expansion to write $N=\sum b_jq_j$, we then use the fact the $S_{q_n}(x)\geq q_n+\Phi(x_{min}^{q_n})-Cq_n^{\gamma}>q_n-Cq_n^{\gamma}$. The statement follows by splitting the ergodic sums into pieces of length $q_j$. Moreover note that if $\alpha$ satisfies $q_{n+1}\leq Cq_n(\log q_n)^2$, then using the above expansion, in fact $S_N(\Phi)(x)\geq N-CN^{\gamma}$. This implies that $t\geq S_{N(x,t)}(\Phi)(x)\geq N(x,t)-C'N(x,t)^\gamma$, which means that $N(x,t)\leq t+Ct^\gamma$.  This finishes the proof.
\end{proof}

A disjoint collection of sets $\{P_i\}\subset \T$ will be called a partial partition. In the lemma below we assume that all the atoms of the relevant partitions are intervals.
\begin{lemma}\label{lem:par} Let $\cP=\{P_i\}$ and $\cQ=\{Q_j\}$ be finite partial partitions of $\T$ into intervals with $\Leb(\bigcup_i P_i), \Leb(\bigcup_j Q_j)>1-\delta$. Consider  $\cR=\{R_i\}$ to be the common refinement of $\cP$ and $\cQ$ consisting only of atoms $P_i\cap Q_j$ for which $\Leb(P_i\cap Q_j)>\varepsilon_0\min (\Leb(P_i), \Leb(Q_j))$ . Then $\Leb(\bigcup_{i} R_i)>(1-4\sqrt{\delta})(1-3\varepsilon_0)$.
\end{lemma}
\begin{proof} For $\{P_i\}_{i\in I}$ let $I_1$ be the set of indices for which $\Leb(\bigcup_j Q_j\cap P_i)\geq (1-2\sqrt{\delta}) \Leb(P_i)$ for $i\in I_1$.  Notice that
$1-2\delta\leq \Leb(\bigcup_j Q_j\cap \bigcup_iP_i)=\sum_{i\in I_1} \Leb(\bigcup_j Q_j\cap P_i)+\sum_{i\in I\setminus I_1}\Leb(\bigcup_j Q_j\cap P_i)\leq \sum_{i\in I_1}\Leb(P_i)+(1-2\sqrt{\delta}) \sum_{i\in I \setminus I_1}\Leb(P_i)=  2\sqrt{\delta} \sum_{i\in I_1}\Leb(P_i)+1-2\sqrt{\delta}$, which implies that 
$\sum_{i\in I_1}\Leb(P_i)\geq 1-\sqrt{\delta}$.  For a fixed $i\in I_1$, we can divide the $\{Q_j\}$ that intersect $P_i$ into two groups: those $j$ for which $Q_j\subset P_i$  and then trivially $\Leb(P_i\cap Q_j)\geq \Leb(Q_j)$ and those for which $Q_j\cap P_i\neq \emptyset$ but $Q_j$ is not a subset of $P_i$. Since the $\{Q_j\}$ and $\{P_i\}$ are intervals it follows that there are at most $2$ $j's$ which satisfy the latter. Call them $j_1^i=j_1$ and $j_2^i=j_2$. Take $L_i\subset \{ j_1,j_2\}$ be such that  for which $\Leb(P_i\cap Q_{\ell})<\varepsilon_0 \Leb(P_i)$ for $\ell\in L_i$. This set might be $\empty$ or a one point set or a two element set. Then $\Leb(\bigcup_{j\notin L_i} Q_j\cap P_i)\geq (1-2\sqrt{\delta}-2\varepsilon_0)\Leb(P_i)$. Every $Q_j$ with $j\notin L_i$ that is  intersecting $P_i$, in fact satisfies $\Leb(P_i\cap Q_{\ell})\geq \varepsilon_0 \Leb(P_i)$. So if $R_{ij}=P_i\cap Q_j$ with $j\neq L_i$, then $\mu(\bigcup_{i,j}R_{ij})\geq  (1-2\sqrt{\delta}-2\varepsilon_0) (1-2\sqrt{\delta})\geq (1-4\sqrt{\delta})(1-3\varepsilon_0).$ Summing over these $\{P_i\}$ we get the result.
\end{proof}

\section{The Van der Corput Inequality and other technical lemmas}\label{sec:vdc}

We recall a version of the Van der Corput's Inequality that we used in the proof of \Cref{prop:park_implies_park+1}.

\begin{lemma}[Van der Corput Inequality]\label{lem:vdc} 
	Let $N \in \N$, and let $(\phi_n)_{n =1}^N$ be a sequence of vectors in a (real) Hilbert space $H$. Assume that there exists $A >0$ such that $\|\phi_n\|\leq A$ for every $n\in \{1,\dots, N\}$.
Then, for every $1 \leq L \leq K \leq N$ for which $K+L \leq N$, we have 
\begin{equation}\label{eq:vdc}
	\left\|\frac{1}{K}\sum_{n=1}^K \phi_n \right\| \leq  \left[ \frac{2}{K} \sum_{n=1}^K \left(\frac{1}{L}\sum_{l=0}^{L-1} |\langle \phi_n, \phi_{n+l}\rangle| \right) \right]^{1/2}+ 4 A \left(\frac{L}{K}\right)^{1/2}.
\end{equation}
\end{lemma}
\begin{proof}
	Let us start by noticing that
	\[
	\left\|\frac{1}{K}\sum_{n=1}^K \phi_n - \frac{1}{K} \sum_{n=1}^K \phi_{n+l} \right\| \leq  2 A \frac{l}{K},
	\]
	for all $l$ such that $K+l \leq N	$, hence
	\begin{equation}\label{eq:vdc_step1}
		\left\|\frac{1}{K} \sum_{n=1}^K \phi_n  - \frac{1}{K} \sum_{n=1}^K \frac{1}{L} \sum_{l=0}^{L-1}\phi_{n+l}   \right\| \leq 2 A \frac{L}{K}.
	\end{equation}
	We now focus on the second term in the left hand side above.
	
	Since $v \mapsto \|v\|^2$ is a convex function on $H$, Jensen's Inequality yields
	\[
	\begin{split}
		\left\|\frac{1}{K} \sum_{n=1}^K \frac{1}{L} \sum_{l=0}^{L-1}\phi_{n+l} \right\|^2 &\leq  \frac{1}{K} \sum_{n=1}^K \left\| \frac{1}{L} \sum_{l=0}^{L-1}\phi_{n+l}   \right\|^2  = \frac{1}{K} \sum_{n=1}^K  \left( \frac{1}{L^2}\sum_{l_1,l_2=0}^{L-1} |\langle \phi_{n+l_1},\phi_{n+l_2}\rangle| \right).
	\end{split}
	\]
	We separate the sum above into the sets $\{l_1 \leq l_2\}$ and $\{l_2<l_1\}$, we obtain
	\[
	\begin{split}
		& \left\| \frac{1}{K} \sum_{n=1}^K \frac{1}{L} \sum_{l=0}^{L-1}\phi_{n+l}  \right\|^2 \\
		& \ \leq \frac{1}{L} \sum_{l_1=0}^{L-1} \left(  \frac{1}{K} \sum_{n=l_1+1}^{l_1+K}   \frac{1}{L} \sum_{l_2=l_1}^{L-1}   | \langle \phi_{n},\phi_{n+l_2-l_1}\rangle|\right) +  \frac{1}{L} \sum_{l_2=0}^{L-1} \left(  \frac{1}{K} \sum_{n=l_2+1}^{l_2+K}  \frac{1}{L} \sum_{l_1=l_2+1}^{L-1}  |  \langle \phi_{n+l_1-l_2},\phi_{n}\rangle| \right) \\
		& \ \leq \frac{2}{L} \sum_{l_1=0}^{L-1} \left(  \frac{1}{K} \sum_{n=l_1+1}^{l_1+K}   \frac{1}{L} \sum_{l_2=l_1}^{L-1}  | \langle \phi_{n},\phi_{n+l_2-l_1}\rangle| \right)
		\leq \frac{2}{L} \sum_{l_1=0}^{L-1} \left(  \frac{1}{K} \sum_{n=l_1+1}^{l_1+K}   \frac{1}{L} \sum_{l=0}^{L-1}  | \langle \phi_{n},\phi_{n+l}\rangle| \right) \\
		& \ \leq \frac{2}{K} \sum_{n=1}^{K}   \frac{1}{L} \sum_{l=0}^{L-1}  | \langle \phi_{n},\phi_{n+l}\rangle|+ 4A^2 \frac{L}{K},
	\end{split}
	\]
	where in the last line we used the Cauchy-Scwarz Inequality.
	Combining the above inequality with \eqref{eq:vdc_step1}, we conclude
	\[
	\left\|\frac{1}{K} \sum_{n=1}^K \phi_n \right\| \leq 2 A \frac{L}{K} + \left[ \frac{2}{K} \sum_{n=1}^K  \frac{1}{L} \sum_{l=0}^{L-1}   | \langle \phi_{n},\phi_{n+l}\rangle|+ 4A^2 \frac{L}{K}\right]^{1/2},
	\]
	from which \eqref{eq:vdc} follows immediately, since $L/K \leq 1$.
\end{proof}

We need the following lemma. 
\begin{lemma}\label{lem:phi_phi_orthogonal}
Assume that $\flowR$ is polynomially mixing. Let $\phi \in \mathscr{C}^{\infty}_c(X)$ and let $\varepsilon\in (0,1)$. There exists $\phi^{\perp} \in \mathscr{C}^{\infty}_c(X)$, with $\mu(\phi^{\perp}) = 0$ and $ \cN_{k}(\phi^{\perp}) \ll \varepsilon^{-8\sigma(k)/\eta_2} \cN_{k}(\phi)$,
		so that 
		\[
		\| \phi - \mu(\phi) - \phi^{\perp}\|_{L^2}\ll \cN_{d_2}(\phi) \varepsilon.
		\]
\end{lemma}
\begin{proof}
Let $\phi \in \mathscr{C}^{\infty}_c(X)$ and $ \varepsilon \in (0,1)$ be fixed.
Let $R \in \N$, to be chosen later; we define
	\[
	\phi^{\perp} = \phi - \frac{1}{R} \sum_{r=0}^{R-1} \phi \circ \flow_r.
	\]
	Then, $\phi^{\perp}$ has the same smoothness as $\phi$ and satisfies $\|\phi^{\perp}\|_{\infty} \leq 2 \|\phi\|_{\infty}$. 
	If $K$ is a compact set containing the support of $\phi$, then the support of $\phi^{\perp}$ is contained in the finite union $\cup_{r=0}^{R-1} \flow_{-r}(K)$, hence $\phi^{\perp} \in \mathscr{C}^{\infty}_c(X)$.
	
	From polynomial mixing, it is possible to deduce an $L^2$ bound on the ergodic integrals, see the proof of the claim in \Cref{prop:QPar1}; in particular,
	\[
	\| \phi - \mu(\phi) - \phi^{\perp}\|_{L^2}= \left\| \frac{1}{R}\sum_{r=0}^{R-1} \phi \circ \flow_r - \mu(\phi)\right\|_{L^2}\leq \cN_{d_2}(\phi) R^{-\eta_2/8}.
	\]
	We choose $R=\varepsilon^{-8/\eta_2}$. Finally, for any $k\geq 1$, by the properties of the Sobolev norms, we have
	\[
	\cN_k(\phi^{\perp}) \leq 2(1+C_k)\cN_k(\phi) \varepsilon^{-8\sigma(k)/\eta_2},
	\]
	which completes the proof.
\end{proof}

The following lemma says that the intersection of two almost partitions is again an almost partition.

\begin{lemma}\label{lem:intersection_almost_partitions}
Let $\cP = \{P_a : a \in \mathscr{A}\}$ and $\cQ = \{Q_b : b \in \mathscr{B} \}$ be two $\varepsilon$-almost partitions of $A \subseteq Y$. Then, $\cP \wedge \cQ =\{P_a \cap Q_b : a \in \mathscr{A}, \  b \in \mathscr{B}\}$ is a $2\varepsilon$-almost partition of $A$.
\end{lemma}	
	\begin{proof}
Clearly, the sets $P_a \cap Q_b$ are pairwise disjoint.
We claim that
\[
A \triangle \bigcup \{ P_a \cap Q_b : P_a \cap Q_b \cap A \neq \emptyset \} \subseteq \left(A \triangle \bigcup \{ P_a : P_a \cap A \neq \emptyset \} \right) \cup \left(A \triangle \bigcup \{  Q_b : Q_b \cap A \neq \emptyset \}\right).
\]
Indeed, if $x \notin A$ but $x \in \bigcup \{ P_a \cap Q_b : P_a \cap Q_b \cap A \neq \emptyset \}$, then in particular $x \notin A$ and $x \in \bigcup \{ P_a : P_a \cap A \neq \emptyset \}$, which implies that $x \in A \triangle \bigcup \{ P_a : P_a \cap A \neq \emptyset \}$. On the other hand, if $x \in A$ and $x \notin \bigcup \{ P_a \cap Q_b : P_a \cap Q_b \cap A \neq \emptyset \}$, then it follows that $x \notin P_a \cap Q_b$ for all $P_a \cap A \neq \emptyset$ and all $Q_b \cap A \neq \emptyset$. Thus, $x \in A \triangle \big(\bigcup \{ P_a : P_a \cap A \neq \emptyset \}  \cap \bigcup  \{  Q_b : Q_b \cap A \neq \emptyset \}\big) \subseteq \big(A \triangle \bigcup \{ P_a : P_a \cap A \neq \emptyset \} \big) \cup \big(A \triangle \bigcup \{  Q_b : Q_b \cap A \neq \emptyset \}\big)$, which proves the claim.

By the definition of almost partitions, we deduce
\begin{multline*}
 \mu \left(A \triangle \bigcup \{ P_a \cap Q_b : P_a \cap Q_b \cap A \neq \emptyset \}\right) \\ \leq \mu \left(A \triangle \bigcup \{ P_a : P_a \cap A \neq \emptyset \} \right) + \mu \left(A \triangle \bigcup \{  Q_b : Q_b \cap A \neq \emptyset \}\right) \leq 2\varepsilon \mu(A),
\end{multline*}
which proves the result.
	\end{proof}
	

\section{Shearing estimates}\label{sec:shearing}

In this appendix, we are going to prove a general version of \Cref{prop:tc_geo_commute}.

Let $Y,V$ be two smooth vector fields on $M$, and let us denote $(h^Y_t)_{t\in \R}$ and $\flowR$ the flow they generate, respectively. Let us assume that there exists a smooth function $\ell \colon M \to \R$ so that 
\[
\mathscr{L}_V(Y) = [V,Y] = -\ell V.
\]
In the setting of \Cref{sec:tc_unip_flows}, $V=U_\alpha = \alpha^{-1} U$, the vector field $Y$ satisfies $[Y,U] = U$, and $\ell = 1-X\alpha/\alpha$.
We start with a couple of technical lemmas.

\begin{lemma}\label{lemma:formulas_derivatives}
	Let $\phi\in \mathscr{C}^1(M)$. For any $x\in M$, $t \in \R$, and for any $s \in \R$, 
	we have
	\[
	\begin{split}
		&X(\phi \circ \flow_t)(x) = (X\phi)\circ \flow_t(x) + \left(\int_{0}^{t}\ell \circ \flow_r \diff r\right) \cdot (V\phi)\circ \flow_t(x),\\
		&V(\phi \circ h^Y_s)(x) = \exp\left(-\int_0^{s} \ell \circ h^Y_r(x) \diff r\right) (V\phi) \circ h^Y_{s}(x).
	\end{split}
	\]
\end{lemma}
\begin{proof}
	Let $(\flow_t)_\ast$ denote the push-forward by $\flow_t$, and write 
	\begin{equation}\label{eq:push-fwd1}
		(\flow_t)_\ast(X) = a_t X + b_t V,
	\end{equation}
	for some smooth functions $a_t,b_t \colon M \to \R$. By definition, we have
	\[
	X(\phi \circ \flow_t)(x) = [(\flow_t)_\ast(X)]_{\flow_tx}\phi(\flow_tx) = (a_t \circ \flow_t)(x) X\phi(\flow_tx) + (b_t \circ \flow_t)(x) V\phi(\flow_tx),
	\]
	so that, in order to prove the formula, we need to compute $a_t \circ \flow_t$ and $b_t \circ \flow_t$.
	
	Differentiating \eqref{eq:push-fwd1}, we get
	\[
	\begin{split}
		a_t' X + b_t' V &= \frac{\diff}{\diff r} \Big\vert_{r=0} (\flow_r)_\ast(a_t X + b_t V) = \frac{\diff}{\diff r} \Big\vert_{r=0}[(a_t \circ \flow_{-r})(\flow_r)_\ast(X) + (b_t \circ \flow_{-r})(\flow_r)_\ast(V)] \\
		&= -Va_t X - a_t \mathscr{L}_V(X) - Vb_t V,
	\end{split}
	\]
	since $\mathscr{L}_V(V)=0$. Since $(a_t \circ \flow_t)' = a_t' \circ \flow_t + Va_t \circ \flow_t$, we can rewrite the previous equation as
	\[
	(a_t \circ \flow_t)' X + (b_t \circ \flow_t)' V = (a_t \cdot \ell) \circ \flow_t V, \qquad \text{ with } a_0 = 1,\ b_0 = 0.
	\]
	The solutions of this ODE are
	\[
	a_t \circ \flow_t = 1, \qquad \text{and} \qquad b_t \circ \flow_t = \int_0^t \ell \circ \flow_r \diff r,
	\]
	which proves the expression for $X(f \circ \flow_t)$.
	
	For the second part, one follows an analogous argument replacing $X$ with $V$ and $\flow_t$ with $h^Y_s$.
	Denoting $(h^Y_s)_\ast(V) = c_s X + d_s V$, the expression for $V(\phi \circ h^Y_s)(x) $ is obtained by solving the ODE
	\[
	(c_s \circ h^Y_s)' X + (d_s \circ h^Y_s)' V = - (d_s \cdot \ell) \circ h^Y_s V, \qquad \text{ with } c_0 = 0,\ d_0 = 1.
	\]
\end{proof}
For any smooth function $\phi \in \mathscr{C}^1(M)$ and any $x \in M$, let us denote 
\[
I_t\phi(x):=\int_{0}^{t} \phi \circ \flow_r(x) \diff r.
\]
\begin{lemma}\label{lemma:deriv_l_X}
	For any $t,s \in \R$, we have
	\[
	\frac{\partial}{\partial s} (I_t\ell \circ h^Y_s)(x) = [ I_t(X\ell - \ell^2) + I_t\ell \cdot (\ell \circ \flow_t)]\circ h^Y_s(x).
	\]
\end{lemma}
\begin{proof}
	By Lemma \ref{lemma:formulas_derivatives}, we have
	\[
	\begin{split}
		\frac{\partial}{\partial s} (I_t\ell \circ h^Y_s)(x) &= \int_0^t \frac{\partial}{\partial s} \ell \circ \flow_r \circ h^Y_s(x) \diff r = \int_0^t X(\ell \circ \flow_r) \circ h^Y_s(x) \diff r \\
		&=\int_0^t (X\ell) \circ \flow_r \circ h^Y_s(x) \diff r  + \int_0^t I_r\ell \circ h^Y_s(x) \cdot (V\ell) \circ \flow_r \circ h^Y_s(x) \diff r.
	\end{split}
	\]
	Integrating by parts, we can rewrite the second summand as
	\[
	\int_0^t I_r\ell \circ h^Y_s(x) \cdot (V\ell) \circ \flow_r \circ h^Y_s(x) \diff r = I_t\ell \circ h^Y_s(x) \cdot \ell \circ \flow_t \circ h^Y_s(x) - \int_0^t \ell^2 \circ \flow_r \circ h^Y_s(x) \diff r. 
	\]
	Hence we conclude
	\[
	\frac{\partial}{\partial s} (I_t\ell \circ h^Y_s)(x) = I_t(X\ell - \ell^2) \circ h^Y_s(x) + [I_t\ell \cdot (\ell \circ \flow_t)]\circ h^Y_s(x),
	\]
	which proves the lemma.
\end{proof}

We now prove the commutation relation in \Cref{prop:tc_geo_commute}.

\begin{lemma}\label{lem:tc_geo_commutation_relation}
	There exists a smooth function $z \colon M \times \R \times \R \to \R$ such that 
	\[
	\flow_t \circ h^Y_s(x) = h^Y_s \circ \flow_{z(x,s,t)}(x),
	\]
	for all $x \in M$ and $t,s \in \R$ and satisfies
	\[
	\frac{\partial}{\partial s} z(x,s,t) = \exp\left(-\int_0^{-s} \ell \circ h^Y_r(\flow_t\circ h^Y_s(x)) \diff r\right) \cdot I_t \ell(h^Y_s(x)).
	\]
\end{lemma}
\begin{proof}
	By assumption, for every $t\in \R$, the curve
	\[
	s \mapsto h^Y_{-s}\circ \flow_t \circ h^Y_s(x), \qquad \text{ for }s \in \R
	\]
	is well-defined and smooth. In order to prove the lemma, we compute its tangent vectors.
	Using Lemma \ref{lemma:formulas_derivatives}, for any smooth function $\phi$ we have 
	\[
	\begin{split}
		\frac{\diff}{\diff s} &(\phi \circ h^Y_{-s}\circ \flow_t \circ h^Y_s)(x) = -(X\phi)\circ h^Y_{-s}\circ \flow_t \circ h^Y_s(x) + X(\phi \circ h^Y_{-s}\circ \flow_t)(h^Y_s(x)) \\
		&= -(X\phi)\circ h^Y_{-s}\circ \flow_t \circ h^Y_s(x) +  X(\phi \circ h^Y_{-s})\circ \flow_t \circ h^Y_s(x) \\
		& \quad +I_t \ell(h^Y_s(x)) V(\phi \circ h^Y_{-s})\circ \flow_t \circ h^Y_s(x) \\
		&=I_t \ell(h^Y_s(x))  \exp\left(-\int_0^{-s} \ell \circ h^Y_r(\flow_t\circ h^Y_s(x)) \diff r\right) (V\phi) \circ h^Y_{-s}\circ \flow_t \circ h^Y_s(x).
	\end{split}
	\]
	Since the tangent vectors of the curve $s \mapsto h^Y_{-s}\circ \flow_t \circ h^Y_s(x)$ are parallel to $V$, we can write $\flow_{z(x,s,t)} =h^Y_{-s}\circ \flow_t \circ h^Y_s(x)$. The expression for the derivative of $z$ follows immediately from the computation above.
\end{proof}
\begin{corollary}\label{corollary:z}
	For all $t,s \in \R$, there exists $\theta \in [0,s]$ such that
	\[
	z(x,s,t) - t = s \, I_t \ell(x) + \frac{s^2}{2} \left(\frac{\partial^2}{\partial r^2} \Big\vert_{r=\theta}z(x,r,t) \right).
	\]
\end{corollary}
\begin{proof}
	By Taylor's formula, we have
	\[
	z(x,s,t) - z(x,0,t) = s \left( \frac{\partial}{\partial r} \Big\vert_{r=0}z(x,r,t) \right) + \frac{s^2}{2} \cdot \left( \frac{\partial^2}{\partial r^2} \Big\vert_{r=\theta}z(x,r,t) \right),
	\]
	and the claim follows from \Cref{lem:tc_geo_commutation_relation}. 
\end{proof}

We now want to express $z(x,s,t) - t = s I_t\ell(x_0) + o(1)$, where the point $x_0$ is of the form $h^Y_d(x)$ for some small $d \in \R$. 
In order to do that, we need to be able to bound the function $\ell$ and its derivatives, as well as the integrals of $\ell$ and of $X\ell - \ell^2$.

\begin{lemma}\label{lem:estimates_z}
	Let us assume that $\ell$, $X\ell$, and $V\ell$ are uniformly bounded on $M$. Then, there exists $C_\ell\geq 1$ so that for all $t\geq 0$ and for all $s,d \in [0,1]$ we have
		\[
		\left\lvert z(x,s,t) - t \right\rvert \leq C_{\ell} s \, t, \qquad \text{and} \qquad
		\left\lvert(z(x,s,t) - t )  -  s \, I_t \ell(\overline x)\right\rvert \leq C_{\ell} s \, t \, ( d + s + s^2t).
	\]
\end{lemma}
\begin{proof}
	The first claim follows from the Mean-Value Theorem and \Cref{lem:tc_geo_commutation_relation}.
	
	Let us denote $\overline{x} = h^Y_d(x)$. 
	By \Cref{corollary:z}, we have
	\[
	\left\lvert(z(x,s,t) - t ) -  s \, I_t \ell(\overline x)\right\rvert  \leq s \left\lvert I_t \ell( x) - I_t \ell(\overline x)\right\rvert + \frac{s^2}{2} \cdot \sup_{r \in [0,s]}\left\lvert \frac{\partial^2}{\partial r^2} z(x,r,t) \right\rvert,
	\]
	and we now estimate the two summands in the right hand-side above.
	By Lemma \ref{lemma:deriv_l_X}, there exists a constant $C_0 >0$ such that 
	\[
	\left\lvert I_t \ell( x) - I_t \ell(\overline x)\right\rvert \leq C_0 \, t \, d.
	\]
	We now focus on the second one. Using \Cref{lem:tc_geo_commutation_relation}, we can find a constant $C_1>0$ so that 
	\[
	\begin{split}
		\left\lvert \frac{\partial^2}{\partial r^2} z(x,r,t) \right\rvert \leq & \left\lvert \frac{\partial}{\partial r} \exp\left(-\int_0^{-r} \ell \circ h^Y_u(\flow_t\circ h^Y_r(x)) \diff u\right)\right\rvert \cdot |I_t \ell( h^Y_r(x))| \\
		&+ e^{C_1r} \left\lvert \frac{\partial}{\partial r} I_t \ell( h^Y_r(x))\right\rvert.
	\end{split}
	\]
	By \Cref{lemma:formulas_derivatives}, we have
	\begin{multline*}
		\left\lvert \frac{\partial}{\partial r} \exp\left(-\int_0^{-r} \ell \circ h^Y_u(\flow_t\circ h^Y_r(x)) \diff u\right)\right\rvert  \\
		\leq e^{C_1r} \left(|\ell \circ h^Y_{-r} \circ \flow_t \circ h^Y_r(x)| + \left\lvert \int_0^{-r} X(\ell \circ h^Y_u \circ \flow_t)( h^Y_r(x)) \diff u\right\rvert \right)  \leq C_2 \, (1+s \, t),
	\end{multline*}
	for some constant $C_2\geq 1$, since $r\leq s\leq 1$.
	Combining the previous two estimates, we deduce
	\[
		\sup_{r \in [0,s]}\left\lvert \frac{\partial^2}{\partial r^2} z(x,r,t) \right\rvert \leq C_3 (t+ s\, t^2).
	\]
	for some constant $C_3\geq 1$. The proof of the lemma is complete.
\end{proof}

\Cref{prop:tc_geo_commute} in \Cref{sec:tc_unip_flows} follows from \Cref{lem:tc_geo_commutation_relation} and \Cref{lem:estimates_z} above, by choosing $d=0$ and recalling that $\ell = 1-X\alpha / \alpha$.

\end{document}